\documentclass[11pt]{article}
\usepackage[english]{babel}
\usepackage{amsfonts,amsmath,amssymb,amsthm}
\usepackage{amsmath}
\usepackage{tikz,graphicx,subfigure,overpic,verbatim}
\usepackage{hyperref}
\usepackage[fixlanguage]{babelbib}

\def\R{\mathbb{R}}
\def\C{\mathbb{C}}

\def\xicr{\xi_{\mathrm{cr}}}

\def\Tr{\mathop{\mathrm{Tr}}\nolimits}
\def\Re{\mathop{\mathrm{Re}}\nolimits}

\def\supp{\mathop{\mathrm{supp}}\nolimits}
\def\dist{\mathop{\mathrm{dist}}\nolimits}

\numberwithin{equation}{section}

\newtheorem{theorem}{Theorem}[section]

\newtheorem{lemma}[theorem]{Lemma}

\newcommand{\ud}{\,\mathrm{d}}

\begin{document}
\title{A vector equilibrium problem for the two-matrix model in the quartic/quadratic case}
\author{Maurice Duits\footnotemark[1] \and Dries Geudens\footnotemark[2] \and Arno B.J. Kuijlaars\footnotemark[3]}

\renewcommand{\thefootnote}{\fnsymbol{footnote}}
\footnotetext[1]{California Institute of Technology,
Mathematics 253-37, Pasadena, CA 91125, U.S.A. email: mduits\symbol{'100}caltech.edu}
\footnotetext[2]{Department of Mathematics, Katholieke
Universiteit Leuven, Celestijnenlaan 200B, B-3001 Leuven, Belgium.
email: dries.geudens\symbol{'100}wis.kuleuven.be. Research Assistant of the Fund for Scientific
Research - Flanders (Belgium). }
\footnotetext[3]{Department of Mathematics, Katholieke
Universiteit Leuven, Celestijnenlaan 200B, B-3001 Leuven, Belgium.
email: arno.kuijlaars\symbol{'100}wis.kuleuven.be.
Supported in part by FWO-Flanders project G.0455.04, by K.U. Leuven
research grant OT/08/33, by the Belgian Interuniversity Attraction
Pole P06/02,
and by a grant from the Ministry of Education and Science of
Spain, project code MTM2005-08648-C02-01.}

\date{\ }
\maketitle

\begin{abstract}
We consider the two sequences of biorthogonal polynomials $(p_{k,n})_{k=0}^\infty$
and $(q_{k,n})_{k=0}^\infty$ related to the Hermitian two-matrix model
with potentials $V(x) = x^2/2$ and $W(y) = y^4/4 + ty^2$.
From an asymptotic analysis of the coefficients in the recurrence relation
 satisfied by these polynomials, we obtain the limiting distribution
of the zeros of the polynomials $p_{n,n}$ as $n \to \infty$.
The limiting zero distribution is characterized as the first measure
of the minimizer in a vector equilibrium problem involving three measures
which for the case $t=0$ reduces to the vector equilibrium problem
that was given recently by two of us.
A novel feature is that for $t < 0$ an external field is active
on the third measure which introduces a new type of critical behavior
for a certain negative value of $t$.

We also prove a general result about the interlacing of
zeros of biorthogonal polynomials.
\end{abstract}

\section{Introduction}

\subsection{Two-matrix model and biorthogonal polynomials}

The Hermitian two-matrix model is a probability measure
\begin{equation}
\frac{1}{Z_n}e^{-n\Tr(V(M_1)+W(M_2)-\tau M_1M_2)}\ud M_1 \ud M_2
\label{eq: 2mm measure}
\end{equation}
defined on the space of pairs of $n \times n$ Hermitian matrices
$(M_1,M_2)$. Here, $\ud M_1 \ud M_2$ is the standard Lebesgue
measure on pairs of Hermitian matrices, $V$ and $W$ are the two
potentials in the model which are typically polynomials of even degree with a positive leading
coefficient, $\tau \neq 0$ is a coupling
constant and $Z_n$ is the normalizing constant
\[
    Z_n = \iint e^{-n\Tr (V(M_1)+W(M_2)-\tau M_1M_2)} \ud M_1 \ud M_2.
\]

With the two-matrix model \eqref{eq: 2mm measure}, we associate
two sequences of  monic polynomials $(p_{k,n})_{k=0}^\infty$ and
$(q_{k,n})_{k=0}^\infty$, where $\deg p_{k,n} = \deg q_{k,n} = k$,
called biorthogonal polynomials, defined by the property
\begin{equation} \label{eq: biorthogonality general}
\int_{-\infty}^{\infty} \int_{-\infty}^{\infty}
p_{k,n}(x)q_{j,n}(y)e^{-n(V(x)+W(y)-\tau x y)} \ud x \ud y=0,
\qquad k \neq j.
\end{equation}
Existence and uniqueness of these polynomials was proved by
Ercolani and McLaughlin in \cite{EMc}. They also showed that the
zeros of these polynomials are real and simple. It will be one of our results that the
zeros interlace, see Theorem \ref{th: interlacing}.

It is well-known that the eigenvalue correlations of the matrices $M_1$ and $M_2$ from
\eqref{eq: 2mm measure} are determinantal with correlation kernels
that can be expressed in terms of the biorthogonal polynomials and
their transforms, see \cite{BEH1,BEH2,EMc,EyM,MS}.
As an example of these relations we have that
\begin{align*}
\mathbb E \left[ \det(xI_n-M_1) \right] & = p_{n,n}(x),  \\
\mathbb E \left[ \det(yI_n-M_2) \right] & = q_{n,n}(y),
\end{align*}
which show that the diagonal biorthogonal polynomials $p_{n,n}$
and $q_{n,n}$ can be considered as  `typical' characteristic
polynomials for $M_1$ and $M_2$, respectively.  In this sense the
zeros of $p_{n,n}$ are typical eigenvalues of the matrix $M_1$.
This explains our interest in these zeros as $n \to \infty$.

\subsection{A vector equilibrium problem} \label{subsec: vector equilibrium}

In \cite{DK2} two of us studied the limiting eigenvalue behavior
of the matrix $M_1$ in the two-matrix model \eqref{eq: 2mm
measure} for the case of an even polynomial $V$ with  positive
leading coefficient and
\[ W(y) = \tfrac{1}{4} y^4. \]
The biorthogonal polynomial $p_{n,n}$ associated with this model
can be characterized by a Riemann-Hilbert problem of size $4
\times 4$ \cite{BEH2,KMc}. The Deift-Zhou steepest descent method
was successfully applied to the Riemann-Hilbert problem from
\cite{KMc}. A crucial ingredient in \cite{DK2} is the introduction
of a vector equilibrium problem with external field and upper
constraint that describes the limiting mean eigenvalue distribution.

To state the vector equilibrium problem we use notions
from logarithmic potential theory \cite{SaTo}. We define
the logarithmic energy of a finite positive measure $\nu$ on
$\mathbb C$ as
\[
I(\nu)=\iint \log \frac{1}{|x-y|} \ud \nu (x) \ud \nu (y),
\]
and the logarithmic potential as
\begin{equation}
U^\nu(z)=\int \log \frac{1}{|z-x|} \ud \nu (x).
\label{eq: logarithmic potential}
\end{equation}
If $\nu_1$ and $\nu_2$ are positive measures on $\mathbb C$ with
$I(\nu_1),I(\nu_2)< \infty$, we also define their mutual
logarithmic energy as
\[
I(\nu_1,\nu_2)=\iint \log \frac{1}{|x-y|} \ud \nu_1 (x) \ud \nu_2 (y).
\]

The vector equilibrium problem from \cite{DK2} asks to minimize
the energy functional
\[
    I(\rho_1)-I(\rho_1,\rho_2)+I(\rho_2)-I(\rho_2,\rho_3) +I(\rho_3)+
    \int \left(V(x)-\frac{3}{4}|\tau x |^{4/3} \right) \ud \rho_1(x),
\]
among all measures $\rho_1$, $\rho_2$ and $\rho_3$ with finite logarithmic
energy that satisfy
\begin{itemize}
\item[(a)] $\rho_1$ is supported on $\R$ and $\rho_1(\R)=1$;
\item[(b)] $\rho_2$ is supported on $i\R$ and $\rho_2(i\R)=2/3$;
\item[(c)] $\rho_3$ is supported on $\R$ and $\rho_3(\R)=1/3$;
\item[(d)] $\rho_2$ satisfies the constraint $\rho_2 \leq \sigma$
    where $\sigma$ is the unbounded measure on $i\R$ defined as
\[
\mathrm d \sigma(z)= \frac{\sqrt 3}{2\pi i}\tau^{4/3}|z|^{1/3} \ud z, \qquad z \in i\R.
\]
Here, $\mathrm d z$ is the complex line element on $i\R$.
\end{itemize}

In \cite{DK2} it is shown that there is a unique minimizer
$(\nu_1, \nu_2, \nu_3)$ of the vector equilibrium problem. The
measure $\nu_1$ is supported on a finite union of disjoint
intervals. For the case of one interval (one-cut case) it was
shown in \cite{DK2} that the density of $\nu_1$ is the limiting
mean density of the eigenvalues of $M_1$. See \cite{Mo} for the
extension to the multi-cut case.

\subsection{Aim of this paper}
It is the aim of this paper to give a new perspective on the nature
of the above vector equilibrium problem. In \cite{DK2} the vector
equilibrium problem was simply posed out of the blue, while in the
present work it arises after certain calculations.

In addition, we derive a similar vector equilibrium problem for the
case that
\begin{equation}
    V(x)=\frac{x^2}{2} \qquad \textrm{and} \qquad W(y)=\frac{y^4}{4}+t\frac{y^2}{2},
\label{eq:quartic/quadratic potentials}
\end{equation}
where $t \in \mathbb R$ is a real parameter. When $t < 0$
this vector equilibrium problem has the novel feature that an
external field is acting on the third measure as well.

Due to the fact that $V(x)$ is quadratic in \eqref{eq:quartic/quadratic potentials}
 the second sequence $(q_{k,n})_k$ of biorthogonal polynomials is actually a sequence
of orthogonal polynomials that satisfy a three term
recurrence relation. The biorthogonal polynomials $p_{k,n}$ also
satisfy a recurrence relation with recurrence coefficients that
can be expressed in terms of the recurrence coefficients for $q_{k,n}$.
We are going to analyze these recurrence coefficients as $n \to \infty$
and obtain the vector equilibrium problem from this analysis.
This approach is similar in spirit to the
analysis of \cite{KR} for the recurrence coefficients of multiple
orthogonal polynomials in a model of non-intersecting squared
Bessel paths.

Furthermore, we prove that the first component $\nu_1$ of the
vector of measures $(\nu_1,\nu_2,\nu_3)$ minimizing this vector
equilibrium problem is equal to the limiting zero distribution
of the diagonal polynomials $p_{n,n}$ as $n \to \infty$.
The measure $\nu_1$ is also the limiting mean eigenvalue distribution
of the matrices $M_1$ in the two-matrix model.
This will be proved in a forthcoming paper.

In the next section we give a more detailed description of the
main results in this paper.

\section{Statement of main results}

\subsection{Interlacing zeros}

Our first result deals with biorthogonal polynomials defined by \eqref{eq: biorthogonality general}
with general potentials $V$ and $W$. Our result is that the zeros of consecutive biorthogonal
$p_{k,n}$ and $p_{k+1,n}$ are interlacing. This was proved
by Woerdeman \cite{Woe} for a special case.
Two ordered sequences of real numbers
$\alpha_1, \ldots, \alpha_k$ and $\beta_1, \ldots, \beta_{k+1}$
are said to interlace if
\[ \beta_1 < \alpha_1 < \beta_2 < \alpha_2  < \cdots < \alpha_{k-1} <
 \beta_{k} < \alpha_k < \beta_{k+1}.
\]

\begin{theorem} \label{th: interlacing}
Take $\tau \neq 0$ and suppose that $V$ and $W$ are functions for which the integrals
\[
\int_{-\infty}^{\infty}\int_{-\infty}^{\infty}
x^ky^je^{-n(V(x)+W(y)-\tau x y)} \ud x \ud y, \quad
k=0,1,2,\ldots, \quad j=0,1,2,\ldots,
\]
converge. Define the two sequences of monic biorthogonal
polynomials $(p_{k,n})_{k=0}^\infty$ and $(q_{k,n})_{k=0}^\infty$
as in \eqref{eq: biorthogonality general}. Then the following statements hold for every $k=1,2, \ldots$
\begin{itemize}
\item[\rm(a)] The zeros of $p_{k,n}$ and $p_{k+1,n}$ interlace, and
similarly, the zeros of $q_{k,n}$ and $q_{k+1,n}$ interlace.
\item[\rm(b)] If the potentials $V$ and $W$ are even, then the positive zeros
of $p_{k,n}$ and $p_{k+2,n}$ interlace, and similarly, the positive zeros of $q_{k,n}$ and $q_{k+2,n}$ interlace.
\end{itemize}
\end{theorem}

Theorem \ref{th: interlacing} is proved in Section \ref{sec: interlacing}.

\subsection{Limit of zero counting measures}
In the rest of the paper we restrict ourselves to the quadratic
and quartic potentials \eqref{eq:quartic/quadratic potentials}.
The biorthogonal polynomials associated with this model are thus
defined by
\begin{equation}
\int_{-\infty}^{\infty} \int_{-\infty}^{\infty}
p_{k,n}(x)q_{j,n}(y)e^{-n(x^2/2+y^4/4+ty^2/2-\tau x y)} \ud x \ud
y=0, \qquad k \neq j. \label{eq: biorthogonality specific}
\end{equation}
We also assume without loss of generality that
\[ \tau > 0. \]

As a second result, we will show that in this case the limiting
zero distribution of the diagonal polynomials $p_{n,n}$ exists.
For a polynomial $P$ of degree $n$, we introduce the normalized
zero counting measure
\[
\nu(P)=\frac1n \sum_{P(x)=0}\delta_x,
\]
where the sum is taken over all zeros of $P$ counted with
multiplicity. We say that a sequence of measures $\nu_n$ converges
weakly to the measure $\nu$ if
\[
\int f \ud \nu_n \to \int f \ud \nu,
\]
for every bounded continuous function $f$.

\begin{theorem} \label{th: zero distribution}
There exists a Borel probability measure $\nu_1$ with $\supp
(\nu_1) \subset \R$ such that $\nu_1$ is the limiting zero
distribution of the diagonal polynomials $p_{n,n}$, i.e.
\[
\lim_{n \to \infty} \nu(p_{n,n})=\nu_1,
\]
where the limit is in the sense of weak convergence of measures.
\end{theorem}

The proof of Theorem \ref{th: zero distribution} is given in Section \ref{sec: zero distribution}. It is constructive. For an explicit formula of the limiting zero distribution $\nu_1$, we refer to \eqref{eq: nu1}.

Note that Theorem \ref{th: zero distribution} is about the
limiting distribution of the zeros of the biorthogonal polynomial
$p_{n,n}$. The measure  $\nu_1$ is also the limiting mean
distribution of the eigenvalues of the matrix $M_1$ from the
two-matrix model, but this is not proved in this paper. This will
follow from an analysis of the Riemann-Hilbert problem as in
\cite{DK2}, which is under current investigation.

\subsection{Vector equilibrium problem} \label{subsection: vector equilibrium problem}

Our final result is that the limiting zero distribution $\nu_1$
can be characterized by a vector equilibrium problem depending
on external fields $V_1$ and $V_3$, and
on a constraint $\sigma$. These objects will be described next in
terms of the solutions of the equation
\begin{equation}
    \omega^3 + t \omega = \tau z.
\label{eq: omegas}
\end{equation}
For $t=0$, the vector equilibrium problem reduces to the vector equilibrium
problem described in Section \ref{subsec: vector equilibrium}.

\paragraph{External field $V_1$ on $\mathbb R$.}
For  $z = x \in \R$, the equation \eqref{eq: omegas} has either one
or three real solutions. We use $\omega_1(x)$ to denote the
real solution with the largest absolute value.
This is the only real solution for $t \geq 0$
and also for $t < 0$ and $z=x$ with $|x| > x^*$ where
\begin{equation}
x^* = x^*(t) = \frac{2 (-t)^{3/2}}{3\sqrt 3 \, \tau}, \qquad t \leq 0.
\label{eq: x star}
\end{equation}
For $t < 0$ and $ - x^* \leq x \leq x^*$ there are three real solutions of \eqref{eq: omegas}
which we denote by $\omega_j(x)$, $j=1,2,3$, and which we  number such that
\begin{equation}
    |\omega_1(x)| \geq |\omega_2(x)| \geq |\omega_3(x)|.
\label{eq: omega123ordering}
\end{equation}
This means in fact that $\omega_2(x) < \omega_3(x) < 0 < \omega_1(x)$ if $x \in (0, x^*)$ and
$\omega_1(x) < 0 < \omega_3(x) < \omega_2(x)$ if $x \in (-x^*, 0)$.

In both cases, the external field $V_1$ is defined by
    \begin{align} \nonumber
     V_1(x) & = \frac{x^2}{2} + \min_{y \in \mathbb R} (W(y) - \tau xy) \\
        & = \frac{x^2}{2}-\frac{3}{4} \omega_1(x)^4 - \frac{1}{2} t \omega_1(x)^2 , \qquad x \in \R.
    \label{eq: V1}
        \end{align}
The second identity in \eqref{eq: V1} comes from the fact that the minimum of $W(y) - \tau xy$
is taken at $y = \omega_1(x)$ and $\tau x= \omega_1(x)^3 + t\omega_1(x)$ by \eqref{eq: omegas}.

\paragraph{External field $V_3$ on $\mathbb R$}
The external field $V_3$ vanishes identically for $t \geq 0$
\begin{equation}
     V_3(x) \equiv 0, \qquad \text{for } x \in \mathbb R,  \text{ if } t \geq 0.
\label{eq: V3 0}
\end{equation}
For $t < 0$ and $x \in \mathbb R$, we define
\begin{equation}
V_3(x)   = \begin{cases} \frac{3}{4}  \omega_2(x)^4 + \frac{1}{2} t  \omega_2(x)^2-\frac{3}{4}  \omega_3(x)^4 -
    \frac{1}{2} t  \omega_3(x)^2,  &\text{for } |x| < x^*, \\
0, & \textrm{for }|x| \geq x^*.
\end{cases}
\label{eq: V3}
\end{equation}
where $x^*$ is given in \eqref{eq: x star}, and $\omega_2(x)$ and $\omega_3(x)$
are the solutions of \eqref{eq: omegas} that satisfy \eqref{eq: omega123ordering}.

While $V_1(x)$ is related to the global minimum of the function
\begin{equation}
     W(y)-\tau x y = \frac{y^4}{4}+t\frac{y^2}{2}-\tau x y, \qquad y \in \mathbb R,
\label{eq: minimumWy}
\end{equation}
$V_3(x)$ can be interpreted as the positive difference
between the local maximum  and the other local minimum of \eqref{eq: minimumWy} on $\mathbb R$,
which indeed exist if and only if $t < 0$ and $|x| < x^*$.

\paragraph{Upper constraint $\sigma$ on $i \R$}
For $z = iy \in i\R$, the equation \eqref{eq: omegas} has either one or
three purely imaginary solutions. There are three purely imaginary
solutions if and only if $t \geq 0$ and $|y| \leq y^*$ where
\begin{equation} \label{eq: ystar}
    y^* = y^*(t) = \frac{2t^{3/2}}{3\sqrt 3 \tau}, \qquad t \geq 0.
    \end{equation}
Otherwise there is only one purely imaginary solution and the two other solutions
are located symmetrically with respect to the imaginary axis.
We then let $\omega_1(z)$ be the solution of \eqref{eq: omegas} with
positive real part. For convenience we put
\[ y^*=0, \qquad \text{if } t < 0. \]

The upper constraint $\sigma$ is defined as follows.
The support of  $\sigma$  is
\begin{equation} \label{eq: supportsigma}
\supp(\sigma)  =   (-i\infty,-i y^*] \cup  [iy^*, i \infty),
\end{equation}
and $\sigma$ has the density on $\supp(\sigma)$ given by
\begin{equation}
    \frac{\mathrm d \sigma(z)}{|\mathrm d z|} =
    \frac{\tau}{\pi} \Re \omega_1(z), \qquad z \in \supp (\sigma),
\label{eq: sigma}
\end{equation}
for every fixed $t \in \R$.

Now we  state our final main result.

\begin{theorem} \label{th: equilibrium problem}
The measure $\nu_1$ from Theorem {\rm \ref{th: zero distribution}}
is the first component of the unique vector of measures
$(\nu_1,\nu_2,\nu_3)$ minimizing the energy functional
\begin{multline*}
E(\rho_1,\rho_2,\rho_3) = I(\rho_1)-I(\rho_1,\rho_2) +
I(\rho_2)-I(\rho_2,\rho_3)+ I(\rho_3) \\ + \int V_1(x) \ud \rho_1(x)
+ \int V_3(x) \ud \rho_3(x),
\end{multline*}
among all vectors  $(\rho_1, \rho_2, \rho_3)$ of measures
with finite logarithmic energy satisfying
\begin{itemize}
 \item[\rm (a)] $\rho_1$ is supported on $\R$ and $\rho_1(\R)=1$,
 \item[\rm (b)] $\rho_2$ is supported on $i\R$ and $\rho_2(i\R)=2/3$,
 \item[\rm (c)] $\rho_3$ is supported on $\R$ and $\rho_3(\R)=1/3$,
 \item[\rm (d)] $\rho_2$ satisfies the constraint $\rho_2 \leq \sigma$.
\end{itemize}
Here $V_1$ and $V_3$ are defined in \eqref{eq: V1} and \eqref{eq: V3 0}--\eqref{eq: V3}, respectively,
and $\sigma$ is defined by \eqref{eq: supportsigma}--\eqref{eq: sigma}.
\end{theorem}

Theorem \ref{th: equilibrium problem} is proved in Section
\ref{sec: integrating}.

\subsection{Phase diagram and critical behavior}
The proof of Theorem \ref{th: equilibrium problem} is
constructive, and we find fairly explicit formulas for the
minimizing measures $\nu_j$. Indeed, we obtain $\nu_j$ as an average
\[ \nu_j = \int_0^1 \mu_j^{\xi} d\xi \]
of measures $\mu_j^{\xi}$ depending on a parameter $\xi$ and these measures
are given by formulas \eqref{eq: muj 1cc} and \eqref{eq: muj 2cc bis}.

It follows from the analysis leading to these formulas that the supports of
the measures $\nu_1$, $\sigma-\nu_2$, and $\nu_3$ have the following form
\begin{align*}
    \supp (\nu_1) & =[-\alpha,-\beta] \cup [\beta, \alpha], \\
    \supp(\sigma-\nu_2) & =i\R \setminus (-i\gamma,i\gamma), \\
    \supp(\nu_3) &= \R \setminus (-\delta,\delta),
\end{align*}
for some $\alpha > \beta \geq 0$, $\gamma, \delta \geq 0$ depending on $t \in \mathbb R$ and $\tau > 0$.

We may distinguish a number of cases, depending on whether $\beta$, $\gamma$, or $\delta$ are equal to
zero, or not. At least one of these is zero, and generically, no two consecutive ones are zero.
Our analysis leads to the phase diagram in the $t \tau$-plane shown in
Figure \ref{fig: tau t phase diagram}.
\begin{description}
\item[Case I:] $\beta=0$, $\gamma>0$, and $\delta=0$.
Thus in this case there are no gaps in the supports of the measures $\nu_1$ and $\nu_3$ on
the real line. The constraint $\sigma$ is active along an interval $[-i\gamma, i\gamma]$
on the imaginary axis.
\item[Case II:] $\beta=0$, $\gamma>0$, and $\delta>0$.
In this case the measure $\nu_1$ is still supported on one interval. However there
is a gap $(-\delta, \delta)$ in the support of $\nu_3$. As in Case I, the constraint
$\sigma$ is  active along an interval $[-i\gamma, i\gamma]$ on the imaginary axis.
\item[Case III:] $\beta>0$, $\gamma = 0$, and $ \delta>0$.
In Case III there is a gap in the supports of $\nu_1$ and $\nu_3$, but the constraint
on the imaginary axis is not active.
\item[Case IV:] $\beta>0$, $\gamma > 0$, and $\delta=0$.
In Case IV there is a gap in the support of $\nu_1$,  but there is no gap in the support of $\nu_3$,
which is now the full real line. The constraint is active
along an interval along the imaginary axis.
\end{description}

Cases II and III are new in the sense that they do
not appear in \cite{DK2}. The opening of a gap in the support of $\nu_3$
is due to the external field $V_3$
that acts on the third measure in the vector equilibrium problem.
As $V_3$ is identically zero for $t \geq 0$, the Cases II and III do
not appear if $t \geq 0$, as can be seen in Figure~\ref{fig: tau t phase diagram}.

Critical behavior occurs at the curves that separate the different cases
from each other. These critical curves are given by the equations
\begin{align*}
    \tau = \sqrt{t + 2}, \qquad -2 \leq t < \infty, \qquad \text{and} \qquad
    \tau = \sqrt{- \frac{1}{t}}, \qquad -\infty < t < 0. \end{align*}
On the critical curves two of the numbers $\beta, \gamma$ and $\delta$
are equal to zero. For example, on the curve between  Case II and Case III, we
have $\beta=\gamma=0$, while $\delta>0$.
Finally, note the multi-critical point
\[ t=-1, \qquad \tau=1 \]
in the phase diagram, where $\beta=\gamma=\delta = 0$.
All four cases come together at this point in the $t \tau$-plane.

We do not discuss the critical and multi-critical behavior any further in this paper.
However, it would be particularly interesting to analyze the nature of the
multi-critical point.

\begin{figure}[t]
\begin{center}
\begin{tikzpicture}[scale=1.5]
\draw[->](0,0)--(0,4.25) node[above]{$\tau$};
\draw[->](-4,0)--(4.25,0) node[right]{$t$};
\draw[help lines] (-1,0)--(-1,1)--(0,1);
\draw[very thick,rotate around={-90:(-2,0)}] (-2,0) parabola (-4.5,6.25) node[above]{$\tau=\sqrt{t+2}$};
\draw[very thick] (-1,1)..controls (0,1.5) and (-0.2,3).. (-0.1,4)
             (-1,1)..controls (-2,0.5) and (-3,0.2).. (-4,0.1) node[above]{$\tau=\sqrt{-\frac{1}{t}}$};
\filldraw  (-1,1) circle (1pt);
\draw (0.1,1) node[font=\footnotesize,right]{$1$}--(-0.1,1);
\draw (-1,0.1)--(-1,-0.1) node[font=\footnotesize,below]{$-1$};
\draw (-2,0.1)--(-2,-0.1) node[font=\footnotesize,below]{$-2$};
\draw (0.1,1.43) node [font=\footnotesize,right]{$\sqrt 2$}--(-0.1,1.43);
\draw[very thick] (2,0.8) node[fill=white]{Case I}
                  (-2.5,0.2) node{Case II}
                  (-2,2) node[fill=white]{Case III}
                  (1,3) node[fill=white]{Case IV};
\end{tikzpicture}
\end{center}
\caption{The phase diagram in the $t\tau$-plane: the critical curves $\tau=\sqrt{t+2}$
and $\tau=\sqrt{-\frac{1}{t}}$ separate the four cases.
The cases are distinghuished by the fact whether $0$ is in the support of the measures
$\nu_1$, $\sigma-\nu_2$, and $\nu_3$, or not.}
\label{fig: tau t phase diagram}
\end{figure}
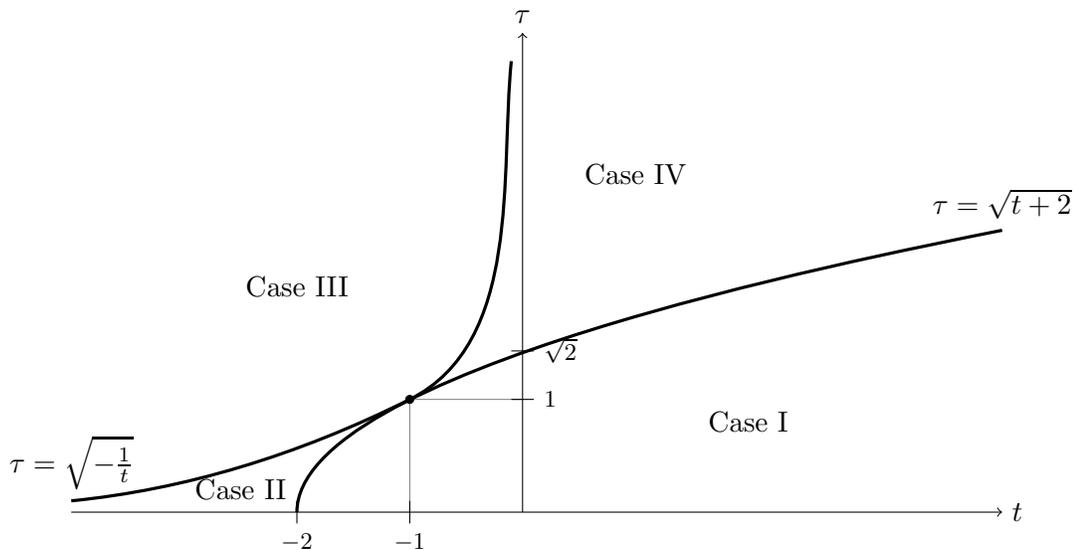

\subsection{Overview of the rest of the paper}
The rest of the paper is devoted to the proofs of the three main results.
Theorem~\ref{th: interlacing} is proved in Section \ref{sec: interlacing}.
The other sections deal with the specific model \eqref{eq:quartic/quadratic potentials}
and lead to the proofs of Theorems \ref{th: zero distribution} 
and \ref{th: equilibrium problem} in Sections \ref{sec: zero distribution} and
\ref{sec: integrating}.
The intermediate sections contain auxiliary results that will be essential to these proofs.

In Section \ref{sec: recurrence} it is shown that the biorthogonal polynomials satisfy
recurrence relations with recurrence coefficients that have certain asymptotic behaviors.
We introduce a new parameter $\xi$ and consider the asymptotic behavior in the $t\xi$-phase space.
Two different types of asymptotic behavior will lead to a separation of the phase space into two
regions $C_1$ and $C_2$, where $C_1$ is the one-cut case region, and $C_2$ the two-cut case region.

With every $\xi$-value, we associate a vector equilibrium problem for three measures.
These equilibrium problems serve as building blocks for the vector equilibrium problem
of Theorem~\ref{th: equilibrium problem}.
We introduce and analyze the one-cut case in Section \ref{sec: 1cc} and
the two-cut case in Section \ref{sec: 2cc}.
Due to the different asymptotic behavior of recurrence coefficients, the analysis
in both sections is significantly different.

\section{Proof of Theorem \ref{th: interlacing}} \label{sec: interlacing}

In this section we prove Theorem \ref{th: interlacing}. The proof
of this theorem is inspired by \cite{AKLR} and uses the following
theorem from \cite{EMc}.

\begin{theorem} \label{th: ercMcL}
Suppose that $\tau \neq 0$. Let $V$ and $W$ be functions for which the integrals
\[
\int_{-\infty}^{\infty}\int_{-\infty}^{\infty}
x^ky^je^{-n(V(x)+W(y)-\tau x y)} \ud x \ud y,\quad k=0,1,2,\ldots,
\quad j=0,1,2,\ldots,
\]
converge. Then, for each $k$, there is a unique monic
polynomial $p_{k,n}$ and a unique monic polynomial $q_{k,n}$ such
that the families of polynomials $(p_{k,n})_{k=0}^\infty$ and
$(q_{k,n})_{k=0}^\infty$ satisfy (\ref{eq: biorthogonality
general}). In addition, the zeros of these polynomials are real and simple.
\end{theorem}

\begin{proof}
This is \cite[Theorem 1]{EMc}.
\end{proof}

To establish Theorem \ref{th: interlacing} (a) it is clearly
enough to prove the statements about the zeros of $p_{k,n}$, since
the results about the zeros of $q_{k,n}$ follow by symmetry.

\begin{proof}[Proof of Theorem \ref{th: interlacing} (a).]
Fix an integer $k \geq 1$ and consider the linear combination
$Ap_{k,n}+Bp_{k+1,n}$ with $(A,B) \neq (0,0)$. We claim that this
polynomial has no real multiple zeros.

To see this, assume that $x_0$ is a real zero of multiplicity at least two.
Then we can write
\[
    Ap_{k,n}(x)+Bp_{k+1,n}(x)=(x-x_0)^2 r(x),
\]
where $r$ is polynomial of degree $\leq k-1$. From the biorthogonality
\eqref{eq: biorthogonality general} it follows that
\begin{multline*}
\int_{-\infty}^{\infty} \int_{-\infty}^{\infty}
r(x)y^l(x-x_0)^2e^{-n(V(x)+W(y)-\tau x y)} \ud x \ud y \\ =
\int_{-\infty}^{\infty} \int_{-\infty}^{\infty}
r(x)y^le^{-n(V(x)-\frac{2}{n}\log |x-x_0|+W(y)-\tau x y)} \ud x \ud
y=0,
\end{multline*}
for $l=0,1,2,\ldots,k-1$.
Applying the uniqueness part of Theorem \ref{th: ercMcL} to  the modified
potential $V(x)-\frac2n \log|x-x_0|$, we obtain $r \equiv 0$.
Thus, $A=B=0$, which yields a contradiction and, therefore, proves the claim
that $A p_{k,n} + B p_{k+1,n}$ has no real multiple zeros if $(A,B) \neq (0,0)$.
It follows that the linear system of equations
\[
\begin{pmatrix}
p_{k,n}(x) & p_{k+1,n}(x) \\
p'_{k,n}(x) & p'_{k+1,n}(x)
\end{pmatrix}
\begin{pmatrix} A \\ B \end{pmatrix} =
\begin{pmatrix} 0 \\ 0 \end{pmatrix},
\]
has only the trivial solution $A= B = 0$, for every $x \in \mathbb R$. Therefore
the matrix has a non-zero determinant, and thus
\[
    p_{k,n}(x)p'_{k+1,n}(x)-p'_{k,n}(x)p_{k+1,n}(x) \neq 0, \qquad x \in \mathbb R.
\]
By continuity and the behavior as $x \to \infty$, we conclude
from this that
\begin{equation} \label{eq: >0}
    p_{k,n}(x)p'_{k+1,n}(x)-p'_{k,n}(x)p_{k+1,n}(x) > 0, \qquad x \in \mathbb R.
\end{equation}

Now consider two consecutive zeros $x_l$ and $x_{l+1}$ of $p_{k+1,n}$.
Because these zeros are simple, we have that
\[
    p'_{k+1,n}(x_l)p'_{k+1,n}(x_{l+1})<0.
\]
From \eqref{eq: >0} we find  $p_{k,n}(x_l)p'_{k+1,n}(x_l) > 0$ and
$p_{k,n}(x_{l+1})p'_{k+1,n}(x_{l+1}) > 0$. Hence, we obtain
\[
    p_{k,n}(x_l) p_{k,n}(x_{l+1}) < 0.
\]
Therefore, $p_{k,n}$ must have a zero between $x_l$ and $x_{l+1}$.
Hence, in between any two consecutive zeros of $p_{k+1,n}$, there is
a zero of $p_{k,n}$, which implies that the zeros of $p_{k,n}$ and $p_{k+1,n}$ interlace.
\end{proof}

\begin{proof}[Proof of Theorem \ref{th: interlacing} (b).]
The proof of (b) follows the same strategy. Let $k \geq 1$
and assume that the linear combination $Ap_{k,n} + Bp_{k+2,n}$ has a positive multiple root, say $x_0 > 0$.
Because $V$ and $W$ are even potentials, $p_{k,n}$ and $p_{k+2,n}$ are either both even or both odd.
Therefore, also $-x_0$ is a double zero and we can write
\[
    Ap_{k,n}(x) + Bp_{k+2,n}(x)= (x^2-x_0^2)^2 r(x),
\]
where $r$ is polynomial of degree $\leq k-2$. Then, $r \equiv 0$ and, thus, $A = B=0$ as in the proof of part (a).

Hence $Ap_{k,n} +  Bp_{k+2,n}$ with $(A,B) \neq (0,0)$ has no positive multiple zeros.
Therefore, the linear system of equations
\[
\begin{pmatrix}
p_{k,n}(x) & p_{k+2,n}(x) \\
p'_{k,n}(x) & p'_{k+2,n}(x)
\end{pmatrix}
\begin{pmatrix} A \\ B \end{pmatrix} =
\begin{pmatrix} 0 \\ 0 \end{pmatrix},
\]
has only the trivial solution for every $x >0$. Thus, as in the proof
of part (a),
\begin{equation} \label{eq: pkandpk+2}
p_{k,n}(x) p'_{k+2,n}(x)-p'_{k,n}(x) p_{k+2,n}(x) > 0, \qquad x >0.
\end{equation}
The proof of interlacing of the positive zeros follows from
\eqref{eq: pkandpk+2} in the same way as before.
\end{proof}

\paragraph{Remark.}
It was shown in \cite{EMc} that the zeros of biorthogonal
polynomials are real and simple. We can prove this result in an alternative way as follows.

First assume that $p_{k,n}$ has a non-real zero $x_0 = a+bi$, $a,b \in \R$, $b \neq 0$.
Then also $\overline{x}_0 = a-bi$ is a zero of $p_{k,n}$. Thus, $p_{k,n}$ can be written as
\[
    p_{k,n}(x)=\left( (x-a)^2+b^2 \right) r(x),
\]
where $r$ is a polynomial of degree $k-2$. Now observe that
\begin{multline*}
\int_{-\infty}^{\infty} \int_{-\infty}^{\infty}
y^l\left( (x-a)^2+b^2 \right) r(x)e^{-n(V(x)+W(y)-\tau x y)} \ud x \ud y \\ =
\int_{-\infty}^{\infty} \int_{-\infty}^{\infty}
r(x)y^le^{-n \left( V(x)-\frac{1}{n}\log \left( (x-a)^2+b^2 \right)+W(y)-\tau x y\right)} \ud x \ud
y=0,
\end{multline*}
for $l=0,1,2,\ldots,k-1$. Note that the second equality follows
from \eqref{eq: biorthogonality general}.
Applying the uniqueness part of Theorem \ref{th: ercMcL} to $W(y)$ and
the modified potential $V(x)-\frac{1}{n}\log \left( (x-a)^2+b^2 \right)$
we obtain $r \equiv 0$ which is a contradiction. Therefore the zeros of $p_{k,n}$ are real.
Moreover, by putting $b=0$ in the above argument we also obtain that the zeros are simple.

Of course, a similar reasoning shows that the zeros of $q_{k,n}$ are real and simple.

\section{Preliminaries on recurrence coefficients} \label{sec: recurrence}

\subsection{Relations between recurrence coefficients}

In the rest of the paper we consider the model with quartic
and quadratic potentials \eqref{eq:quartic/quadratic potentials}.
In that case both
sequences $(p_{k,n})_k$ and $(q_{k,n})_k$ of biorthogonal
polynomials defined by \eqref{eq: biorthogonality specific}
satisfy a recurrence relation with recurrence coefficients
that are related to each other as
described in the following lemma.

\begin{lemma}\label{lemma:relation coefficients}
For $n \geq 1$, the sequence of polynomials $(q_{k,n})_{k=0}^\infty$
is orthogonal with respect to the weight
\begin{equation}
    w(y)=e^{ -n(y^4/4+ty^2/2-\tau^2y^2/2)}, \qquad y \in \mathbb R,
\label{eq: weight}
\end{equation}
and, therefore, satisfies a recurrence relation of the form
\[
    yq_{k,n}(y)=q_{k+1,n}(y) + a_{k,n}q_{k-1,n}(y), \quad k=0,1,2,\ldots ,
\]
where $q_{-1,n} \equiv 0,$ $a_{0,n}=0$, and $a_{k,n}>0$, $k=1,2,\ldots$.

In addition, the sequence of polynomials $(p_{k,n})_{k=0}^\infty$ satisfies a
recurrence relation of the form
\begin{equation} \label{eq: 5trr}
    xp_{k,n}(x)=p_{k+1,n}(x)+b_{k,n}p_{k-1,n}(x)+c_{k,n}p_{k-3,n}(x), \quad k=0,1,2,\ldots ,
\end{equation}
where $p_{-1,n} \equiv p_{-2,n} \equiv p_{-3,n} \equiv 0$.
The recurrence coefficients are related as follows
\begin{align}
    b_{k,n} &=a_{k,n}(a_{k-1,n}+a_{k,n}+a_{k+1,n}+t)=\tau^2a_{k,n}+\frac{k}{n}, &&\ k \geq 1,\label{eq: ba}\\
    c_{k,n} &=\tau^2a_{k-2,n}a_{k-1,n}a_{k,n}, &&\ k \geq 3. \label{eq: ca}
\end{align}
\end{lemma}
\begin{proof}
The lemma is known in much greater generality for general polynomial potentials 
$V$ and $W$, see \cite{BEH1}.

Explicit calculations that lead to \eqref{eq: 5trr}, \eqref{eq: ba} and \eqref{eq: ca}
for the case $t=0$ are given in \cite{Woe}. These calculations extend to
general  $t \in \mathbb R$ in a straightforward way.
\end{proof}

\subsection{Asymptotic behavior of recurrence coefficients}

Our next result concerns the asymptotic behavior of the recurrence
coefficients $a_{k,n}$, $b_{k,n}$ and $c_{k,n}$ as $k, n \to
\infty$ such that $k/n \to \xi > 0$. Let us introduce some
notation. We write
\[
\lim_{k/n \to \xi} X_{k,n}=X,
\]
if $\lim\limits_{j \to \infty} X_{k_j,n_j}=X$ holds
for every two sequences of positive integers $(k_j)$ and $(n_j)$ that satisfy $k_j,n_j \to \infty$
and $k_j/n_j \to \xi$ as $j \to \infty$. In the same spirit we write
\[
\lim_{\substack{k/n \to \xi \\ k \textrm{ even}}}X_{k,n}=X
\]
if $\lim\limits_{j \to \infty}X_{k_j,n_j}=X$ holds
for every sequence of positive even integers $(k_j)$ and every sequence of
positive integers $(n_j)$ that satisfy $k_j, n_j \to \infty$ and
$k_j/n_j \to \xi$ if $j \to \infty$. The limit with the subscript `odd'
is defined similarly.

The limiting behavior of the recurrence coefficients
$a_{k,n}$, $b_{k,n}$, $c_{k,n}$ as $k, n \to \infty$, $k/n \to \xi$ depends on
the values of $t \in \mathbb R$ and $\xi > 0$. We consider the coupling constant $\tau > 0$
to be fixed. We define the critical $\xi$-values
\begin{align*} \label{eq: xicr}
\xicr = \begin{cases}
    \frac{1}{4} (\tau^2-t)^2, & \text{ if } t < \tau^2, \\
    0, & \text{ if } t \geq \tau^2.
    \end{cases}
    \end{align*}
In the $t\xi$-plane the equation $\xi = \xi_{cr}$, $t < \tau^2$, defines
a semi-parabola that separates the upper half of the $t\xi$-plane
into two regions
\begin{equation*} \label{eq: C1C2}
    \begin{aligned}
    C_1 : &  \quad \xi > \xicr, \\
    C_2 : &  \quad 0 < \xi < \xicr, \quad - \infty < t < \tau^2,
\end{aligned} \end{equation*}
see Figure \ref{fig: phase diagram}.
We refer to $C_1$ as the one-cut case region, since the zeros of the
orthogonal polynomials $q_{k,n}$ accumulate
on one interval as $k,n \to \infty$ and $k/n \to \xi > \xicr$. If $t<\tau^2$ and $\xi \in (0,\xicr)$ the zeros of $q_{k,n}$ accumulate on two
disjoint intervals and therefore we call $C_2$ the two-cut case.

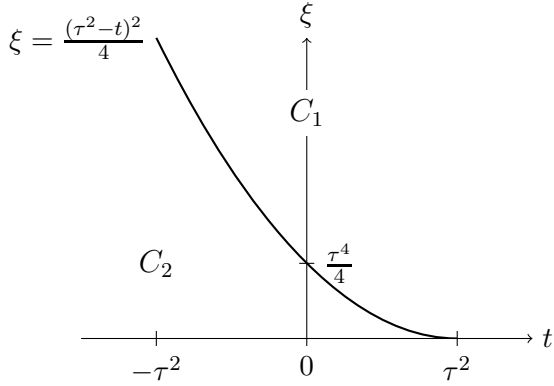
\begin{figure}[t]
\centering
\begin{tikzpicture}
\clip (-4.5,-1) rectangle (4.5,4.5);
\draw[->] (-3,0)--(3,0) node[right]{$t$};
\draw[->] (0,-0.1)node[below]{0}--(0,4) node[above]{$\xi$};
\draw (2,0.1)--(2,-0.1) node[below]{$\tau^2$};
\draw (-2,0.1)--(-2,-0.1) node[below]{$-\tau^2$};
\draw (-0.1,1)--(0.1,1) node[right] {$\frac{\tau^4}{4}$};
\draw (0,3) node[fill=white] {\large{$C_1$}};
\draw (-2,1) node[fill=white] {\large{$C_2$}};
\clip (-4,0) rectangle (4,8);
\draw[thick] (2,0) parabola  (-2,4) node[below,left]{$\xi=\frac{(\tau^2-t)^2}{4}$};
\end{tikzpicture}
\caption{$t\xi$-phase diagram: the semi-parabola separates the
one-cut case region $C_1$ from the two-cut case region $C_2$.}
\label{fig: phase diagram}
\end{figure}

We now state the following theorem.

\begin{theorem} \label{th:asympt recurrence coefficients}

\begin{enumerate}
\item[\rm (a)] If $\xi > \xicr$  then
the limits of $a_{k,n}$, $b_{k,n}$, $c_{k,n}$ as $k/n \to \xi$ exist
and we have
\begin{align}
 \lim_{k/n \to \xi} a_{k,n} & = a(\xi) :=
    \frac{\tau^2-t+\sqrt{(\tau^2-t)^2+12\xi}}{6}, \label{eq: a}\\
 \lim_{k/n \to \xi} b_{k,n} & = b(\xi) :=
    a(\xi)(3a(\xi)+t) = \tau^2 a(\xi)+\xi, \label{eq: b} \\
 \lim_{k/n \to \xi} c_{k,n} & = c(\xi) :=
    \tau^2 a^3(\xi). \label{eq: c}
 \end{align}

\item[\rm (b)] If $t<\tau^2$ and $0 < \xi < \xicr$, then
the recurrence coefficients $a_{k,n}$, $b_{k,n}$, $c_{k,n}$ exhibit
$2$-periodic behavior as $k/n \to \xi$ and we have
\begin{align}
 \lim_{\substack{k/n \to \xi \\ k \textrm{ even}}} a_{k,n} &= a_0(\xi) :=
    \frac{\tau^2-t-\sqrt{(\tau^2-t)^2-4\xi}}{2}, \label{eq: a0}  \\
    \lim_{\substack{k/n \to \xi \\ k \textrm{ odd}}} a_{k,n} &= a_1(\xi) :=
     \frac{\tau^2-t+\sqrt{(\tau^2-t)^2-4\xi}}{2}, \label{eq: a1} \\
  \lim_{\substack{k/n \to \xi \\ k \textrm{ even}}} b_{k,n} &= b_0(\xi) :=
    a_0(\xi)(a_0(\xi)+2a_1(\xi)+t), \label{eq: b0}  \\
    \lim_{\substack{k/n \to \xi \\ k \textrm{ odd}}} b_{k,n} &=  b_1(\xi) :=
    a_1(\xi)(2a_0(\xi)+a_1(\xi)+t), \label{eq: b1}  \\
    \lim_{\substack{k/n \to \xi \\ k \textrm{ even}}} c_{k,n} &= c_0(\xi) :=
    \tau^2a_0^2(\xi)a_1(\xi), \label{eq: c0}  \\
    \lim_{\substack{k/n \to \xi \\ k \textrm{ odd}}} c_{k,n} &= c_1(\xi)  :=
    \tau^2a_0(\xi)a_1^2(\xi). \label{eq: c1}   \end{align}

 \item[\rm (c)] If $t<\tau^2$ and $\xi = \xicr$, then
\begin{equation}
a(\xi) = a_0(\xi) = a_1(\xi), \quad b(\xi) = b_0(\xi) = b_1(\xi), \quad     c(\xi) = c_0(\xi) = c_1(\xi),
\label{eq: a a0 a1}
\end{equation}
 and all of the above limit relations continue to hold for $\xi = \xicr$.
 \end{enumerate}
\end{theorem}

\begin{proof}
The recurrence coefficients $a_{k,n}$ appear in a recurrence
relation for the orthogonal polynomials $p_{k,n}$. These
polynomials are orthogonal with respect to the weight \eqref{eq:
weight}. For this type of orthogonal polynomials the asymptotic
behavior of recurrence coefficients was studied by Bleher and Its
in two papers. The paper \cite{BI1} deals with the two-cut case $0< \xi < \xicr$.
In \cite{BI2} the critical case $\xi = \xicr$ is
studied and the results for the one-cut case $\xi > \xicr$ are
given as well.

Having \eqref{eq: a}, \eqref{eq: a0}, and \eqref{eq: a1}, the limits,
\eqref{eq: b}, \eqref{eq: c}, and \eqref{eq: b0}--\eqref{eq: c1}
follow directly from Lemma~\ref{lemma:relation coefficients}.
\end{proof}

\section{Asymptotic analysis in one-cut case} \label{sec: 1cc}

\subsection{Results from the literature} \label{subsec: bandedToeplitz}

In what follows we will associate with each $\xi > 0$  a function of the form
\begin{equation} \label{eq: symbols}
     s(w) = w + d^{(0)} +\frac{d^{(1)}}{w}+\frac{d^{(2)}}{w^2}+\frac{d^{(3)}}{w^3}, \qquad d^{(3)} \neq 0.
\end{equation}
Such functions appear as symbols of banded Toeplitz matrices \cite{BG},
and we need certain results \cite{DK1,KR} that were derived
in that context. Although we will not use Toeplitz matrices in this paper, we still
refer to $s$ as the symbol.

We denote the solutions of the algebraic equation $s(w) = z$ by $w_j(z)$, $j=1,\ldots,4$ and order
them by their absolute value, such that
\begin{equation}
    |w_1(z)| \geq |w_2(z)| \geq |w_3(z)| \geq |w_4(z)| > 0.
\label{eq: order 1}
\end{equation}
Typically, there is strict inequality in \eqref{eq: order 1}.
If for certain $z \in \mathbb C$ two solutions have the same absolute
value, then we pick an arbitrary numbering that satisfies \eqref{eq: order 1}.
Furthermore, we define
\begin{equation}
    \Gamma_j = \{z \in \mathbb C \mid |w_j(z)| = |w_{j+1}(z)|\}, \quad j=1,2,3,
    \label{eq: Gammaj}
\end{equation}
which are finite unions of analytic arcs and exceptional points,
see \cite{BG,DK1}. A point $z \in \mathbb C$ for which the algebraic
equation $s(w) = z$ has a multiple solution is called a branch
point.

We use the solutions $w_j(z)$ to the algebraic equation to define three Borel measures
\begin{equation}
    \ud \mu_j(z) = \frac{1}{2\pi i} \sum_{k=1}^{j}
    \left( \frac{{w_k'}_-(z)}{{w_k}_-(z)} - \frac{{w_k'}_+(z)}{{w_k}_+(z)}\right) \ud z,
\label{eq: muj}
\end{equation}
for $z \in \Gamma_j$, $j=1,2,3$. Here, it is assumed that every analytic arc of
$\Gamma_j$ is provided with an orientation and that $\ud z$
denotes the complex line element on $\Gamma_j$ according to this orientation. Furthermore, $w_{k\pm}(z)$ is the limiting
value of $w_{k}(\tilde z)$ as $\tilde z \to z$ from the $\pm$
side on each of the arcs in $\Gamma_j$. The $+$~side ($-$~side) is
on the left (right) if one traverses $\Gamma_j$ according to the
orientation.

The vector of measures $(\mu_1,\mu_2,\mu_3)$ is characterized as
the unique minimizer of a vector equilibrium problem.

\begin{theorem} \label{th: DuKuToeplitz}
Define the energy functional $E_0$ as
\[
    E_0(\rho_1,\rho_2,\rho_3) = I(\rho_1)-I(\rho_1,\rho_2)+I(\rho_2)-I(\rho_2,\rho_3)+I(\rho_3),
\]
where $\rho_1, \rho_2,\rho_3$ are positive measures on $\mathbb C$ with
finite logarithmic energy. Then the following statements hold.
\begin{itemize}
\item[\rm (a)] The vector of measures $(\mu_1, \mu_2, \mu_3)$ given
by \eqref{eq: muj} is the unique minimizer for the functional
 $E_0$ among all vectors $(\rho_1, \rho_2,\rho_3)$ of positive measures with
 finite logarithmic energy, satisfying
\begin{itemize}
\item[\rm (i)] $\supp (\rho_j) \subset \Gamma_j$, for $j=1,2,3$, and
\item[\rm (ii)] $\rho_1(\Gamma_1)= 1$, $\rho_2(\Gamma_2) = 2/3$, and $\rho_3(\Gamma_3) = 1/3$.
\end{itemize}
\item[\rm (b)]
The measures $\mu_1,\mu_2,\mu_3$ satisfy for some constant $\ell$
\begin{align}
\ell-2U^{\mu_{1}}(z)+U^{\mu_{2}}(z) &=\log \left| \frac{w_1(z)}{w_2(z)}\right| ,   \label{eq: EL 1} \\
U^{\mu_{1}}(z)-2U^{\mu_{2}}(z)+U^{\mu_{3}}(z) &= \log \left| \frac{w_2(z)}{w_3(z)}\right|,   \label{eq: EL 2} \\
U^{\mu_{2}}(z)-2U^{\mu_{3}}(z) &= \log \left| \frac{w_3(z)}{w_4(z)}\right|,   \label{eq: EL 3}
\end{align}
for every $z \in \C$.
\end{itemize}
\end{theorem}

\begin{proof}
The proof of Theorem \ref{th: DuKuToeplitz} can be found
in \cite{DK1}. The conditions (\ref{eq: EL 1})--(\ref{eq: EL 3}) are the
Euler-Lagrange variational conditions for the vector equilibrium
problem. Note that the right-hand side of the $j$th variational condition
vanishes if $z \in \Gamma_j$. In \cite{DK1} there also appear constants
$\ell_2$ and $\ell_3$ in \eqref{eq: EL 2} and \eqref{eq: EL 3}.
However, these constants vanish because $\Gamma_2$ and $\Gamma_3$
are unbounded and
\[ U^{\mu_j}(z) = - \mu_j(\Gamma_j) \, \log |z| + o(1), \qquad \text{as } z \to \infty. \qedhere
\]
\end{proof}

If the symbol \eqref{eq: symbols} depends on a parameter $\xi > 0$, say
\begin{equation} \label{eq: swxi}
    s(w;\xi) = w + d^{(0)}(\xi) + \frac{d^{(1)}(\xi)}{w} + \frac{d^{(2)}(\xi)}{w^2}
    + \frac{d^{(3)}(\xi)}{w^3},
    \end{equation}
then we use  $w_j(z; \xi)$, $\Gamma_j(\xi)$ and $\mu_j^{\xi}$
to indicate the dependence of the notions from \eqref{eq: order 1}, \eqref{eq: Gammaj}
and \eqref{eq: muj}, respectively, on the parameter $\xi$.

Next, we state a result of Kuijlaars and Rom\'an \cite{KR} on polynomials
satisfying certain recurrence relations.
It will be the key ingredient of the proof of Theorem \ref{th: zero distribution}
given in Section \ref{sec: zero distribution}.

\begin{theorem} \label{th: KR}
Let for each $n \in \mathbb N$ a sequence of monic polynomials
$(p_{k,n})_{k=0}^\infty$ be given where $\deg p_{k,n}=k$. Furthermore, suppose that
\begin{itemize}
\item[\rm (a)] these polynomials satisfy the recurrence relations
\begin{equation} \label{eq: fiveterm}
xp_{k,n}(x)=p_{k+1,n}(x)+d_{k,n}^{(0)}p_{k,n}+d_{k,n}^{(1)}p_{k-1,n}+d_{k,n}^{(2)}p_{k-2,n}+d_{k,n}^{(3)}p_{k-3,n},
\end{equation}
for certain real recurrence coefficients $d_{k,n}^{(j)}$, $j=0,1,2,3$;
\item[\rm (b)] the polynomials $p_{k,n}$ have real and simple zeros
$x_1^{k,n} < \cdots < x_k^{k,n}$ satisfying for each $k$ and $n$ the interlacing property
\[
x_j^{k+1,n}<x_j^{k,n}<x_{j+1}^{k+1,n}, \quad \textrm{ for }j=1,\ldots,k;
\]
\item[\rm (c)] for each $j = 0, 1,2,3$ the set of recurrence coefficients
\[ \{ d_{k,n}^{(j)} \mid k+1 \leq n \}
\]
is bounded;
\item[\rm (d)] there exist continuous functions $d^{(j)}:(0,+\infty) \to \R$, $j=0,1,2,3$,
    such that for each $\xi > 0$
\begin{equation} \label{eq: limitsdjxi}
    \lim_{k/n \to \xi} d_{k,n}^{(j)} = d^{(j)}(\xi),
\end{equation}
and $d^{(3)}(\xi) \neq 0$;
\item[\rm (e)] we have
\begin{equation*} \label{eq: Gamma1real}
    \Gamma_1(\xi) \subset \R, \qquad \text{for every } \xi > 0
    \end{equation*}
where $\Gamma_1(\xi)$ is the set defined as in \eqref{eq: Gammaj}
corresponding to the $\xi$-dependent function \eqref{eq: swxi}
with $d^{(j)}(\xi)$ coming from \eqref{eq: limitsdjxi}.
\end{itemize}

Then, the normalized zero counting measures $\nu(p_{k,n})$ have
a weak limit as $k,n \to \infty$ with $k/n \to \lambda>0$ given by
\begin{equation} \label{eq: averagemu1}
    \lim_{k/n \to \lambda} \nu(p_{k,n}) = \frac{1}{\lambda} \int_0^\lambda \mu_1^\xi \ud \xi,
\end{equation}
where for each $\xi > 0$ the measure $\mu_1^{\xi}$ is given by
\eqref{eq: muj} corresponding to the function \eqref{eq: swxi}.
\end{theorem}

\begin{proof}
This is \cite[Theorem 1.2]{KR} for the case of a five term recurrence \eqref{eq: fiveterm}.
\end{proof}

The intuition behind Theorem \ref{th: KR} is that the zeros of $p_{k,n}$
are eigenvalues of a $k \times k$ matrix with five non-zero diagonals
\begin{equation} \label{eq: Toeplitz}
    \begin{pmatrix} d_{0,n}^{(0)} & 1 & 0 & \ldots & \ldots & \ldots & 0 \\
    d_{1,n}^{(1)} & d_{1,n}^{(0)} & 1 & 0 &  & & \vdots \\
    d_{2,n}^{(2)} & d_{2,n}^{(1)} & d_{2,n}^{(0)} & 1 & 0 & & \vdots  \\
    d_{3,n}^{(3)} & d_{3,n}^{(2)} & d_{3,n}^{(1)} & d_{3,n}^{(0)} & 1 & 0 & \\
    0 & \ddots & \ddots & \ddots & \ddots & \ddots & \ddots \\
      &   \ddots & \ddots & \ddots & \ddots & \ddots & \ddots \\
      & &   \ddots & \ddots & \ddots & \ddots & \ddots &
    \end{pmatrix}. \end{equation}
Under the assumption \eqref{eq: limitsdjxi} the entries are slowly varying
along the diagonals if $k$ and $n$ are large, so that locally the matrix \eqref{eq: Toeplitz}
looks like a five-diagonal Toeplitz matrix.  Then for each $\xi$ one considers
the exact Toeplitz matrices with the entries $d^{(3)}(\xi), d^{(2)}(\xi), d^{(1)}(\xi), d^{(0)}(\xi), 1$
along the diagonals for which it is known, see \cite{BG,DK1,Hi,SS},  that the eigenvalues accumulate
on $\Gamma_1(\xi)$ as the size grows, with $\mu_1^{\xi}$ as  limiting normalized eigenvalue counting
measure.  The distribution of the eigenvalues of \eqref{eq: Toeplitz}
is then obtained by averaging of the measures $\mu_1^{\xi}$ as in \eqref{eq: averagemu1}.

\subsection{Analysis of the symbol $s_1$ in the one-cut case} \label{subsec: 1cc}

It will be our goal to apply Theorem \ref{th: KR} to the biorthogonal
polynomials $p_{k,n}$ that have the recurrence relation \eqref{eq: 5trr}.
The recurrence coefficients in \eqref{eq: 5trr} have the appropriate
limits only in the one-cut case. We discuss this case first.

We therefore assume that $\xi > \xicr$. In that case we have
by Theorem \ref{th:asympt recurrence coefficients} that the recurrence coefficients
$b_{k,n}$, $c_{k,n}$ have limits $b(\xi)$ and $c(\xi)$ given by \eqref{eq: b} and \eqref{eq: c}
as $k, n \to \infty$ and $k/n \to \xi$.  We therefore associate with $\xi > \xicr$ the
symbol
\begin{equation}
    s_1(w;\xi) = w + \frac{b(\xi)}{w} + \frac{c(\xi)}{w^3}.
\label{eq: s1}
\end{equation}
As already noted before, we use $w_j(z;\xi)$ for  $j=1,2,3,4$, and $\Gamma_j(\xi)$, $\mu_j^{\xi}$
for $j=1,2,3$ to denote the quantities related to the symbol \eqref{eq: s1}.

In order to apply Theorem \ref{th: KR} we need to know that
$\Gamma_1(\xi) \subset \mathbb R$, see assumption (e) in Theorem
\ref{th: KR}. Then the proof of Theorem~\ref{th:
equilibrium problem} follows the approach outlined in
\cite[Section 7]{KR}. So, to obtain an external field $V_1$ acting
on $\nu_1$ and an upper constraint $\sigma$ acting on $\nu_2$ we
will need that $\Gamma_1(\xi)$ is an increasing  and
$\Gamma_2(\xi)$ a decreasing set as a function of $\xi$. These
features are contained in the following theorem.

\begin{theorem} \label{th: gamma 1cc}
Let $\tau > 0$ and $t \in \mathbb R$. Then for every $\xi > \xicr$
we have that
\[ \Gamma_1(\xi) \subset \mathbb R, \quad \Gamma_2(\xi) \subset i \mathbb R,
    \quad \text{and} \quad  \Gamma_3(\xi) \subset \mathbb R. \]
For $\xi>\xicr$ the set $\Gamma_1(\xi)$ is increasing as a function of $\xi$,
while the set $\Gamma_2(\xi)$ is decreasing.
More precisely, there exist $\alpha(\xi), \gamma(\xi) > 0$ such that
\begin{equation}
\begin{aligned}
    \Gamma_1(\xi) & = [-\alpha(\xi),\alpha(\xi)], \\
    \Gamma_2(\xi) & = i\R \setminus (-i\gamma(\xi),i\gamma(\xi)), \\
    \Gamma_3(\xi) & = \R.
    \end{aligned}
\label{eq: gamma 1cc}
\end{equation}
In addition, we have that
\begin{itemize}
\item[\rm (a)] $\xi \mapsto \alpha(\xi)$ is strictly increasing for $\xi \geq \xicr$ with
\begin{align*}
    \lim_{\xi \to +\infty} \alpha(\xi) = +\infty, \quad \text{and} \quad
    \lim_{\xi \to 0+} \alpha(\xi)=0 \text{ if } t \geq \tau^2,
    \end{align*}
\item[\rm (b)] $\xi \mapsto \gamma(\xi)$ is strictly increasing for $\xi \geq \xicr$ with
\begin{align*}
    \lim_{\xi \to +\infty} \gamma(\xi) = +\infty, \quad \text{and} \quad
    \lim_{\xi \to 0+} \gamma(\xi) = y^* \text{ if } t \geq \tau^2,
    \end{align*}
    see \eqref{eq: ystar} for the definition of $y^*$.
\end{itemize}
\end{theorem}
Figure \ref{fig: 1cc} shows the sets $\Gamma_1(\xi)$, $\Gamma_2(\xi)$, and $\Gamma_3(\xi)$
in the complex plane in the one-cut case.

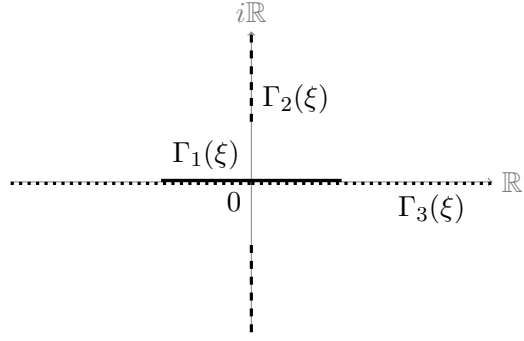
\begin{figure}
\begin{center}
\begin{tikzpicture}[scale=0.4]
\draw[help lines,->] (-8,0)--(8,0) node[right]{$\mathbb R$};
\draw[help lines,->] (0,-5)--(0,5) node[above]{$i \mathbb R$};
\draw (0,0) node[below left] {0};
\draw[very thick] (-3,0.05)-- node[near start,above]{$\Gamma_1(\xi)$}(3,0.05);
\draw[dashed, very thick] (0,-5)--(0,-2)  (0,2)-- node[near start,right]{$\Gamma_2(\xi)$}(0,5);
\draw[dotted, very thick] (-8,-0.05)--(0,-0.05);
\draw[dotted, very thick] (0,-0.05)-- node[near end,below]{$\Gamma_3(\xi)$} (8,-0.05);
\end{tikzpicture}
\caption{The sets $\Gamma_1(\xi)$ (plain), $\Gamma_2(\xi)$ (dashed), and $\Gamma_3(\xi)$ (dotted)
in the one-cut case. We have that $\Gamma_1(\xi) = [-\alpha(\xi), \alpha(\xi)]$,
$\Gamma_2(\xi) = (-i \infty, -i \gamma(\xi)] \cup [i \gamma(\xi), i \infty)$, and $\Gamma_3(\xi) = \mathbb R$.}
\label{fig: 1cc}
\end{center}
\end{figure}

The proof of Theorem \ref{th: gamma 1cc} is given in the next
subsection. Here we note that as a consequence of
Theorem \ref{th: gamma 1cc} two consecutive sets among
$\Gamma_1(\xi)$, $\Gamma_2(\xi)$ and $\Gamma_3(\xi)$
are not overlapping, which implies that \eqref{eq: muj}
can be written more simply as
\begin{equation}
    \ud \mu_j^{\xi} (z) =
    \frac{1}{2\pi i}
    \left( \frac{{w_j'}_-(z;\xi)}{{w_j}_-(z;\xi)} - \frac{{w_j'}_+(z;\xi)}{{w_j}_+(z;\xi)}\right) \ud z,
    \quad z \in \Gamma_j(\xi), \quad j=1,2,3.
\label{eq: muj 1cc}
\end{equation}

\subsection{Proof of Theorem \ref{th: gamma 1cc}}

\subsubsection{Branch points}

As a first step to the proof of Theorem \ref{th: gamma 1cc}
we calculate the branch points for the algebraic equation $s_1(w;\xi) = z$.

\begin{lemma} \label{lem: branch points1}
Let $\xi > \xicr$. Define $u(\xi),v(\xi)>0$ such that
\begin{align}
u(\xi)^2 & = \frac{\tau^2a(\xi)+\xi+\sqrt{(\tau^2a(\xi)+\xi)^2+12\tau^2a^3(\xi)}}{2} > 0, \label{eq: u} \\
-v(\xi)^2 & = \frac{\tau^2a(\xi)+\xi-\sqrt{(\tau^2a(\xi)+\xi)^2+12\tau^2a^3(\xi)}}{2} < 0. \label{eq: v}
\end{align}
Then, the branch points are $\pm \alpha(\xi),\pm i \gamma(\xi) $ where
\begin{align}
\alpha(\xi) & = s_1(u(\xi),\xi)=2u(\xi)-\frac{2v(\xi)^2}{3u(\xi)}>0 \nonumber \\
-i\gamma(\xi) & =s_1(iv(\xi),\xi)=i\left( 2v(\xi)-\frac{2u(\xi)^2}{3v(\xi)}\right), \label{eq: alpha gamma}
\end{align}
and $\alpha(\xi),\gamma(\xi)>0$.
Moreover, we can rewrite the symbol as
\[
s_1(w;\xi)=w+\frac{u(\xi)^2-v(\xi)^2}{w}+\frac{u(\xi)^2v(\xi)^2}{3w^3}.
\]
\end{lemma}
\begin{proof}
The proof is straightforward. Note that $\pm u(\xi)$ and $\pm i v(\xi)$
are the zeros of the derivative of $s_1(w;\xi)$ with respect to
$w$. From (\ref{eq: u}) and (\ref{eq: v}) it can be shown that
$3u(\xi)^2-v(\xi)^2>u(\xi)^2-3v(\xi)^2>0$ if $\xi > \xicr$. The positivity of $\alpha(\xi)$ and $\gamma(\xi)$ follows from these inequalities.
\end{proof}

\subsubsection{The restriction of $s_1$ to $\R$ and $i\R$}

Consider the algebraic equation
\begin{equation} \label{eq: alg eq s1 R}
s_1(x;\xi)=x+\frac{u(\xi)^2-v(\xi)^2}{x}+\frac{u(\xi)^2v(\xi)^2}{3x^3}=z,
\end{equation}
for real values of $z$. Figure \ref{fig: s1R} shows a sketch of
the graph of the symbol $s_1(x; \xi)$ for real values of $x$.
The solutions to \eqref{eq: alg eq s1 R} for $z \in \mathbb R$ are real or
come in pairs of complex conjugate numbers. By Lemma \ref{lem: branch points1},
$\pm u(\xi)$ is a double
solution of the equation if $z=\pm \alpha(\xi)$. Then, as is clear
from the graph in Figure \ref{fig: s1R}, the other two
solutions are complex conjugate and their modulus is smaller than
$u(\xi)$.

Now consider the restriction of the symbol $s_1$ to the imaginary axis
\[
s_1(iy;\xi)=i\left(y-\frac{u(\xi)^2-v(\xi)^2}{y}+\frac{u(\xi)^2v(\xi)^2}{3y^3}\right).
\]
We claim that $s_1$ has four purely imaginary zeros: $\pm iy_1,
\pm iy_2$, where $y_1 > y_2 > 0$. To see this, recall that
$u(\xi)^2-3v(\xi)^2>0$, so that $(u(\xi)-\sqrt 3 v(\xi))(u(\xi)+\sqrt 3 v(\xi)/3)> 0$. This can
be rewritten as
\begin{equation} \label{eq: u^2-v^2}
u(\xi)^2-v(\xi)^2> \frac{2u(\xi)v(\xi)}{\sqrt 3}.
\end{equation}
Now consider the biquadratic equation
\[
y^4-(u(\xi)^2-v(\xi)^2)y^2+\frac{u(\xi)^2v(\xi)^2}{3}=0.
\]
By (\ref{eq: u^2-v^2}) the discriminant is positive and less than
$(u(\xi)^2-v(\xi)^2)^2 >0$. Then the claim follows. Figure \ref{fig: s1iR}
shows the graph of the restriction of $s_1$ to the imaginary axis.

Consider the algebraic equation
\begin{equation}
y-\frac{u(\xi)^2-v(\xi)^2}{y}+\frac{u(\xi)^2v(\xi)^2}{3y^3}=z,
\label{eq: s1iR}
\end{equation}
for $z \in \R$. For $z=\pm \gamma(\xi)$ the equation has four real
solutions: the double solution $\mp v(\xi)$ and two strictly
positive/negative solutions. The latter are simple and $v(\xi)$ lies
between their moduli.

\begin{figure}[t]
\centering
\subfigure[Graph of $x \mapsto s_1(x;\xi)$.]{
\begin{tikzpicture}[scale=0.48]
\clip (-5,-5)rectangle (6,6.5);
\draw[->] (-5,0) -- (5,0) node[font=\footnotesize,right]{$x$};
\draw[->] (0,-5)-- (0,5);
\draw[thick] (0.1,4.8) .. controls (0.2,1) and (1,1.1) .. (4.8,4.9);
\draw[thick] (-0.1,-4.8) .. controls (-0.2,-1) and (-1,-1.1) .. (-4.8,-4.9);
\draw[help lines] (1,0) -- (1,2);
\draw (1,0.1)--(1,-0.1) node[font=\footnotesize,below]{$u(\xi)$};
\draw[help lines] (0,2)--(1,2);
\draw (0.1,2)--(-0.1,2) node[font=\footnotesize,left]{$\alpha(\xi)$};
\draw[very thick] (0.05,2)--(0.05,-2) node[font=\footnotesize,near start,right,fill=white]{$\Gamma_1(\xi)$};
\draw[very thick,dotted] (-0.05,5)--(-0.05,-5) node[font=\footnotesize,very near start, left]{$\Gamma_3(\xi)$};
\draw (1,2) circle (1pt);
\draw[help lines] (-1,0) -- (-1,-2);
\draw (-1,-0.1)--(-1,0.1) node[font=\footnotesize,above]{$-u(\xi)$};
\draw[help lines] (0,-2)--(-1,-2);
\draw (-0.1,-2)--(0.1,-2) node[font=\footnotesize,right]{$-\alpha(\xi)$};
\filldraw (-1,-2) circle (1pt);
\end{tikzpicture}
\label{fig: s1R}
} \qquad
\subfigure[Graph of $y \mapsto -i s_1(iy;\xi)$.]{
\begin{tikzpicture}[scale=0.48]
\clip (-5,-5)rectangle (6,6.5);
\draw[->] (-5,0) -- (5,0) node[font=\footnotesize,right]{$y$};
\draw[->] (0,-5)-- (0,5);
\draw[thick] (0.1,4.8) .. controls (2,-10) and (3,2.9) .. (4.9,4.8);
\draw[thick] (-0.1,-4.8) .. controls (-2,+10) and (-3,-2.9) .. (-4.9,-4.8);
\draw[help lines] (1.87,-2.23)--(1.87,0);
\draw[help lines] (1.87,-2.23)--(0,-2.23);
\filldraw (1.87,-2.23) circle (1pt);
\draw (1.87,-0.1)--(1.87,0.1) node[font=\footnotesize,above]{$v(\xi)$};
\draw (0.1,-2.23)--(-0.1,-2.23)node[font=\footnotesize,left,fill=white]{$-\gamma(\xi)$};
\draw[very thick, dashed] (0,2.23)--(0,5) node[font=\footnotesize,midway,left]{$-i \Gamma_2(\xi)$} (0,-2.23)--(0,-5);
\draw[help lines] (-1.87,2.23)--(-1.87,0);
\draw[help lines] (-1.87,2.23)--(0,2.23);
\filldraw (-1.87,2.23) circle (1pt);
\draw (-1.87,0.1)--(-1.87,-0.1) node[font=\footnotesize,below]{$-v(\xi)$};
\draw (-0.1,2.23)--(0.1,2.23)node[font=\footnotesize,right,fill=white]{$\gamma(\xi)$};
\end{tikzpicture}
\label{fig: s1iR} } \label{fig: s1R and s1iR} \caption{Graphs
of $x \mapsto s_1(x;\xi)$ and $y \mapsto -i s_1(iy;\xi)$, the restrictions
of the symbol $s_1$ to the real and imaginary axes.}
\end{figure}
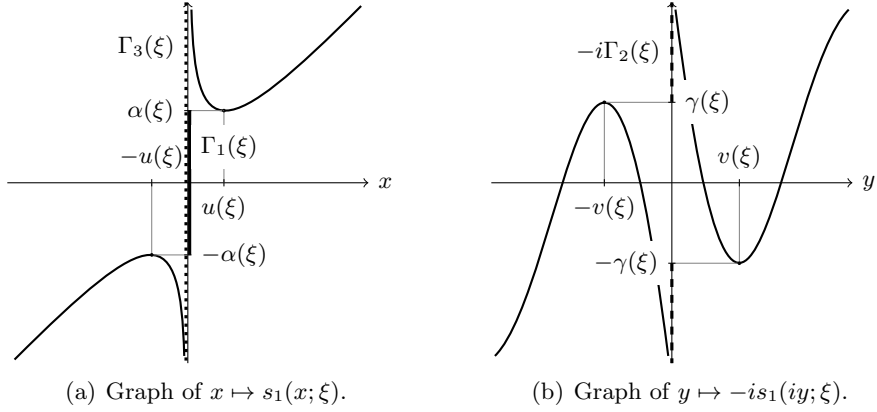

\subsubsection{Auxiliary lemmas}
%The bulk of the ideas in this and the following parts are based on \cite{KR}.
We start by proving two lemmas.

\begin{lemma}\label{lemma: axes 1cc}
Let $\xi > \xicr$.  Assume that $w_a,w_b \in \mathbb C$ are
such that $w_a \neq w_b$, $|w_a|=|w_b|$, and
$s_1(w_a;\xi)=s_1(w_b;\xi)=z$. Then $z \in \mathbb R \cup i
\mathbb R$.
\end{lemma}

\begin{proof}
The complex numbers $w_a^2$ and $w_b^2$ lie on a circle of radius
$\rho=|w_a^2|=|w_b^2|$ centered at the origin of the complex
plane. We can factorize $s_1(w;\xi)$ as
\[
s_1(w;\xi)=\frac{(w^2+y_1^2)(w^2+y_2^2)}{w^3},
\]
where $\pm i y_1,\pm i y_2$ are the zeros of the symbol $s_1(w;\xi)$. Thus
\[
|s_1(w,\xi)|=\frac{\dist(w^2,-y_1^2)\dist(w^2,-y_2^2)}{\rho^{3/2}}, \qquad \textrm{if } |w^2|=\rho.
\]
Since $-y_1^2,-y_2^2<0$, it follows that
\[
[-\pi,\pi]\to \mathbb R: \theta \mapsto \frac{\dist(\rho e^{i\theta},-y_1^2)\dist(\rho e^{i\theta},-y_2^2)}{\rho^{3/2}}
\]
is an even function that is strictly decreasing as $\theta$ increases from $0$ to $\pi$.
Thus, equality
\[
   \dist(\rho e^{i\theta_a},-y_1^2)\dist(\rho e^{i\theta_a},-y_2^2)=\dist(\rho e^{i\theta_b},-y_1^2)\dist(\rho e^{i\theta_b},-y_2^2),
\]
with $\theta_a,\theta_b \in [-\pi,\pi]$, can only occur if
$\theta_b=\pm \theta_a$. Then it follows from the assumptions of
the lemma that $w_a^2=w_b^2$ or $w_a^2=\overline{w_b^2}$. This
gives rise to three possible cases: $w_a=-w_b$ and $w_a=\pm
\overline w_b$. Substituting these results into the algebraic
equation yields
\begin{align*}
z=& s_1(w_a;\xi)=s_1(-w_b;\xi)=-s_1(w_b;\xi)=-z, \\
z=& s_1(w_a;\xi)=s_1(\overline{w_b};\xi)=\overline{s_1(w_b;\xi)}=\overline z, \textrm{ or}\\
z=& s_1(w_a;\xi)=s_1(-\overline{w_b};\xi)=-\overline{s_1(w_b;\xi)}=-\overline z.
\end{align*}
In all three cases $z \in \mathbb R \cup i \mathbb R$.
\end{proof}

\begin{lemma} \label{lemma: gamma 1cc}
Let $\xi  > \xicr$. Then $\Gamma_j(\xi) \subset \R \cup i\R$ for $j=1,2,3$. Moreover,
\[
    \Gamma_1(\xi) \cap \Gamma_2(\xi) = \Gamma_2(\xi) \cap \Gamma_3(\xi)= \emptyset.
\]
\end{lemma}
\begin{proof}
Fix $j \in \{1,2,3\}$. If $z\in \Gamma_j(\xi)$ and $s_1(w,\xi)=z$
has a double solution, then $z$ is one of the branch points $\pm
\alpha(\xi)$ or $\pm i \gamma(\xi)$. So $z \in \R \cup i\R$.
If $z\in \Gamma_j(\xi)$ and $s_1(w,\xi)=z$ does not have a double
solution, then $w_j(\xi) \neq w_{j+1}(\xi)$ and $|w_j(\xi)| =
|w_{j+1}(\xi)|$. Then it follows from Lemma \ref{lemma: axes 1cc}
that $z \in \R \cup i\R$. This proves that $\Gamma_j(\xi) \subset \R
\cup i\R$.

Now assume that for a certain value of $z \in \R \cup i\R$ and $j
\in \{1,2\}$ $z \in \Gamma_j(\xi) \cap \Gamma_{j+1}(\xi)$. Then
$|w_j(z;\xi)|=|w_{j+1}(z;\xi)|=|w_{j+2}(z;\xi)|$. We distinguish
four cases. First consider  the case $z=0$. Recall that the
algebraic equation $s_1(w,\xi)=0$ has four different imaginary
solutions: $\pm iy_1$ and $\pm i y_2$. This contradicts the
assumption. $z$ cannot be one of the branch points $\{\pm
\alpha(\xi), \pm i \gamma(\xi) \}$ either. Next assume that $z \in
\mathbb R \setminus \{0,\pm \alpha(\xi)\}$. From the proof of Lemma
\ref{lemma: axes 1cc} it follows that
$w_j(z;\xi)=\overline{w_{j+1}(z;\xi)}=w_{j+2}(z;\xi)$, so that two
roots coincide and $z$ is a branch point. This possibility was
already excluded. If $z \in i\mathbb R \setminus \{0,\pm
i\gamma(\xi)\}$ we analogously obtain
$w_j(z;\xi)=-\overline{w_{j+1}(z;\xi)}=w_{j+2}(z;\xi)$. Then again
$z$ is a branch point. Since we could exclude all four possible
cases, we can conclude that $\Gamma_1(\xi) \cap
\Gamma_2(\xi)=\Gamma_2(\xi) \cap \Gamma_3(\xi)= \emptyset$.
\end{proof}

\subsubsection{Transformed symbol}
Lemma \ref{lemma: gamma 1cc} is used in the proof of Theorem
\ref{th: gamma 1cc}. Another ingredient of that proof is the
transformed symbol $S_1$, defined as
\begin{equation}
S_1(W;\xi)=s_1(a(\xi) W;\xi)=a(\xi) \left( W+ \frac{3}{W}\right)+\frac{t}{W}+\frac{\tau^2}{W^3},
\label{eq: S1}
\end{equation}
see \eqref{eq: s1}, \eqref{eq: b}, and \eqref{eq: c}. The
advantage of the transformed symbol $S_1$ over the symbol $s_1$ is
that it depends on $\xi$ in a much easier way. This will
significantly simplify the calculations.

We denote the zeros of $\frac{\mathrm d S_1}{\mathrm d W}(W;\xi)$
by $\pm U(\xi)$ and $\pm iV(\xi)$, where $U(\xi),V(\xi)>0$. One
can check that $\pm U(\xi)= \pm u(\xi)/a(\xi)$ and $\pm iV(\xi)=
\pm iv(\xi)/a(\xi)$, where $u(\xi)$ and $v(\xi)$ are defined as in
(\ref{eq: u}) and (\ref{eq: v}). Moreover, $S_1$ gives rise to the
same branch points as $s_1$ does, namely
\begin{align}
S_1(\pm U(\xi);\xi) &=s_1(\pm u(\xi);\xi)=\pm \alpha(\xi)  \textrm{ and} \nonumber \\
S_1(\pm iV(\xi);\xi) &=s_1(\pm iv(\xi);\xi)=\mp i \gamma(\xi). \label{eq: branch points S1}
\end{align}

\subsubsection{Proof of Theorem \ref{th: gamma 1cc}}

\begin{proof}  %[Proof of Theorem \ref{th: gamma 1cc}.]

We will use the restriction of the symbol to the real axis to
determine the sets $\Gamma_j(\xi) \cap \R$. Figure \ref{fig: s1R}
shows the typical form of the graph of this restriction.

Recall the algebraic equation \eqref{eq: alg eq s1 R}. For $z \in
(-\alpha(\xi),\alpha(\xi))$ we find two pairs of complex conjugate
solutions. Therefore, $(-\alpha(\xi),\alpha(\xi))\subset \Gamma_1(\xi)
\cap \Gamma_3(\xi)$. If $z=\pm \alpha(\xi)$ the equation has the
double solution $\pm u(\xi)$ and one pair of complex conjugate
solutions with smaller modulus. We conclude that
$[-\alpha(\xi),\alpha(\xi)]\subset \Gamma_1(\xi) \cap \Gamma_3(\xi)$.
Next, take $z> \alpha(\xi)$. The equation then has two real solutions
and one pair of complex conjugate solutions. The largest real
solution is denoted by $x_1(z)$, the smallest real solution by
$x_2(z)$, and the complex conjugate solutions by
$x_3(z)=\overline{x_4(z)}$. From Lemma \ref{lemma: gamma 1cc} it
follows that the equation does not admit three solutions with
equal moduli. Thus, the situation $|x_2(z)|=|x_3(z)|=|x_4(z)|$
cannot occur if $z>\alpha(\xi)$. Because
$|x_2(\alpha(\xi))|>|x_3(\alpha(\xi))|=|x_4(\alpha(\xi))|$ and the roots of
a polynomial equation are continuous with respect to the
coefficients of the equation, we have that $(\alpha(\xi),+\infty)
\subset \Gamma_3(\xi)$ and $(\alpha(\xi),+\infty) \cap
\Gamma_1(\xi)=(\alpha(\xi),+\infty) \cap
\Gamma_2(\xi)=\emptyset$. We obtain similar results if $z \in
(-\infty,-\alpha(\xi))$. At this moment we conclude
\begin{equation} \label{eq: gamma 1cc 1}
\Gamma_1(\xi) \cap \R=[-\alpha(\xi),\alpha(\xi)],\qquad \Gamma_2(\xi)
\cap \R = \emptyset, \quad \textrm{and } \R \subset \Gamma_3(\xi).
\end{equation}

Let us now restrict the symbol $s_1$ to the imaginary axis to
determine the sets $\Gamma_j(\xi) \cap i\R$. Figure \ref{fig:
s1iR} shows a typical graph of $y \mapsto -is_1(iy;\xi)$. We proceed in a
similar way. Consider for real values of $z$ the algebraic
equation \eqref{eq: s1iR}. For $z \in (0,\gamma(\xi))$ this equation
has four real solutions with different moduli. Therefore,
$(0,i\gamma(\xi))\cap \Gamma_j(\xi)=\emptyset$ for $j=1,2,3$. For
$z=0$ we obtain the solutions $\pm y_1,\pm y_2$. Since $y_1>y_2$,
0 belongs to $\Gamma_1(\xi)$ and $\Gamma_3(\xi)$, but not to
$\Gamma_2(\xi)$. This is consistent with \eqref{eq: gamma 1cc 1}.
If $z=\gamma(\xi)$, the equation has the double solution $-v(\xi)$, a real
solution with modulus less than $v(\xi)$, and a real solution with
modulus greater than $v(\xi)$. Thus, $i\gamma(\xi)$ belongs to
$\Gamma_2(\xi)$. Next take $z>\gamma(\xi)$. The equation then has two
different real solutions and one pair of complex conjugate
solutions. Using a similar continuity argument as before we obtain
$(i\gamma(\xi),+i\infty) \subset \Gamma_2(\xi)$ and
$(i\gamma(\xi),+i\infty) \cap \Gamma_1(\xi)=(i\gamma(\xi),+i\infty) \cap
\Gamma_3(\xi)=\emptyset$. The same procedure works for $z<0$.
Summarized this is
\begin{equation}\label{eq: gamma 1cc 2}
 \Gamma_1(\xi) \cap i\R=\Gamma_3(\xi) \cap i\R=\{0\} \quad
 \textrm{and} \quad  \Gamma_2(\xi) \cap i\R=i\R \setminus (-i\gamma(\xi),i\gamma(\xi)).
\end{equation}
Combining (\ref{eq: gamma 1cc 1}) and (\ref{eq: gamma 1cc 2}) proves (\ref{eq: gamma 1cc}).

To prove (a) and (b) the transformed symbol $S_1$ will be useful.
First, we prove that $\xi \mapsto \alpha(\xi)$ is an increasing
function. Take $\xi > \xicr$. It follows from \eqref{eq: branch
points S1} that
\begin{align*}
\frac{\textrm d \alpha(\xi)}{\textrm d \xi} &= \frac{\textrm d S_1}{\textrm d \xi}(U(\xi);\xi) \\
 &= \frac{\partial S_1}{\partial W}(U(\xi);\xi)
\frac{\textrm d U(\xi)}{\textrm d \xi}+\frac{\partial S_1}{\partial \xi}(U(\xi);\xi).
\end{align*}
Since $U(\xi)$ is a zero of the derivative of $S_1$, the first
term on the right-hand side vanishes. Taking the partial derivative with respect to $\xi$ in \eqref{eq:
S1} yields
\[
\frac{\textrm d \alpha(\xi)}{\textrm d \xi} = \left( U(\xi)+\frac{3}{U(\xi)} \right)\frac{\textrm d a(\xi)}{\textrm d \xi}.
\]
Observe that the function $\xi \mapsto a(\xi)$ is increasing, see
(\ref{eq: a}). Because $U(\xi)>0$ we conclude that $\xi \mapsto
\alpha(\xi)$ is an increasing function. It can be proved similarly that $\xi \mapsto
\gamma(\xi)$ is an increasing function for $\xi > \xicr$.

Next, we show that $\lim\limits_{\xi \to +\infty}\alpha(\xi)=+\infty$.
From (\ref{eq: a}) it follows that $a(\xi) \sim \sqrt \xi$ as $\xi
\to \infty$. Using (\ref{eq: u}) -- (\ref{eq: alpha gamma}) we
compute
\begin{multline*}
\alpha(\xi) = \frac{2}{3 u(\xi)}(3u(\xi)^2-v(\xi)^2) \\
=\frac{2}{3 u(\xi)} \left(2(\tau^2a(\xi)+\xi)+\sqrt{(\tau^2a(\xi)+\xi)^2+12 \tau^2a(\xi)^3} \right) \sim \sqrt \xi.
\end{multline*}
Therefore, $\xi \mapsto \alpha(\xi)$ is unbounded. The limit
$\lim\limits_{\xi \to +\infty}\gamma(\xi)=+\infty$ can be proved
analogously.

Our final task is to calculate the limits of $\alpha(\xi)$ and
$\gamma(\xi)$ as $\xi \to 0+$ for $t \geq \tau^2$. In the limit
$\xi=0$ the transformed symbol is
\begin{equation}
S_1(W;0)=\frac{t}{W}+\frac{\tau^2}{W^3},
\label{eq: S1 00}
\end{equation}
because $\lim\limits_{\xi \to 0+}a(\xi)=0$. Its derivative
\[
\frac{\textrm d S_1}{\textrm d W}(W;0)=-\frac{t}{W^2}-3\frac{\tau^2}{W^4},
\]
has only two zeros, denoted by
\[
\pm i V(0)=\pm i \sqrt{\frac{3 \tau^2}{t}}.
\]
It follows that
\[
\lim_{\xi \to 0 }V(\xi)=V(0)=\sqrt{\frac{3 \tau^2}{t}} \quad \textrm{and} \quad \lim_{\xi \to 0 }U(\xi)=+\infty.
\]
Substituting these results into (\ref{eq: S1 00}) yields
\begin{align*}
\lim_{\xi \to 0}\alpha(\xi)& = \lim_{\xi \to 0}S_1(U(\xi);\xi)=\lim_{W \to \infty}\frac{t}{W}+\frac{\tau^2}{W^3}=0  \textrm{ and} \\
\lim_{\xi \to 0}i\gamma(\xi)& = \lim_{\xi \to 0}S_1(-iV(\xi);\xi)=\frac{t}{-iV(0)}+\frac{\tau^2}{(-iV(0))^3}=iy^*,
\end{align*}
see \eqref{eq: ystar} for $y^*$.
\end{proof}

\section{Asymptotic analysis in two-cut case} \label{sec: 2cc}

\subsection{Doubling the recurrence relation} \label{subsec: iteration 2cc}
In Section \ref{sec: 1cc} we introduced and analyzed
the symbol $s_1(w; \xi)$ in the one-cut case.

Here we want to do something similar for the two-cut case. Note, however,
that the recurrence coefficients $b_{k,n}$ and
$c_{k,n}$ in \eqref{eq: 5trr} do not have limits as $k/n \to \xi \in (0, \xicr)$.
Instead, there is two-periodic limiting behavior given
by \eqref{eq: b0}--\eqref{eq: c1}. This is a fundamental difference with
the one-cut case and, therefore, the construction from the previous section
does not apply to the two-cut case.

We analyze the two-cut case by doubling the
recurrence relation \eqref{eq: 5trr}. This yields a new recurrence
relation in which the coefficients have limits. Indeed, we obtain
\begin{multline}
    x^2p_{k,n}(x) = p_{k+2,n}(x)+A_{k,n}p_{k,n}(x) \\
        +B_{k,n}p_{k-2,n}(x)+C_{k,n}p_{k-4,n}(x)+D_{k,n}p_{k-6,n}(x),
\label{eq: iterated rr}
\end{multline}
where
\begin{align*}
A_{k,n} &= b_{k,n}+b_{k+1,n}, \\
B_{k,n} &= c_{k,n}+c_{k+1,n}+b_{k,n}b_{k-1,n},\\
C_{k,n} &= b_{k,n}c_{k-1,n}+c_{k,n}b_{k-3,n}, \\
D_{k,n} &= c_{k,n}c_{k-3,n}.
\end{align*}
The limits of these coefficients as $k, n \to \infty$ such that $k/n \to \xi \in(0, \xicr)$ exist
and are denoted by
\begin{align}
A(\xi)&=\lim_{k/n \to \xi}A_{k,n}=b_0(\xi)+b_1(\xi), \label{eq: A 2cc} \\
B(\xi)&=\lim_{k/n \to \xi}B_{k,n}=c_0(\xi)+c_1(\xi)+ b_0(\xi)b_1(\xi),  \label{eq: B 2cc}\\
C(\xi)&=\lim_{k/n \to \xi}C_{k,n}=b_0(\xi)c_1(\xi) + c_0(\xi)b_1(\xi), \label{eq: C 2cc}\\
D(\xi)&=\lim_{k/n \to \xi}D_{k,n}=c_0(\xi)c_1(\xi), \label{eq: D 2cc}
\end{align}
see also  \eqref{eq: b0}--\eqref{eq: c1}.

In analogy with \eqref{eq: s1} we define the symbol
\begin{equation}
\widehat s_2 (w;\xi)=w+A(\xi)+\frac{B(\xi)}{w}+\frac{C(\xi)}{w^2}+\frac{D(\xi)}{w^3}.
\label{eq: s2 hat}
\end{equation}
We use the subscript $2$ to remind us that we are in the two-cut
case. The hat refers to the fact that the recurrence relation was
doubled to obtain this symbol. Also the quantities that are
associated with the symbol \eqref{eq: s2 hat} will be equipped
with a hat.
Thus we use $\widehat{w}_j(z; \xi)$ to denote the solutions of $\widehat s_2(w;\xi)=z$
with the usual ordering
\[
    |\widehat w_1(z;\xi)| \geq |\widehat w_2(z;\xi)| \geq |\widehat w_3(z;\xi)| \geq| \widehat w_4(z;\xi)|.
\]
Furthermore, we have $\widehat \Gamma_j(\xi)$ and $\widehat \mu_j^{\xi}$
for $j=1,2,3$.

It is remarkable that $\widehat s_2$ has the factorization
\begin{equation}
\widehat s_2(w; \xi) =\frac{(w+ \xi)^2}{w^3} (w^2- t \tau^2 w+ \tau^4 w+ \tau^4 \xi).
\label{eq: s2 hat factorization}
\end{equation}
This follows from \eqref{eq: A 2cc}--\eqref{eq: D 2cc} and the explicit
expressions for $b_0(\xi)$, $b_1(\xi)$, $c_0(\xi)$ and $c_1(\xi)$ from
Theorem \ref{th:asympt recurrence coefficients}. Note that $w = -\xi$
is always a double zero of \eqref{eq: s2 hat factorization}.
Thus $0$ is always a branch point and in fact
an endpoint of one of the sets $\widehat \Gamma_j(\xi)$, as will follow
from the analysis in the next subsection.

\subsection{Analysis of the symbol $\widehat s_2$ in the two-cut case}

The remainder of this section is devoted to the proof of the following theorem,
which is the two-cut case version of Theorem \ref{th: gamma 1cc}.

\begin{theorem} \label{th: gamma 2cc}
Fix $t< \tau^2$ and $0 < \xi < \xicr$. Then we have that
\[ \widehat \Gamma_1(\xi) \subset \mathbb R^+, \qquad \widehat \Gamma_2(\xi) \subset \R^-,
    \qquad \text{and} \qquad \widehat \Gamma_3(\xi) \subset \R^+. \]
More precisely, there exist $\widehat \alpha(\xi) > \widehat \beta(\xi) \geq 0$,
$\widehat \gamma(\xi) \leq 0$, and $\widehat \delta(\xi) \geq 0$  such that
\begin{equation}
\begin{aligned}
\widehat \Gamma_1(\xi) & = [\widehat \beta(\xi), \widehat \alpha(\xi)],\\
\widehat \Gamma_2(\xi) & = (-\infty, \widehat\gamma(\xi)], \\
\widehat \Gamma_3(\xi) & = [\widehat \delta(\xi), +\infty).
 \end{aligned} \label{eq: hatGamma 2cc}
\end{equation}
In addition, we have for every fixed $t < \tau^2$
\begin{itemize}
\item[\rm (a)] $\xi \mapsto \widehat \alpha(\xi)$ is strictly increasing for $0 < \xi < \xicr$ with
\[ \lim_{\xi \to 0 + } \widehat \alpha(\xi) = \tau^2 (\tau^2-t),
    \quad \textrm{and }
    \lim_{\xi \to \xicr-} \widehat \alpha(\xi) = \lim_{\xi \to \xicr+} \alpha(\xi)^2; \]
\item[\rm (b)] $\widehat \beta(\xi) = 0$ if and only if $t < - \tau^2$ and $- t \tau^2 \leq \xi < \xicr$.
    Otherwise
     $\xi \mapsto \widehat \beta(\xi)$ is positive and strictly decreasing with
\begin{align*}
    \lim_{\xi \to 0+} \widehat \beta(\xi)  & = \tau^2 (\tau^2-t),  \\
    \lim_{\xi \to \xicr-} \widehat \beta(\xi) & = 0,  && \text{for } -\tau^2 \leq t < \tau^2, \\
    \lim_{\xi \to -t\tau^2-} \widehat \beta(\xi) & = 0, && \text{for } t \leq -\tau^2;
\end{align*}
\item[\rm (c)]
    $\widehat \gamma(\xi)=0$ if and only if $t < 0$ and $ \xi \leq - t \tau^2$. Otherwise
    $\xi \mapsto \widehat \gamma(\xi)$ is negative and strictly decreasing with
    \begin{align*}
    \lim_{\xi \to 0+} \widehat \gamma(\xi) & = -\frac{4 t^3}{27 \tau^2} = -(y^*)^2,
     && \text{for } 0<t<\tau^2, \\
    \lim_{\xi \to -t\tau^2} \widehat \gamma(\xi) & = 0, && \text{for }  t < 0,  \\
    \lim_{\xi \to \xicr-} \widehat \gamma(\xi) & = - \lim_{\xi \to \xicr+} \gamma(\xi)^2;
    \end{align*}
\item[\rm (d)]
    $\widehat \delta(\xi) =0$ if and only if $-\tau^2 < t < \tau^2$ and $-t \tau^2 \leq \xi < \xicr$.
    Otherwise $\xi \mapsto \widehat \delta(\xi)$ is positive and strictly decreasing with
    \begin{align*}
    \lim_{\xi \to 0+} \widehat \delta(\xi) & = \frac{4 (-t)^3}{27} = (x^*)^2, && \text{for } t < 0, \\
    \lim_{\xi \to \xicr-} \widehat \delta(\xi) & = 0, && \text{for } t \leq -\tau^2, \\
    \lim_{\xi \to -t \tau^2-} \widehat \delta(\xi) & = 0, && \text{for } -\tau^2 \leq t<0.
    \end{align*}
\end{itemize}
\end{theorem}

\begin{figure}[t]
\centering
\begin{tikzpicture}
\draw[->] (-4,0)--(4,0) node[right]{$t$};
\draw[->] (0,-0.1)node[below]{0}--(0,8) node[above]{$\xi$};
\draw (-3.5,7.8) node[right]{$\xi=\frac{(\tau^2-t)^2}{4}$};
\clip (-7,-1) rectangle (7,8);
\draw[very thick] (2,0) parabola (-4,9);
\draw[very thick] (0,0)--(-4,8);
\draw[dashed] (-4,7.8) node[left]{$\xi=-t \tau^2$};
\draw(0.2,0.4) node[draw,fill=white]{$C_{2c}$} ;
\draw(-3,3) node[draw,fill=white]{$C_{2b}$} ;
\draw(-3.5,6) node[draw,fill=white](nummer I){$C_{2a}$};
\draw (-3.7,7.9) node (gebied I){};
\draw [->] (nummer I) to [out=120,in=240] (gebied I);
\draw[help lines] (-2,0)--(-2,4)--(0,4);
\filldraw (-2,4) circle (1pt);
\draw (-0.1,4)--(0.1,4) node[right]{$\tau^4$};
\draw (2,0.1)--(2,-0.1) node[below]{$\tau^2$};
\draw (-2,0.1)--(-2,-0.1) node[below]{$-\tau^2$};
\draw (-0.1,1)-- (0.1,1) node[right] {$\frac{\tau^4}{4}$};
\draw (0,5) node[fill=white] {\large{$C_1$}};
\end{tikzpicture}
\caption{Regions and subregions of the $t\xi$-phase diagram.
The critical semi-parabola separates the one-cut case
region $C_1$ from the two-cut case region $C_2$. The critical
ray divides $C_2$ into the subregions $C_{2a}$, $C_{2b}$ and $C_{2c}$.}
\label{fig: subregions}
\end{figure}
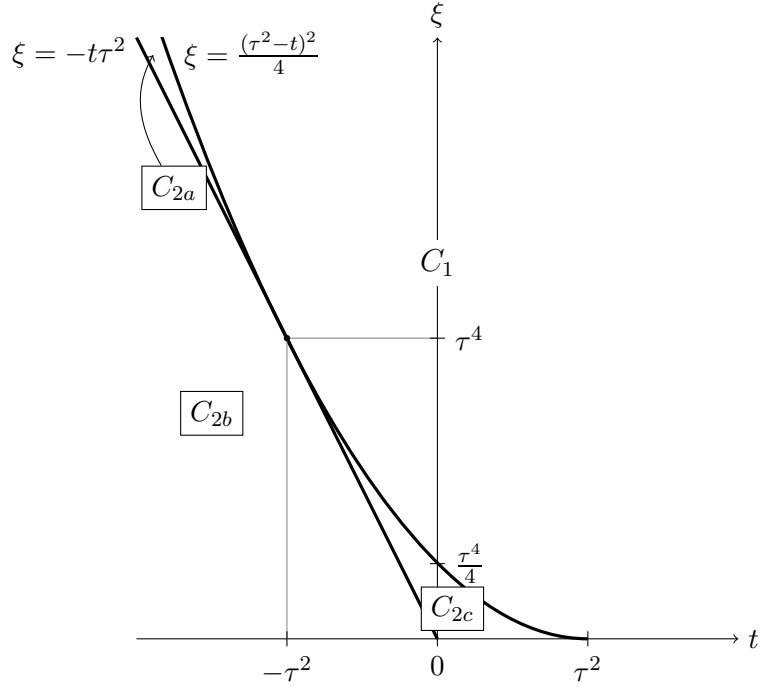

Theorem \ref{th: gamma 2cc} indicates a refinement of the $t\xi$-phase diagram.
We divide the two-cut case region $C_2$ into three subregions $C_{2a}$, $C_{2b}$ and $C_{2c}$,
depending on whether $\widehat{\beta}(\xi)$, $\widehat{\gamma}(\xi)$ and $\widehat{\delta}(\xi)$
are zero, or not. The regions are separated by the critical ray
\begin{align*}
  \xi = -t \tau^2, \qquad t < 0,
\end{align*}
which is tangent to the critical semi-parabola. The three regions are
\begin{align*}
C_{2a} : & \qquad t<-\tau^2, \quad -t\tau^2 <\xi < \xicr, \\
C_{2b} : & \qquad t<0, \quad 0<\xi<-t \tau^2, \\
C_{2c} : & \qquad -\tau^2<t<\tau^2, \quad \max(-t\tau^2,0)<\xi<\xicr,
\end{align*}
see Figure \ref{fig: subregions}. Then, according to
Theorem~\ref{th: gamma 2cc}, we have the following for $\xi < \xicr$,
\begin{itemize}
\item if $(t,\xi) \in C_{2a}$ then $\widehat{\beta}(\xi) = 0$, $\widehat{\gamma}(\xi) < 0$ and $\widehat{\delta}(\xi) > 0$;
\item if $(t,\xi) \in C_{2b}$ then $\widehat{\beta}(\xi) > 0$, $\widehat{\gamma}(\xi) = 0$ and $\widehat{\delta}(\xi) > 0$;
\item if $(t,\xi) \in C_{2c}$ then $\widehat{\beta}(\xi) > 0$, $\widehat{\gamma}(\xi) < 0$ and $\widehat{\delta}(\xi) = 0$.
\end{itemize}

The proof of Theorem \ref{th: gamma 2cc} is in the following subsection.
The  approach is similar to the one used throughout the previous section,
but there are certain complications.

\subsection{Proof of Theorem \ref{th: gamma 2cc}}

\subsubsection{The zeros of $\widehat{s}_2$ and graphs}
The symbol $\widehat{s}_2$ has a double zero in $-\xi$, see \eqref{eq: s2 hat factorization}.
The two remaining zeros are also negative. We order them such that
$x_1 \leq x_2<0$. The location of $(t,\xi)$ in the phase diagram
determines the way the zeros are ordered, see Figure \ref{fig: subregions}.

\begin{lemma}\label{lemma: subregions 2cc}
The zeros of $\widehat{s}_2(w;\xi)$ are ordered as follows:
\begin{align*}
-\xi<x_1<x_2,  & \qquad \text{if } (t,\xi) \in C_{2a},  \\
 x_1<-\xi<x_2, & \qquad \text{if } (t,\xi) \in C_{2b}, \\
 x_1<x_2<-\xi, & \qquad \text{if } (t,\xi) \in C_{2c},  \\
 x_1=-\xi<x_2, & \qquad \text{if } \xi=-t \tau^2 \textrm{ and } t<-\tau^2,  \\
 x_1<x_2=-\xi, & \qquad \text{if } \xi=-t \tau^2 \textrm{ and } -\tau^2<t<0,  \\
 x_1=x_2=-\xi, & \qquad \text{if } (t,\xi) =(-\tau^2,\tau^4).
\end{align*}
\end{lemma}
\begin{proof}
The proof of this lemma is straightforward.
\end{proof}

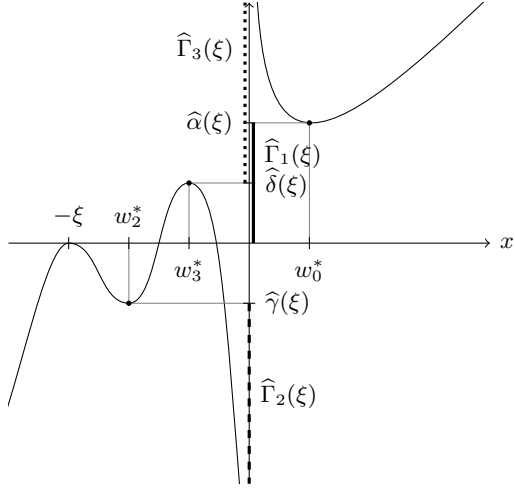
\begin{figure}[tb]
\centering
\begin{tikzpicture}[scale=0.8]
\draw[->] (-4,0) -- (4,0) node[font=\footnotesize, right]{$x$};
\draw[->] (0,-4)-- (0,4);
\clip (-4,-4)rectangle (4,4);
\draw (0.1,4.8) .. controls (0.2,1) and (1,1.1) .. (4.8,4.9);
\draw (-0.1,-4.8) .. controls (-0.4,-1) and (-0.5,1) .. (-1,1);
\draw (-1,1) .. controls (-1.5,1) and (-1.5,-1) .. (-2,-1);
\draw (-2,-1) .. controls (-2.5,-1) and (-2.5,0) .. (-3,0);
\draw (-3,0) .. controls (-3.5,0) and (-4,-3.7) .. (-4.9,-4.8);
\draw[help lines] (1,0) -- (1,2);
\draw (1,0.1)--(1,-0.1) node[font=\footnotesize,below]{$w_0^*$};
\draw[help lines] (0,2)--(1,2);
\draw (0.1,2)--(-0.1,2) node[font=\footnotesize,left]{$\widehat \alpha(\xi)$};
\draw (-3,-0.1)--(-3,0.1) node[font=\footnotesize,above]{$-\xi$};
\draw[help lines] (-2,-1)--(-2,0);
\draw (-2,-0.1)--(-2,0.1) node[font=\footnotesize,above]{$w_2^*$};
\draw[help lines] (-2,-1)--(0,-1);
\draw (-0.1,-1)--(0.1,-1) node[font=\footnotesize,right]{$\widehat \gamma(\xi)$};
\draw[help lines] (-1,1)--(-1,0);
\draw (-1,0.1)--(-1,-0.1) node[font=\footnotesize,below]{$w_3^*$};
\draw[help lines] (-1,1)--(0,1);
\draw (-0.1,1)--(0.1,1) node[font=\footnotesize,right]{$\widehat \delta(\xi)$};
\draw[very thick] (0.07,0)--(0.07,2) node[near end, right, font=\footnotesize]{$\widehat \Gamma_1(\xi)$};
\draw[very thick,dashed] (0,-1)--(0,-4) node[midway, right, font=\footnotesize]{$\widehat \Gamma_2(\xi)$};
\draw[very thick,dotted] (-0.07,1)--(-0.07,4) node[near end, left, font=\footnotesize]{$\widehat \Gamma_3(\xi)$};
\filldraw (1,2) circle (1pt);
\filldraw (-2,-1) circle (1pt);
\filldraw (-1,1) circle (1pt);
\end{tikzpicture}
\caption{The graph of $x \mapsto \widehat s_2(x;\xi)$ for $(t,\xi) \in C_{2a}$.}
\label{fig: s2R1}
\end{figure}

The three Figures \ref{fig: s2R1}, \ref{fig: s2R2}, and \ref{fig: s2R3} show
sketches of the graph of $\widehat{s}_2$ for $(t,\xi)$
belonging to $C_{2a},$ $C_{2b}$, and $C_{2c}$, respectively.
The graphs have the following properties:
\begin{itemize}
\item $\widehat{s}_2$ has four negative zeros, counted with multiplicity. $-\xi$ is a double zero.
\item $\widehat{s}_2$ attains a local minimum $\widehat \alpha (\xi)>0$ in a point $w_0^*>0$.
\item $\widehat{s}_2$ attains three local extrema $\widehat \beta (\xi)\geq0$, $\widehat \gamma (\xi)\leq0$, and $\widehat \delta (\xi) \geq 0$ in the respective points $w_1^*<w_2^*<w_3^*<0$. $\widehat \beta(\xi)$ and $\widehat\delta(\xi)$ are local maxima. $\widehat \gamma(\xi)$ is a local minimum. Only the extremum attained at the double zero $-\xi$ is zero.
\item If  $(t,\xi) \in C_{2a}$, then $w_1^*=-\xi$, so that $\widehat\beta(\xi)=0$, $\widehat \gamma(\xi)<0$, and $\widehat \delta(\xi)>0$.
\item If  $(t,\xi) \in C_{2b}$, then $w_2^*=-\xi$, so that $\widehat\beta(\xi)>0$, $\widehat \gamma(\xi)=0$, and $\widehat \delta(\xi)>0$.
\item If  $(t,\xi) \in C_{2c}$, then $w_3^*=-\xi$, so that $\widehat\beta(\xi)>0$, $\widehat \gamma(\xi)<0$, and $\widehat \delta(\xi)=0$.
\end{itemize}

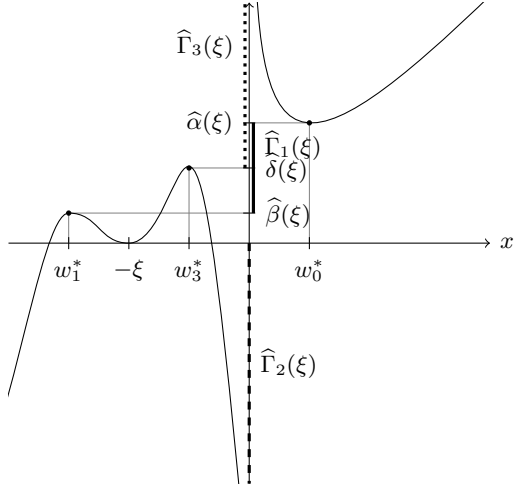
\begin{figure}[tb]
\centering
\begin{tikzpicture}[scale=0.8]
\draw[->] (-4,0) -- (4,0) node[font=\footnotesize,right]{$x$};
\draw[->] (0,-4)-- (0,4);
\clip (-4,-4)rectangle (4,4);
\draw[help lines] (1,0) -- (1,2);
\draw (1,0.1)--(1,-0.1) node[font=\footnotesize,below]{$w_0^*$};
\draw[help lines] (0,2)--(1,2);
\draw (0.1,2)--(-0.1,2) node[font=\footnotesize,left]{$\widehat \alpha(\xi)$};
\draw (-2,0.1)--(-2,-0.1) node[font=\footnotesize,below]{$-\xi$};
\draw[help lines] (-1,1.25)--(0,1.25);
\draw (-0.1,1.25)--(0.1,1.25) node[font=\footnotesize,right]{$\widehat \delta(\xi)$};
\draw[help lines] (-1,1.25)--(-1,0);
\draw (-1,0.1)--(-1,-0.1) node[font=\footnotesize,below]{$w_3^*$};
\draw[help lines] (-3,0.5)--(-3,0);
\draw (-3,0.1)--(-3,-0.1) node[font=\footnotesize,below]{$w_1^*$};
\draw[help lines] (-3,0.5)--(0,0.5);
\draw (-0.1,0.5)--(0.1,0.5) node[font=\footnotesize,right]{$\widehat \beta(\xi)$};
\draw (0.1,4.8) .. controls (0.2,1) and (1,1.1) .. (4.8,4.9);
\draw (-0.1,-4.8) .. controls (-1,5) and (-1,-0) .. (-2,0);
\draw (-3,0.5) .. controls (-2.5,0.5) and (-2.5,0) .. (-2,0);
\draw (-3,0.5) .. controls (-3.5,0.5) and (-4,-3.7) .. (-4.9,-4.8);
\draw[very thick] (0.07,0.5)--(0.07,2) node[near end, right, font=\footnotesize]{$\widehat \Gamma_1(\xi)$};
\draw[very thick,dashed] (0,0)--(0,-4) node[midway, right, font=\footnotesize]{$\widehat \Gamma_2(\xi)$};
\draw[very thick,dotted] (-0.07,1.25)--(-0.07,4) node[near end, left, font=\footnotesize]{$\widehat \Gamma_3(\xi)$};
\filldraw (1,2) circle (1pt);
\filldraw (-1,1.25) circle (1pt);
\filldraw (-3,0.5) circle (1pt);
\end{tikzpicture}
\caption{The graph of $x \mapsto \widehat s_2(x;\xi)$ for $(t,\xi) \in C_{2b}$.
The figure shows a situation where $\widehat{\beta}(\xi) < \widehat{\delta}(\xi) < \widehat{\alpha}(\xi)$.
It is also possible that $\widehat{\beta}(\xi) \geq \widehat{\delta}(\xi)$
or $\widehat{\delta}(\xi) \geq \widehat{\alpha}(\xi)$.}
\label{fig: s2R2}
\end{figure}

\begin{figure}[tb]
\centering
\begin{tikzpicture}[scale=0.8]
\draw[->] (-4,0) -- (4,0) node[font=\footnotesize,right]{$x$};
\draw[->] (0,-4)-- (0,4);
\clip (-4,-4)rectangle (4,4);
\draw[help lines] (1,0) -- (1,2);
\draw (1,0.1)--(1,-0.1) node[font=\footnotesize,below]{$w_0^*$};
\draw[help lines] (0,2)--(1,2);
\draw (0.1,2)--(-0.1,2) node[font=\footnotesize,left]{$\widehat \alpha(\xi)$};
\draw (-1,-0.1)--(-1,0.1) node[font=\footnotesize,above]{$-\xi$};
\draw[help lines] (-2,-1)--(-2,0);
\draw (-2,-0.1)--(-2,0.1) node[font=\footnotesize,above]{$w_2^*$};
\draw[help lines] (-2,-1)--(0,-1);
\draw (-0.1,-1)--(0.1,-1) node[font=\footnotesize,right]{$\widehat \gamma(\xi)$};
\draw[help lines] (-3,1)--(-3,0);
\draw (-3,0.1)--(-3,-0.1) node[font=\footnotesize,below]{$w_1^*$};
\draw[help lines] (-3,1)--(0,1);
\draw (-0.1,1)--(0.1,1) node[font=\footnotesize,right]{$\widehat \beta(\xi)$};
\draw (0.1,4.8) .. controls (0.2,1) and (1,1.1) .. (4.8,4.9);
\draw (-0.1,-4.8) .. controls (-0.4,-1) and (-0.5,0) .. (-1,0);
\draw (-1,0) .. controls (-1.5,0) and (-1.5,-1) .. (-2,-1);
\draw (-2,-1) .. controls (-2.5,-1) and (-2.5,1) .. (-3,1);
\draw (-3,1) .. controls (-3.5,1) and (-4,-3.7) .. (-4.9,-4.8);
\draw[very thick] (0.07,1)--(0.07,2) node[near end, right, font=\footnotesize]{$\widehat \Gamma_1(\xi)$};
\draw[very thick,dashed] (0,-1)--(0,-4) node[midway, right, font=\footnotesize]{$\widehat \Gamma_2(\xi)$};
\draw[very thick,dotted] (-0.07,0)--(-0.07,4) node[near end, left, font=\footnotesize]{$\widehat \Gamma_3(\xi)$};
\filldraw (1,2) circle (1pt);
\filldraw (-2,-1) circle (1pt);
\filldraw (-3,1) circle (1pt);
\end{tikzpicture}
\caption{The graph of $x \mapsto \widehat s_2(x;\xi)$ for $(t,\xi) \in C_{2c}$.}
\label{fig: s2R3}
\end{figure}
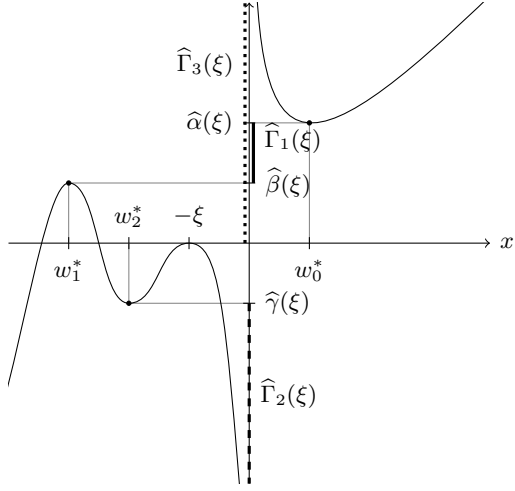

\subsubsection{Auxiliary lemmas}

The following two lemmas are the analogues of Lemmas \ref{lemma: axes 1cc}
and \ref{lemma: gamma 1cc}.

\begin{lemma}\label{lemma: axes 2cc}
Let $\xi  \in (0, \xicr)$. Assume that $w_a,w_b \in \mathbb C$ are such that
$w_a \neq w_b$, $|w_a|=|w_b|$, and $\widehat{s}_2(w_a;\xi)=\widehat{s}_2(w_b;\xi)=z$.
Then, $z \in \mathbb R$.
\end{lemma}

\begin{proof}
The complex numbers $w_a$ and $w_b$ lie on a circle of radius
$\rho=|w_a|=|w_b|$ centered at the origin of the complex plane.
We can factorize $\widehat{s}_2(w;\xi)$ as
\[
\widehat{s}_2(w;\xi)=\frac{(w+\xi)^2(w-x_1)(w-x_2)}{w^3},
\]
where $x_1,x_2<0$. Thus,
\[
    |\widehat{s}_2(w,\xi)|=\frac{\dist(w,-\xi)^2\dist(w,x_1)\dist(w,x_2)}{\rho^3}, \qquad \textrm{if } |w|=\rho.
\]
Since $x_1,x_2,-\xi<0$, it follows that
\[
    [-\pi,\pi]\to \mathbb R: \theta \mapsto \frac{\dist(\rho e^{i\theta},-\xi)^2\dist(\rho e^{i\theta},x_1)\dist(\rho e^{i\theta},x_2)}{\rho^3},
\]
is an even function that is strictly decreasing as $\theta$ increases from $0$ to $\pi$.
Thus the equality for $\theta = \theta_a$, and $\theta = \theta_b$
can only occur if $\theta_b=\pm \theta_a$.
Then, it follows from the assumptions of the lemma that $w_b=\overline{w_a}$ and
\[
    z=\widehat{s}_2(w_b;\xi)=\widehat{s}_2(\overline{w_a},\xi)=\overline{(\widehat{s}_2(w_a,\xi)}=\overline {z}.
\]
Therefore, $z \in \R$.
\end{proof}

\begin{lemma} \label{lemma: gamma 2cc}
Let $\xi \in (0,\xicr)$. Then, $\widehat \Gamma_j(\xi) \subset \R$ for $j=1,2,3$.
Moreover, $\widehat \Gamma_1(\xi) \cap \widehat \Gamma_2(\xi) = \widehat \Gamma_2(\xi) \cap
\widehat \Gamma_3(\xi)= \emptyset$.
\end{lemma}
\begin{proof}
The proof is analogous to that of Lemma \ref{lemma: gamma 1cc}.
\end{proof}

\subsubsection{Transformed symbol}
For convenience, let us again introduce a transformed symbol $\widehat S_2$
\begin{equation}
\widehat S_2(W;\xi)=\widehat{s}_2(\xi W;\xi)=\xi \frac{(W+1)^2}{W}-t\tau^2\frac{(W+1)^2}{W^2}+\tau^4\frac{(W+1)^3}{W^3},
\label{eq: S2}
\end{equation}
see \eqref{eq: s2 hat factorization}.
Note that $\widehat S_2$ depends on $\xi$ in a simple way.
The branch points of the transformed symbol $\widehat S_2$ coincide with
the branch points of the symbol $\widehat{s}_2$. To see this, define
$W_j^*(\xi)=w_j^*(\xi)/\xi$ for $j=0,1,2,3$. Then,
\[ \frac{\partial }{\partial W} \widehat S_2(W_j^*(\xi);\xi)=0, \qquad \text{for } j=0,1,2,3, \]
and
\begin{equation} \label{eq: branch points in S2}
\begin{aligned}
\widehat \alpha(\xi)&=\widehat S_2(W_0^*(\xi);\xi), & \widehat \beta(\xi) &= \widehat S_2(W_1^*(\xi);\xi), \\
\widehat \gamma(\xi)&=\widehat S_2(W_2^*(\xi);\xi), & \widehat \delta(\xi) &= \widehat S_2(W_3^*(\xi);\xi).
\end{aligned}
\end{equation}

\subsubsection{Proof of Theorem \ref{th: gamma 2cc}}

\begin{proof} %[Proof of Theorem \ref{th: gamma 2cc}.]
We prove the first part of the theorem only for the case that $(t,\xi) \in C_{2a}$.
The other cases can be treated similarly. Figure \ref{fig: s2R1} shows the graph of the symbol $\widehat s_2$
for the case $(t, \xi) \in C_{2a}$.

For $z \in (\widehat \gamma(\xi),0)$, the equation
$\widehat s_2(x;\xi)=z$ has four different negative solutions.
Therefore, $(\widehat \gamma(\xi),0)\cap \widehat \Gamma_j(\xi) = \emptyset$, $j=1,2,3$.
If $z=\widehat \gamma(\xi)$ the equation has the double solution $w_2^*$,
a negative solution with a greater modulus and a negative solution with a smaller modulus.
We conclude that $\widehat \gamma(\xi) \in \widehat \Gamma_2(\xi)$. Next, take $z <\widehat \gamma(\xi)$.
Then, the equation has two negative solutions and one pair of complex conjugate solutions.
Using the continuity argument (which was also used in the proof of Theorem \ref{th: gamma 1cc})
we obtain that the modulus of the complex conjugate solutions lies
between the moduli of the negative solutions.
Therefore, $(-\infty,\widehat \gamma(\xi)] \subset \widehat \Gamma_2(\xi)$.

Now focus on $z \geq 0$ with the extra assumption that $\widehat \alpha(\xi)>\widehat \delta(\xi)$.
We can make a similar reasoning if $\widehat \alpha(\xi) \leq \widehat \delta(\xi)$.
If $z=0$ the equation has a double solution $-\xi$ and two different negative solutions
with a smaller modulus. Therefore, $0$ belongs to $\widehat \Gamma_1(\xi)$.
If $z \in (0,\widehat \delta(\xi))$ we find two different negative solutions
and one pair of complex conjugate solutions. The continuity argument guarantees
that the modulus of the complex solutions is the greatest, so
that $[0,\widehat \delta(\xi)) \subset \widehat \Gamma_1$.
Proceeding in the same way, we obtain
$[\widehat \delta(\xi),\widehat \alpha(\xi)] \subset \widehat \Gamma_1(\xi) \cap \widehat \Gamma_3(\xi)$.

For $z > \widehat \alpha(\xi)$ the equation has two different positive
solutions and one pair of complex conjugate solutions. By the
continuity argument we obtain $(\widehat \alpha(\xi),\infty) \subset
\widehat \Gamma_1(\xi)$ or $(\widehat \alpha(\xi),\infty)\subset \widehat \Gamma_3(\xi)$. Since
$\widehat \Gamma_1(\xi)$ is a compact set, see \cite{DK1}, only the second
inclusion holds.

Collecting this information we obtain
\[
\widehat \Gamma_1(\xi)=[0,\widehat \alpha(\xi)], \quad
\widehat \Gamma_2(\xi)=(-\infty,\widehat \gamma(\xi)], \quad \text{ and }
\widehat \Gamma_3(\xi)=[\widehat \delta(\xi),\infty).
\]
so that \eqref{eq: hatGamma 2cc} is proved under the assumption that $(t,\xi) \in C_{2a}$
and $\widehat \alpha(\xi) > \widehat \delta(\xi)$. One can prove all other cases in a similar way.

Next, let us study the behavior of  $\widehat\alpha(\xi),$ $\widehat\beta(\xi)$, $\widehat\gamma(\xi)$,
and $\widehat\delta(\xi)$ as $\xi$ increases.
Since all branch points are of the form $\widehat S_2(W_j^*(\xi);\xi)$ for some $j = 0,1,2,3$,
see \eqref{eq: S2} and \eqref{eq: branch points in S2}, we are interested in the sign of
\begin{align}
\frac{\textrm d}{\textrm d \xi}\widehat S_2(W_j^*(\xi);\xi) & = \frac{\partial \widehat S_2}{\partial W}(W_j^*(\xi);\xi)
\frac{\textrm d W_j^*(\xi)}{\textrm d \xi}+\frac{\partial \widehat S_2}{\partial \xi}(W_j^*(\xi);\xi) \nonumber \\
 & = \frac{\partial \widehat S_2}{\partial \xi}(W_j^*(\xi);\xi) \nonumber \\
 & = \frac{(W_j^*(\xi)+1)^2}{W_j^*(\xi)}, \label{eq: key}
\end{align}
where the last equality holds because of \eqref{eq: S2}.

Since $W_0^*(\xi)=w_0^*(\xi)/\xi>0$, we obtain from \eqref{eq: key}
that $\xi \mapsto \widehat \alpha(\xi)$ is a strictly increasing function.
Putting $\xi = -1$ in \eqref{eq: key}, we also get that
that $0 = \widehat S_2(-1;\xi)$ remains zero for fixed $t$ and $(t,\xi)$ in a fixed subregion $C_{2a},C_{2b},$ or $C_{2c}$.
The other two branch points can be written as $\widehat s_2(w_j^*(\xi);\xi)$
with $w_j^*(\xi)<0$ and $w_j^*(\xi) \neq -\xi$.
In terms of the transformed symbol they are $\widehat S_2(W_j^*(\xi);\xi)$ with
$W_j^*(\xi)<0$ and $W_j^*(\xi) \neq -1$. Equation \eqref{eq: key} then implies
that these branch points are strictly decreasing functions of $\xi$.

Let us now concentrate on the behavior of the branch points as
$\xi \to \xicr-$ for fixed $t$, so that $(t,\xi)$ approaches the
critical semi-parabola. This behavior is a straightforward corollary of the following claim
\begin{equation}
\widehat s_2(w^2;\xicr)=  s_1(w;\xicr)^2, \qquad w \in \C \setminus \{0\}.
\label{eq: s1 s2}
\end{equation}
Let us prove this claim. Using \eqref{eq: s1}, \eqref{eq: s2 hat}, and \eqref{eq: a a0 a1}, we compute
\begin{align*}
\lim_{\xi \to \xicr-}\widehat{s}_2(w^2;\xi) &=  w^2\left(1+\frac{b_0(\xicr)}{w^2}+ \frac{c_0(\xicr)}{w^4}\right) \left(1+\frac{b_1(\xicr)}{w^2}+ \frac{c_1(\xicr)}{w^4}\right)\\
&= \left( w+\frac{b(\xicr)}{w}+\frac{c(\xicr)}{w^3} \right)^2 \\
&= \lim_{\xi \to \xicr+} s_1(w;\xi)^2.
\end{align*}

The calculation of the limits of the branch points if $\xi \to 0+$ can be done in the same way as in the proof of Theorem \ref{th: gamma 1cc} and is left to the reader.
\end{proof}

\subsection{Reformulation of results in the two-cut case} \label{subsec: reformulation 2cc}

In what follows it will be convenient to undo the doubling of the recurrence
and consider instead of the functions $\widehat{w}_j(z; \xi)$,
the sets $\widehat{\Gamma}_j(\xi)$, and the measures $\widehat \mu_j^{\xi}$
also for $\xi \in (0,\xicr)$,
\begin{equation}
    w_j(z;\xi)=\widehat w_j(z^2,\xi), \qquad j=1,2,3,4,
    \label{eq: wj 2cc}
\end{equation}
\[
\Gamma_j(\xi) = \{z \in \C \mid z^2 \in \widehat \Gamma_j(\xi)\}, \quad j=1,2,3,
\]
and
\begin{equation}
\begin{aligned}
    \ud \mu_1^\xi(x) &=|x|\frac{\mathrm d \widehat{\mu}_1^\xi}{\mathrm d x}(x^2) \ud x,
    && x \in \Gamma_1(\xi), \\
    \ud \mu_2^\xi(z) &= | z|\frac{\mathrm d \widehat{\mu}_2^\xi}{\mathrm d x}(z^2) |\ud z|,
    && z \in \Gamma_2(\xi), \\
    \ud  \mu_3^\xi(x) &=|x|\frac{\mathrm d \widehat{\mu}_3^\xi}{\mathrm d x}(x^2) \ud x,
    && x \in \Gamma_3(\xi).
\end{aligned} \label{eq: muj 2cc}
\end{equation}

Then, with these definitions
\begin{equation*} \label{eq: Gammaj3}
    \Gamma_j(\xi) = \{z \in \mathbb C \mid |w_j(z;\xi)|=|w_{j+1}(z;\xi)|\},\quad j=1,2,3;
\end{equation*}
and
\begin{equation}
    \ud \mu_j^{\xi} (z) =
    \frac12 \cdot \frac{1}{2\pi i}
    \left( \frac{{w_j'}_-(z;\xi)}{{w_j}_-(z;\xi)} - \frac{{w_j'}_+(z;\xi)}{{w_j}_+(z;\xi)}\right) \ud z,
    \quad z \in \Gamma_j(\xi), \quad j=1,2,3.
\label{eq: muj 2cc bis}
\end{equation}
One can also check that the $\mu_j^{\xi}$ are measures on $\Gamma_j(\xi)$ with total masses
\[ \int \ud \mu_1^{\xi} = 1, \quad \int \ud \mu_2^{\xi} = \frac23, \quad \text{and } \int \ud \mu_3^{\xi} = \frac13. \]

By Theorem \ref{th: gamma 2cc} we have
\begin{align*}
 \Gamma_1(\xi) & = [-\alpha(\xi),-\beta(\xi)]\cup[\beta(\xi), \alpha(\xi)],\nonumber \\
 \Gamma_2(\xi) & = (-i\infty, - i\gamma(\xi)]\cup[i\gamma(\xi),+i\infty),\nonumber \\
 \Gamma_3(\xi) & = (-\infty,-\delta(\xi)]\cup[\delta(\xi), +\infty). \label{eq: Gamma 2cc}
\end{align*}
with
\begin{equation*}
    \alpha(\xi)=\sqrt{\widehat \alpha(\xi)}, \quad \beta(\xi)=\sqrt{\widehat \beta(\xi)},
    \quad \gamma(\xi)=\sqrt{-\widehat \gamma(\xi)}, \quad \text{ and } \quad \delta(\xi)=\sqrt{\widehat \delta(\xi)}.
    \label{eq: branch points hat}
\end{equation*}
Theorem \ref{th: gamma 2cc} also shows that for every fixed $t < \tau^2$
\begin{itemize}
\item[\rm (a)] $\xi \mapsto \alpha(\xi)$ is strictly increasing for $0 < \xi < \xicr$ with
\[ \lim_{\xi \to 0 + } \alpha(\xi) = \tau \sqrt{\tau^2-t},
    \quad \textrm{and }
    \lim_{\xi \to \xicr-} \alpha(\xi) = \lim_{\xi \to \xicr+} \alpha(\xi), \]
\item[\rm (b)] $\beta(\xi) = 0$ if and only if $t < - \tau^2$ and $- t \tau^2 \leq \xi < \xicr$. Otherwise
     $\xi \mapsto \beta(\xi)$ is positive and strictly decreasing with
\begin{align*}
    \lim_{\xi \to 0+} \beta(\xi)  & = \tau \sqrt{\tau^2-t},  \\
    \lim_{\xi \to \xicr-} \beta(\xi) & = 0,  && \text{for } -\tau^2 \leq t < \tau^2, \\
    \lim_{\xi \to -t\tau^2-} \beta(\xi) & = 0, && \text{for } t \leq -\tau^2,
\end{align*}
\item[\rm (c)]
    $\gamma(\xi)=0$ if and only if $t < 0$ and $ 0<\xi \leq - t \tau^2$. Otherwise
    $\xi \mapsto \gamma(\xi)$ is positive and strictly increasing with
    \begin{align*}
    \lim_{\xi \to 0+} \gamma(\xi) & =  \frac{2 t^{3/2}}{3 \sqrt{3} \tau} = y^*,
     && \text{for } 0\leq t<\tau^2, \\
    \lim_{\xi \to -t\tau^2} \gamma(\xi) & = 0, && \text{for }  t \leq 0,  \\
    \lim_{\xi \to \xicr-} \gamma(\xi) & = \lim_{\xi \to \xicr+} \gamma(\xi)
    \end{align*}
\item[\rm (d)]
    $\delta(\xi) =0$ if and only if $-\tau^2 < t < \tau^2$ and $-t \tau^2 \leq \xi < \xicr$.
    Otherwise $\xi \mapsto \delta(\xi)$ is positive and strictly decreasing with
    \begin{align*}
    \lim_{\xi \to 0+} \delta(\xi) & = \frac{2 (-t)^{3/2}}{3 \sqrt{3}} = x^*, && \text{for } t \leq 0, \\
    \lim_{\xi \to \xicr-} \delta(\xi) & = 0, && \text{for } t \leq -\tau^2, \\
    \lim_{\xi \to -t \tau^2} \delta(\xi) & = 0, && \text{for } -\tau^2 \leq t<0.
    \end{align*}
\end{itemize}

Figure~\ref{fig: gamma 2cc} shows sketches of the sets $\Gamma_j(\xi)$, $j=1,2,3$, for $(t,\xi)$ belonging to $C_{2a}$, $C_{2b}$, and $C_{2c}$.

\begin{figure}[h]
\centering
\subfigure[]{
\begin{tikzpicture}[scale=0.35]
\draw[help lines,->] (-7,0)--(7,0) node[right]{$\mathbb R$};
\draw[help lines,->] (0,-5)--(0,5) node[above]{$i \mathbb R$};
\draw (0,0) node[below left] {0};
\draw[very thick] (-3,0.05)-- node[near start,above]{$\Gamma_1(\xi)$}(3,0.05);
\draw[dashed, very thick] (0,-5)--(0,-2)  (0,2)-- node[near start,right]{$\Gamma_2(\xi)$}(0,5);
\draw[dotted, very thick] (-7,-0.05)--(-1.5,-0.05);
\draw[dotted, very thick] (1.5,-0.05)-- node[near end,below]{$\Gamma_3(\xi)$} (7,-0.05);
\end{tikzpicture}
\label{fig: gamma C2a}
}
\subfigure[]{
\begin{tikzpicture}[scale=0.35]
\draw[help lines,->] (-7,0)--(7,0) node[right]{$\mathbb R$};
\draw[help lines,->] (0,-5)--(0,5) node[above]{$i \mathbb R$};
\draw (0,0) node[below left] {0};
\draw[very thick] (-5,0.05)-- node[near start,above]{$\Gamma_1(\xi)$}(-1,0.05) (1,0.05)--(5,0.05);
\draw[dashed, very thick] (0,-5)--(0,0)  (0,0)-- node[near end,right]{$\Gamma_2(\xi)$}(0,5);
\draw[dotted, very thick] (-7,-0.05)--(-2.5,-0.05);
\draw[dotted, very thick] (2.5,-0.05)-- node[near end,below]{$\Gamma_3(\xi)$} (7,-0.05);
\end{tikzpicture}
\label{fig: gamma C2b}}
\subfigure[]{
\begin{tikzpicture}[scale=0.35]
\draw[help lines,->] (-7,0)--(7,0) node[right]{$\mathbb R$};
\draw[help lines,->] (0,-5)--(0,5) node[above]{$i \mathbb R$};
\draw (0,0) node[below left] {0};
\draw[very thick] (-3,0.05)-- node[near start,above]{$\Gamma_1(\xi)$}(-1,0.05) (1,0.05)--(3,0.05);
\draw[dashed, very thick] (0,-5)--(0,-2)  (0,2)-- node[near start,right]{$\Gamma_2(\xi)$}(0,5);
\draw[dotted, very thick] (-7,-0.05)--(0,-0.05);
\draw[dotted, very thick] (0,-0.05)-- node[near end,below]{$\Gamma_3(\xi)$} (7,-0.05);
\end{tikzpicture}
\label{fig: gamma C2c}}  \caption{The sets $\Gamma_1(\xi)$ (plain), $\Gamma_2(\xi)$ (dashed), and $\Gamma_3(\xi)$ (dotted)
for $(t,\xi) \in C_{2a}$ (a), $(t,\xi) \in C_{2b}$ (b), and $(t,\xi) \in C_{2c}$ (c).} \label{fig: gamma 2cc}
\end{figure}
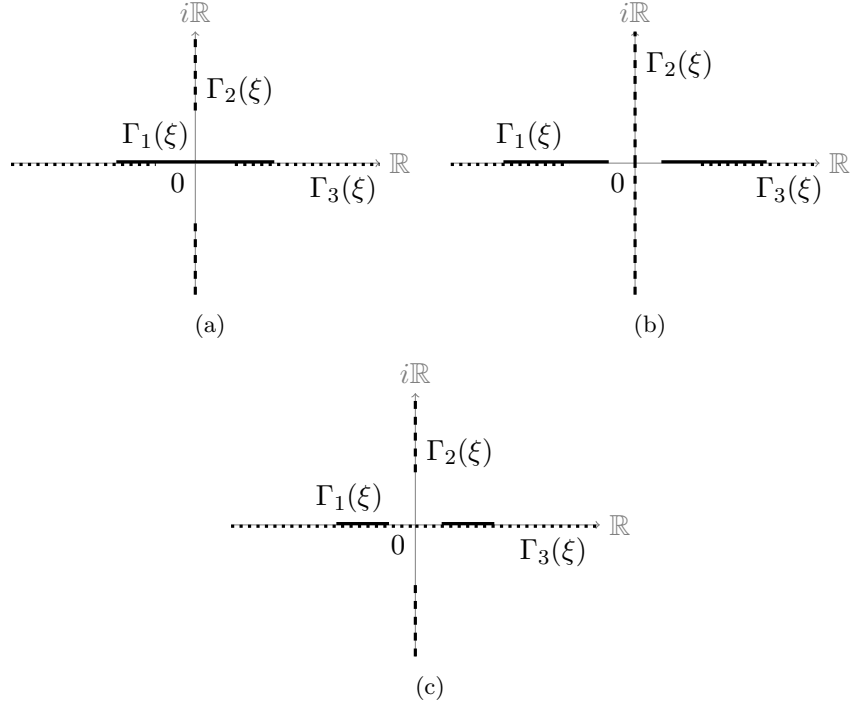

By Theorem \ref{th: DuKuToeplitz} we know that the measures
$(\widehat \mu_1^\xi,\widehat \mu_2^\xi,\widehat \mu_3^\xi)$ are
characterized by a vector equilibrium problem.
The transformed measures $(\mu_1^{\xi}, \mu_2^{\xi}, \mu_3^{\xi})$
from \eqref{eq: muj 2cc}  then also are characterized by
a vector equilibrium problem, as stated in the following theorem.

\begin{theorem} \label{th: DK 2cc}

\begin{itemize}
\item[\rm (a)] The vector of measures $(\mu_1^\xi, \mu_2^\xi, \mu_3^\xi)$ given in \eqref{eq: muj 2cc}
is the unique minimizer for the functional
\[
    E_0(\rho_1,\rho_2,\rho_3) = I(\rho_1)-I(\rho_1,\rho_2)+I(\rho_2)-I(\rho_2,\rho_3)+I(\rho_3),
\]
 among all vectors $(\rho_1, \rho_2,\rho_3)$ of positive measures with
 finite logarithmic energy $I(\rho_j) < \infty$, satisfying
\begin{itemize}
\item[\rm (i)] $\supp (\rho_j) \subset \Gamma_j(\xi)$, for $j=1,2,3$, and
\item[\rm (ii)] $\rho_1(\Gamma_1(\xi))= 1$, $\rho_2(\Gamma_2(\xi)) = 2/3$, and $\rho_3(\Gamma_3(\xi)) = 1/3$.
\end{itemize}
\item[\rm (b)]
The measures $\mu_1^\xi,\mu_2^\xi,\mu_3^\xi$ satisfy for some constant $\ell^\xi$
\begin{align}
\ell^\xi-2U^{ \mu_{1}^\xi}(z)+U^{\mu_{2}^\xi}(z) &= \frac12 \log \left| \frac{w_1(z;\xi)}{w_2(z;\xi)}\right|,  \label{eq: EL 1 2cc}\\
U^{\mu_{1}^\xi}(z)-2U^{\mu_{2}^\xi}(z)+U^{\mu_{3}^\xi}(z) &= \frac12 \log \left| \frac{w_2(z;\xi)}{w_3(z;\xi)}\right|, \label{eq: EL 2 2cc} \\
U^{\mu_{2}^\xi}(z)-2U^{\mu_{3}^\xi}(z) &= \frac12 \log \left| \frac{w_3(z;\xi)}{w_4(z;\xi)}\right|, \label{eq: EL 3 2cc}
\end{align}
\end{itemize}
\end{theorem}

\begin{proof}
It follows from \eqref{eq: logarithmic potential} and \eqref{eq: muj 2cc} that
\begin{equation*}
U^{\mu_j^\xi}(z)=\frac12 U^{\widehat \mu_j^\xi}(z^2).
\label{eq: U hat}
\end{equation*}
Applying this to the variational conditions \eqref{eq: EL 1}--\eqref{eq: EL 3} in the context of the symbol $\widehat s_2$
and using \eqref{eq: wj 2cc} establishes \eqref{eq: EL 1 2cc}--\eqref{eq: EL 3 2cc}.
\end{proof}

\section{Proof of Theorem \ref{th: zero distribution}} \label{sec: zero distribution}

We now give the proof of Theorem \ref{th: zero distribution},
which is based on Theorem \ref{th: KR}. For fixed $t$, we define the
measure $\nu_1$ as an average of the measures $\mu_1^\xi$
\begin{equation}
    \nu_1 = \int_0^1 \mu_1^\xi \ud \xi.
    \label{eq: nu1}
\end{equation}
This measure will be the limiting normalized zero distribution of the
polynomials $p_{n,n}$ as $n \to \infty$. Note that the measure
$\mu_1^\xi$ is  defined by \eqref{eq: muj} in the one-cut case (i.e., $\xi > \xicr$),
and by \eqref{eq: muj 2cc} in the two-cut case.

\subsection{Proof of Theorem \ref{th: zero distribution} for $t \geq \tau^2$}

We first prove Theorem \ref{th: zero distribution} for $t \geq \tau^2$.
This is the simplest case, since we deal with the one-cut case for
every $\xi > 0$ and we can apply Theorem \ref{th: KR} almost immediately.

\begin{proof} Let $t \geq \tau^2$.
Conditions (a) and (d) of Theorem \ref{th: KR} follow from Lemma~\ref{lemma:relation coefficients}
and Theorem \ref{th:asympt recurrence coefficients}.
The interlacing condition (b) follows from Theorem \ref{th: interlacing}~(a),
and condition (e) is contained in Theorem \ref{th: gamma 1cc}.
So in order to be able to apply Theorem \ref{th: KR}
it remains to establish condition~(c). This will be done in the following lemma.

Having verified all conditions, Theorem \ref{th: KR} can be applied
and Theorem \ref{th: zero distribution} follows for $t \geq \tau^2$.
\end{proof}

To complete the preceding proof we still have to establish the following lemma.
It states that the recurrence coefficients $b_{k,n}$ and $c_{k,n}$ remain bounded
in case $k/n$ is bounded.

\begin{lemma} \label{lemma: bounded zeros}
Let $t \in \mathbb R$. Then the two sets of recurrence coefficients
$\{ b_{k,n} \mid k+1 \leq n \}$ and $\{ c_{k,n} \mid k+1 \leq n \}$ are bounded.
\end{lemma}

\begin{proof}
By \eqref{eq: ba} and \eqref{eq: ca} it is sufficient to proof that the set
\[
\{ a_{k,n} \mid k+1 \leq n \}
\]
is bounded. From the asymptotics of the orthogonal polynomials $q_{n,n}$,
see \cite{BI1,BI2}, it follows that there exists $M>0$ such
that the zeros of the diagonal polynomials $q_{n,n}$, $n=1,2,3,\ldots$,
belong to $[-M,M]$. Then, by interlacing, see Theorem \ref{th: interlacing},
the zeros of all $q_{k,n}$ with $k \leq n$ belong to $[-M,M]$.
Note that the zeros of $q_{k,n}$ coincide with the eigenvalues of the tridiagonal Jacobi matrix
\[ J_{k,n}=
    \begin{pmatrix}
    0 & \sqrt{a_{1,n}} &&& \\
    \sqrt{a_{1,n}} & 0 & \sqrt{a_{2,n}} && \\
    & \sqrt{a_{2,n}} & 0 & \ddots & \\
    && \ddots & \ddots & \sqrt{a_{k-1,n}} \\
    &&& \sqrt{a_{k-1,n}} & 0
\end{pmatrix}.
\]
It is well-known that the eigenvalues of a real symmetric matrix $X$ interlace
with the eigenvalues of the matrix obtained from $X$ by deleting the first
row and column. Applying this $k-2$ times, we find that that the eigenvalues of
\[ \begin{pmatrix} 0 & \sqrt{a_{k-1,n}} \\ \sqrt{a_{k-1,n}} & 0 \end{pmatrix} \]
are in $[-M,M]$, which implies that $a_{k-1,n} \leq M^2$ for every $k \leq n$.
This proves the lemma.
\end{proof}

\subsection{Proof of Theorem \ref{th: zero distribution} for $t < \tau^2$}

The situation for $t < \tau^2$ is  more complicated.
For small values of $\xi$, namely $\xi < \xicr$, we are in the two-cut case,
while for $\xi > \xicr$ we deal with the one-cut case.

The two-cut case was handled in Section \ref{sec: 2cc} by doubling the recurrence relation.
This led to the symbol $\widehat{s}_2$ and related notions.
It will be convenient to double the recurrence relation also in the one-cut case.
The doubled recurrence relation is \eqref{eq: iterated rr}, and so
for $\xi > \xicr$, we define
\begin{align}
A(\xi)&=\lim_{k/n \to \xi}A_{k,n}=2b(\xi), \label{eq: A 1cc} \\
B(\xi)&=\lim_{k/n \to \xi}B_{k,n}=2c(\xi)+b(\xi)^2,  \label{eq: B 1cc}\\
C(\xi)&=\lim_{k/n \to \xi}C_{k,n}=2b(\xi)c(\xi),  \label{eq: C 1cc}\\
D(\xi)&=\lim_{k/n \to \xi}D_{k,n}=c(\xi)^2,  \label{eq: D 1cc}
\end{align}
and
\begin{equation*}
\widehat s_1 (w;\xi)=w+A(\xi)+\frac{B(\xi)}{w}+\frac{C(\xi)}{w^2}+\frac{D(\xi)}{w^3}.
\label{eq: s1 hat}
\end{equation*}
The subscript $1$ is added to remind ourselves that we are in the one-cut case.
The hat refers to the fact that the recurrence relation was doubled to obtain this symbol.

As before, we consider the algebraic equation $\widehat s_1(w;\xi)=z$ for complex $z$ and
denote its solutions by $\widehat w_j(z;\xi)$, $j=1,2,3,4$, ordered such that
\[
    |\widehat w_1(z;\xi)| \geq |\widehat w_2(z;\xi)| \geq |\widehat w_3(z;\xi)| \geq| \widehat w_4(z;\xi)|.
\]
Furthermore, we define for $\xi > \xicr$
\[
    \widehat \Gamma_1(\xi) = \{z \in \mathbb C \mid |\widehat w_1(z;\xi)| = |\widehat w_{2}(z;\xi)|\},
\]
and
\[
\ud \widehat{\mu}_1^{\xi} (z) = \frac{1}{2\pi i}
    \left( \frac{{\widehat w_{1-}'}(z)}{{\widehat w_{1-}}(z)} - \frac{{\widehat w_{1+}'}(z)}{{\widehat w_{1+}}(z)}\right) \ud z,
\]
for $z \in \widehat \Gamma_1(\xi)$. The following lemma establishes the link between these
notions and the corresponding ones related to the original symbol $s_1$.

\begin{lemma} \label{lemma: relation 1cc hat}
Let $\xi > \xicr$. Then, $z \in \Gamma_1(\xi)$ if and only if $z^2 \in \widehat \Gamma_1(\xi)$. Moreover,
\begin{equation}
\ud \mu_1^\xi(x) =|x|\frac{\mathrm d \widehat{\mu}_1^\xi}{\mathrm d x}(x^2) \ud x, \quad x \in \Gamma_1(\xi). \label{eq: mu hat 1 1cc} \\
\end{equation}
\end{lemma}
\begin{proof}
The proof of this lemma is straightforward. We omit the details.
\end{proof}

We now prove Theorem \ref{th: zero distribution} if $t<\tau^2$.

\begin{proof}[Proof of Theorem \ref{th: zero distribution} for $t< \tau^2$.]
Let $t < \tau^2$.

For even $k = 2l$, we have that $p_{k,n}$ is an  even polynomial and
we can write
\begin{equation}
    p_{k,n}(x)=r_{l,n}(x^2),
    \label{eq: p r even}
\end{equation}
where  $r_{l,n}$ is a monic polynomial of degree $l$.
Rewriting the doubled recurrence relation \ref{eq: iterated rr} in terms of the polynomials $r_{l,n}$ yields
\[
    xr_{l,n}(x) = r_{l+1,n}(x) + \tilde A_{l,n}r_{l,n}(x)+\tilde B_{l,n}r_{l-1,n}(x)
    +\tilde C_{l,n}r_{l-2,n}(x)+\tilde D_{l,n}r_{l-3,n}(x),
\]
where
\[
\tilde A_{l,n}=A_{2l,n}, \quad \tilde B_{l,n}=B_{2l,n}, \quad \tilde C_{l,n}=C_{2l,n}, \quad \text{ and } \quad \tilde D_{l,n}=D_{2l,n}.
\]
Therefore, the polynomials $r_{l,n}$ satisfy condition (a) of Theorem \ref{th: KR}.

We have for $\xi >0$
\begin{align*}
\lim_{k/n \to \xi}A_{k,n}&=A(\xi), & \lim_{k/n \to \xi}B_{k,n}&=B(\xi),\\
    \lim_{k/n \to \xi}C_{k,n}&=C(\xi), & \lim_{k/n \to \xi}D_{k,n}&=D(\xi),
\end{align*}
where $A(\xi),B(\xi),C(\xi)$, and $D(\xi)$ are defined by \eqref{eq: A 1cc}--\eqref{eq: D 1cc}
if $\xi > \xicr$ and by \eqref{eq: A 2cc}--\eqref{eq: D 2cc} if $\xi < \xicr$. It follows that
\begin{align*}
\lim_{l/n \to \xi}\tilde A_{l,n}&=A(2\xi), & \lim_{l/n \to \xi}\tilde B_{l,n}&=B(2\xi), \\
   \lim_{l/n \to \xi}\tilde C_{l,n}&=C(2\xi), & \lim_{l/n \to \xi}\tilde D_{l,n}&=D(2\xi).
\end{align*}
This establishes the condition (d) of Theorem \ref{th: KR}.
Condition (c) follows from Lemma \ref{lemma: bounded zeros} and
(b) from Theorem \ref{th: interlacing}(b).
For condition (e) we need that $\widehat \Gamma_1(2\xi) \subset \R$.
If $2\xi > \xicr$, this is guaranteed by Lemma \ref{lemma: relation 1cc hat} and
Theorem \ref{th: gamma 1cc}. If $2\xi \in (0, \xicr)$, this follows from Theorem \ref{th: gamma 2cc}. Applying Theorem \ref{th: KR}, we obtain
\[
    \lim_{l/n \to 1/2} \nu(r_{l,n})=2 \int_0^{1/2} \widehat \mu_1^{2\xi} \ud \xi=
     \int_0^1 \widehat \mu_1^\xi \ud \xi.
\]
Then by \eqref{eq: p r even}, \eqref{eq: muj 2cc}, and \eqref{eq: mu hat 1 1cc}, we obtain
\[
    \lim_{\substack{n \to \infty \\ n \text{ even}}} \nu(p_{n,n}) = \int_0^1 \mu_1^\xi \ud \xi,
        = \nu_1
\]
by the definition of $\nu_1$.
In a similar way one finds the same limit $\nu_1$ for the subsequence of odd $n$, completing the proof of Theorem \ref{th: zero distribution}. Alternatively, this follows from the interlacing of zeros.
\end{proof}

\section{Proof of Theorem \ref{th: equilibrium problem}} \label{sec: integrating}

\subsection{Averaging the vector equilibrium problems}

In this section we prove that the limiting zero distribution
$\nu_1$ can be characterized by a vector equilibrium problem for
three measures with external fields acting on the first and the
third measure, and a constraint acting on the second measure, see
Theorem \ref{th: integrating}. Recall that $\nu_1$ is given by \eqref{eq: nu1}.
Similarly we define the measures $\nu_2$ and $\nu_3$ by
\begin{equation}
\nu_j  = \int_0^{1} \mu_j^\xi \ud \xi, \qquad j=2,3.  \label{eq: nu 23}
\end{equation}

It follows from Theorem \ref{th: gamma 1cc} and Theorem \ref{th: gamma 2cc} that
the supports of $\mu_j^{\xi}$ as a function of $\xi$
are increasing if $j=1,3$ and decreasing if $j=2$. Hence,
\[
\supp(\nu_j) = \bigcup_{0 < \xi \leq 1} \supp(\mu_j^{\xi})
    = \begin{cases}
        \Gamma_1(1) \subset \mathbb R, & \text{for } j=1, \\
        \Gamma_2(0+) \subset i \mathbb R, & \text{for } j = 2, \\
        \Gamma_3(1) \subset \mathbb R, & \text{for } j = 3.
    \end{cases}
\]
We also know that $\nu_1(\R) = 1$, $\nu_2(i \R) = 2/3$, $\nu_3 (\R) = 1/3$.

The vector of measures $(\nu_1,\nu_2,\nu_3)$ is characterized by a vector equilibrium problem.

\begin{theorem} \label{th: integrating}
Define
\begin{align}
    \widetilde V_1(x) &= \frac12 \int_{0}^{\xicr}  \log \left| \frac{w_1(x;\xi)}{w_2(x;\xi)}\right| \ud \xi + \int_{\xicr}^{\infty} \log \left| \frac{w_1(x;\xi)}{w_2(x;\xi)}\right| \ud \xi, \quad  x \in \R, \label{eq: V1 def}\\
    \widetilde \sigma &= \int_0^{\infty} \mu_2^\xi \ud \xi, \label{eq: sigma def}\\
    \widetilde V_3(x) &= \frac12 \int_{0}^{\xicr}  \log \left| \frac{w_3(x;\xi)}{w_4(x;\xi)}\right| \ud \xi + \int_{\xicr}^{\infty} \log \left| \frac{w_3(x;\xi)}{w_4(x;\xi)}\right| \ud \xi, \quad  x \in \R. \label{eq: V3 def}
\end{align}
Then, $(\nu_1,\nu_2,\nu_3)$ is the unique vector of measures minimizing the energy functional
\begin{multline*}
    E(\rho_1,\rho_2,\rho_3)= I(\rho_1)-I(\rho_1,\rho_2) +
    I(\rho_2)-I(\rho_2,\rho_3) + I(\rho_3) \\
    +\int \widetilde V_1(x) \ud \rho_1(x)
    + \int \widetilde V_3(x) \ud \rho_3(x), \label{eq: energy functional0}
\end{multline*}
among all measures $\rho_1$, $\rho_2$ and $\rho_3$ satisfying
\begin{itemize}
\item[\rm (a)] $\rho_1$ is supported on $\R$ and $\rho_1(\R)=1$,
\item[\rm (b)] $\rho_2$ is supported on $i\R$ and $\rho_2(i\R)=2/3$,
\item[\rm (c)] $\rho_3$ is supported on $\R$ and $\rho_3(\R)=1/3$,
\item[\rm (d)] $\rho_2$ satisfies the constraint $\rho_2 \leq \widetilde \sigma$.
\end{itemize}
\end{theorem}

\begin{proof}
We have already shown that $(\nu_1, \nu_2, \nu_3)$ satisfies the conditions (a), (b) and (c).
Also (d) holds, because of the definitions \eqref{eq: nu 23} and \eqref{eq: sigma def}.
Since $\Gamma_2(\xi)$ is a set that decreases as $\xi$ increases, we have
\[
\supp (\widetilde \sigma-\nu_2)= \Gamma_2(1).
\]

Since the energy functional is strictly convex it is sufficient to show
that the Euler-Lagrange variational conditions associated with the minimization
problem are satisfied, namely
\begin{equation}
    2U^{\nu_1}(x)-U^{\nu_2}(x)+ \widetilde V_1(x)
    \begin{cases}
    = \ell &  \textrm{for } x \in \supp (\nu_1), \\
    > \ell &  \textrm{for } x \in \R \setminus \supp (\nu_1),
    \end{cases}
\label{eq: el nu 1}
\end{equation}
for some $\ell$,
\begin{equation}
-U^{\nu_1}(z)+2U^{\nu_2}(z)-U^{\nu_3}(z)
\begin{cases}
= 0 &  \textrm{for } z \in \supp (\widetilde \sigma - \nu_2), \\
< 0 &  \textrm{for } z \in i\R \setminus \supp (\widetilde \sigma - \nu_2),
\end{cases}
\label{eq: el nu 2}
\end{equation}
and
\begin{equation}
-U^{\nu_2}(x)+2U^{\nu_3}(x)+ \widetilde V_3(x)
\begin{cases}
= 0 &  \textrm{for } x \in \supp (\nu_3), \\
> 0 &  \textrm{for } x \in \R \setminus \supp (\nu_3).
\end{cases}
\label{eq: el nu 3}
\end{equation}

We establish \eqref{eq: el nu 1}--\eqref{eq: el nu 3} by integrating
the Euler Lagrange variational conditions \eqref{eq: EL 1} -- \eqref{eq: EL 3}
\begin{align}
\ell^\xi-2U^{ \mu_{1}^\xi}(z)+U^{\mu_{2}^\xi}(z) &= \log \left| \frac{w_1(z;\xi)}{w_2(z;\xi)}\right|,  \label{eq: log1 1cc}\\
U^{\mu_{1}^\xi}(z)-2U^{\mu_{2}^\xi}(z)+U^{\mu_{3}^\xi}(z) &= \log \left| \frac{w_2(z;\xi)}{w_3(z;\xi)}\right|, \label{eq: log2 1cc} \\
U^{\mu_{2}^\xi}(z)-2U^{\mu_{3}^\xi}(z) &= \log \left| \frac{w_3(z;\xi)}{w_4(z;\xi)}\right|, \label{eq: log3 1cc}
\end{align}
if $\xi > \xicr$ and \eqref{eq: EL 1 2cc} --\eqref{eq: EL 3 2cc}
\begin{align}
\ell^\xi-2U^{ \mu_{1}^\xi}(z)+U^{\mu_{2}^\xi}(z) &= \frac12 \log \left| \frac{w_1(z;\xi)}{w_2(z;\xi)}\right|,  \label{eq: log1 2cc}\\
U^{\mu_{1}^\xi}(z)-2U^{\mu_{2}^\xi}(z)+U^{\mu_{3}^\xi}(z) &= \frac12 \log \left| \frac{w_2(z;\xi)}{w_3(z;\xi)}\right|, \label{eq: log2 2cc} \\
U^{\mu_{2}^\xi}(z)-2U^{\mu_{3}^\xi}(z) &= \frac12 \log \left| \frac{w_3(z;\xi)}{w_4(z;\xi)}\right|, \label{eq: log3 2cc}
\end{align}
for $\xi \in (0,\xicr)$ with respect to $\xi$.
\bigskip

Integrating \eqref{eq: log1 2cc} from 0 to  $\min (\xicr,1)$ and \eqref{eq: log1 1cc} from $\min (\xicr,1)$ to 1 yields
\[
\ell-2U^{ \nu_1}(z)+U^{\nu_2}(z) = \frac12 \int_0^{\min (\xicr,1)} \log \left| \frac{w_1(z;\xi)}{w_2(z;\xi)}\right| \ud \xi + \int_{\min (\xicr,1)}^1 \log \left| \frac{w_1(z;\xi)}{w_2(z;\xi)}\right| \ud \xi,
\]
for some constant $\ell \in \R$.

Let $x \in \mathbb R$. Since $|w_1(x;\xi)| \geq |w_2(x;\xi)|$ for
every $\xi>0$, we can extend the integration to infinity and
obtain an inequality
\begin{equation}
    \ell-2U^{ \nu_{1}}(x)+U^{\nu_{2}}(x) \leq \widetilde V_1(x),
\label{eq: integrating 1}
\end{equation}
since $\widetilde V_1$ is given by \eqref{eq: V1 def}.
Equality holds in \eqref{eq: integrating 1} if and only if
$|w_1(x;\xi)| = |w_2(x;\xi)|$ for every $\xi > 1$. That is, if and only if
$x \in \Gamma_1(\xi)$ for every $\xi > 1$, which means
\[
x \in \bigcap_{\xi \geq 1}\Gamma_1(\xi) = \Gamma_1(1)=\supp (\nu_1).
\]
The first equality holds since the sets $\Gamma_1(\xi)$ are increasing as $\xi$ increases.
This proves \eqref{eq: el nu 1}.

The proof of \eqref{eq: el nu 3} is similar.

Integrating \eqref{eq: log2 2cc} from $0$ to $\min (\xicr,1)$ and \eqref{eq: log2 1cc}
from $\min (\xicr,1)$ to 1 yields
\begin{multline*}
    U^{ \nu_{1}}(z)-2U^{\nu_{2}}(z)+U^{ \nu_{3}}(z) \\
    = \frac12 \int_0^{\min (\xicr,1)} \log \left| \frac{w_2(z;\xi)}{w_3(z;\xi)}\right| \ud \xi + \int_{\min (\xicr,1)}^1 \log \left| \frac{w_2(z;\xi)}{w_3(z;\xi)}\right| \ud \xi.
\end{multline*}
Since $|w_2(z;\xi)| \geq |w_3(z;\xi)|$ for every $\xi>0$ it follows that
\begin{equation}
U^{ \nu_{1}}(z)-2U^{\nu_{2}}(z)+U^{ \nu_{3}}(z) \geq 0, \quad z \in \mathbb C.
\label{eq: integrating 2}
\end{equation}
Equality holds in \eqref{eq: integrating 2} if and only if $|w_2(z;\xi)| = |w_3(z;\xi)|$ for every $\xi \in (0,1)$. That is, if and only if
\[
z \in  \bigcap_{0<\xi<1}\Gamma_2(\xi) = \Gamma_2(1)=\supp (\widetilde \sigma-\nu_2).
\]
The first equality holds since the sets $\Gamma_2(\xi)$ are decreasing as $\xi$ increases.
This proves \eqref{eq: el nu 2}.

This completes the proof of Theorem \ref{th: integrating}
\end{proof}

%\section{Proof of Theorem \ref{th: equilibrium problem}: part 2} \label{sec: evaluation}

\subsection{Auxiliary lemmas}

From Theorem \ref{th: integrating} we know that $(\nu_1, \nu_2, \nu_3)$
is the minimizer for a vector equilibrium problem with external fields and an
upper constraint. To complete the proof of Theorem \ref{th: equilibrium problem}
we evaluate the external fields $\widetilde V_1$ and $\widetilde V_3$
and the  measure $\widetilde \sigma$ and show that they are equal to
the external fields $V_1$ (up to a constant) and $V_3$  and the constraint $\sigma$ that
appear in Theorem \ref{th: equilibrium problem}.

Although the calculations involved may not look too elegant,
we think it is remarkable that they can be performed at all.
We start by defining and investigating two functions.
We define for  $j = 1,2,3,4$,
\begin{equation}
F_j (z, \xi) = z-\tau^2\frac{a(\xi)}{w_j(z;\xi)}-w_j(z;\xi), \qquad \xi > \xicr, \quad z \in \C,
\label{eq: F}
\end{equation}
where $w_j(z;\xi)$ is defined as in subsection \ref{subsec: 1cc}.
If $t < \tau^2$ we also introduce for $j = 1,2,3, 4$,
\begin{equation}
G_j(z,\xi) =  \frac{z \xi}{w_j(z;\xi)+\xi}, \qquad 0 < \xi < \xicr, \quad z \in \C,
\label{eq: G}
\end{equation}
where $w_j(z;\xi)$ is defined by \eqref{eq: wj 2cc}.
We will establish three lemmas concerning these functions.
The first lemma states that $F_j$ and $G_j$ are
antiderivatives of the integrands in \eqref{eq: V1 def} and \eqref{eq: V3 def}.
This explains our interest in these functions.
The second lemma states that $F_j$ and $G_j$ continuously connect
to each other at the boundary point $\xicr$ of their domains of
definition. The third lemma gives the limiting behavior of $F_j(z,\xi)$ and $G_j(z,\xi)$ as $\xi \to 0$.

\begin{lemma} \label{lemma: integration lemma}
Let $'$ denote the partial derivative with respect to $z$.
\begin{itemize}
\item[\rm (a)]  For every  $t \in \R$ and $z \in \C$ the equality
\begin{equation}
\frac{\partial F_j}{\partial \xi}(z,\xi)= \frac{w_j'(z;\xi)}{w_j(z;\xi)}, \qquad \xi > \xicr,
\end{equation}
holds. Here, $w_j(z;\xi)$ is defined as in subsection \ref{subsec: 1cc}.
\item[\rm (b)] For every  $t < \tau^2$ and $z \in \C$ the equality
\begin{equation}
\frac{\partial G_j}{\partial \xi}(z,\xi)= \frac12 \cdot \frac{w_j'(z;\xi)}{w_j(z;\xi)}, \qquad 0 < \xi < \xicr,
\end{equation}
holds. Here, $w_j(z;\xi)$ is defined by \eqref{eq: wj 2cc}.
\end{itemize}
\end{lemma}

\begin{proof}
(a) Let $\xi > \xicr$, $z \in \C$ and define $W_j(z;\xi)$ for $j=1,2,3,4$ as
\begin{equation}
    W_j(z;\xi)=\frac{w_j(z;\xi)}{a(\xi)}.
    \label{eq: W w}
\end{equation}
Since $s_1(w_j(z,\xi)) = z$ we find by \eqref{eq: s1} and the explicit
expressions \eqref{eq: b} and \eqref{eq: c} for $b(\xi)$ and $c(\xi)$ that
\begin{equation}
    S_1(W_j(z; \xi); \xi)  =
    a(\xi) \left( W_j(z;\xi)+ \frac{3}{W_j(z;\xi)}\right)+\frac{t}{W_j(z;\xi)}+\frac{\tau^2}{W_j(z;\xi)^3}=z,
    \label{eq: alg eq S1}
\end{equation}
see \eqref{eq: S1} for the definition of $S_1(W; \xi)$.
By \eqref{eq: F} and \eqref{eq: W w}, we can  write
\begin{equation}
    F_j(z,\xi)= z-\frac{\tau^2}{W_j(z;\xi)}-a(\xi)W_j(z;\xi),
\label{eq: F W}
\end{equation}
so that
\begin{equation}
    \frac{\partial F_j}{\partial \xi}(z,\xi)
    =-a'(\xi)W_j(z;\xi)+\left(\frac{\tau^2}{W_j(z;\xi)^2}-a(\xi)\right)\frac{\partial W_j}{\partial \xi}(z;\xi).
\label{eq: Fd}
\end{equation}
Taking the  derivative of \eqref{eq: alg eq S1}  with respect to $\xi$,
we eliminate $a'(\xi)$ from \eqref{eq: Fd} to obtain after some calculations
\begin{equation}
    \begin{aligned}
    \frac{\partial F_j}{\partial \xi}(z,\xi) &
    =\frac{\partial W_j}{\partial \xi}(z;\xi) \frac{\tau^2-t-6a(\xi)}{W_j(z;\xi)^2+3} \\
    & =-\frac{\partial W_j}{\partial \xi}(z;\xi) \cdot \frac{1}{a'(\xi)}\cdot \frac{1}{W_j(z;\xi)^2+3},
    \end{aligned}
\label{eq: int lemma 1}
\end{equation}
where the second equality follows from \eqref{eq: a}.
Calculating partial derivatives of  \eqref{eq: alg eq S1} with respect to $z$ and $\xi$ yields
\begin{align*}
\frac{\partial S_1}{\partial W}\frac{\partial W_j}{\partial z} = 1, \qquad
\frac{\partial S_1}{\partial W}\frac{\partial W_j}{\partial \xi}+\frac{\partial S_1}{\partial \xi}=0,
\end{align*}
from which we obtain
\[
    -\frac{\partial W_j}{\partial \xi}=
        \frac{\partial W_j}{\partial z} \cdot \frac{\partial S_1}{\partial \xi}
        =\frac{\partial W_j}{\partial z} \cdot a'(\xi)\frac{W_j^2+3}{W_j},
\]
where the last equality follows from \eqref{eq: S1}. Now, combine this with \eqref{eq: int lemma 1} to obtain
\[
\frac{\partial F_j}{\partial \xi}(z,\xi)=\frac{W_j'(z;\xi)}{W_j(z;\xi)}.
\]
Recalling \eqref{eq: W w} completes the proof of part (a).

\medskip

(b) The proof of part (b) follows along the same lines. Let $0 < \xi < \xicr$ and $z \in \C$.
Define $ \widehat W_j(z;\xi)$ for $j=1,2,3,4$ by
\begin{equation}
    \widehat W_j(z;\xi)=\frac{\widehat w_j(z;\xi)}{\xi},
\label{eq: W w 2}
\end{equation}
so that $\widehat W_j(z^2; \xi) = \frac{w_j(z;\xi)}{\xi}$, see \eqref{eq: wj 2cc}.
Then by \eqref{eq: G} we can write
\begin{equation}
G_j(z,\xi)= \frac{z}{\widehat W_j(z^2;\xi)+1},
\label{eq: G W}
\end{equation}
so that
\begin{equation}
    \frac{\partial G_j}{\partial \xi}(z,\xi)
    =-\frac{\partial \widehat W_j}{\partial \xi}(z^2;\xi)\frac{z}{\left( \widehat W_j(z^2;\xi)+1 \right)^2}.
\label{eq: int lemma 2}
\end{equation}

Recall the definition \eqref{eq: S2} of the transformed symbol $\widehat S_2$
and note that by the definition \eqref{eq: W w 2}
\begin{equation}
    \widehat S_2(\widehat W_j(z;\xi);\xi)=
    \xi\frac{( \widehat W_j(z;\xi)+1)^2}{\widehat W_j(z;\xi)}-
    t\tau^2\frac{( \widehat W_j(z;\xi)+1)^2}{\widehat W_j(z;\xi)^2}
    +\tau^4\frac{( \widehat W_j(z;\xi)+1)^3}{\widehat W_j(z;\xi)^3}=z,
\label{eq: alg eq S2}
\end{equation}
see \eqref{eq: S2}.
Calculating partial derivatives of \eqref{eq: alg eq S2} with respect to $z$ and $\xi$ yields
\begin{align*}
\frac{\partial \widehat S_2}{\partial W}\frac{\partial \widehat W_j}{\partial z} = 1, \qquad
\frac{\partial \widehat S_2}{\partial W}\frac{\partial \widehat W_j}{\partial \xi}+\frac{\partial \widehat S_2}{\partial \xi}=0.
\end{align*}
Therefore,
\[
-\frac{\partial \widehat W_j}{\partial \xi}=\frac{\partial \widehat W_j}{\partial z} \cdot \frac{\partial \widehat S_2}{\partial \xi}= \frac{\partial \widehat W_j}{\partial z} \cdot \frac{(\widehat W_j+1)^2}{\widehat W_j}
\]
where the last equality follows from \eqref{eq: S2}. If we combine this with \eqref{eq: int lemma 2} we obtain
\[
    \frac{\partial G_j}{\partial \xi}(z,\xi)=z \frac{\widehat W_j'(z^2;\xi)}{\widehat W_j(z^2;\xi)}.
\]
This completes the proof of part (b) because of \eqref{eq: W w 2}.
\end{proof}

\begin{lemma} \label{lemma: continuity lemma}
Assume $t<\tau^2$. Then for $z \in \C$ and $j = 1, 2,3,4$ the equality of the limits
\begin{equation}
    \lim_{\xi \to \xicr-}G_j(z,\xi)=\lim_{\xi \to \xicr+}F_j(z,\xi)
    \label{eq: continuity}
\end{equation}
holds. Here, $F_j$ and $G_j$ are defined as in \eqref{eq: F} and \eqref{eq: G}.
\end{lemma}
\begin{proof}
We claim that
\[
\lim_{\xi \to \xicr+}W_j(z;\xi)^2=\lim_{\xi \to \xicr-}\widehat W_j(z^2;\xi).
\]
To prove the claim, observe that, as a corollary of \eqref{eq: s1 s2},
\[
\lim_{\xi \to \xicr+}w_j(z;\xi)^2=\lim_{\xi \to \xicr-}\widehat w_j(z^2;\xi).
\]
Then, from \eqref{eq: a} it follows that
\begin{equation}
\lim_{\xi \to \xicr+}a(\xi)=\frac{\tau^2-t}{2},
\label{eq: a xicr}
\end{equation}
so that also
\begin{multline*}
\lim_{\xi \to \xicr+}W_j(z;\xi)^2=\left(\frac{\tau^2-t}{2}\right)^{-2}\lim_{\xi \to \xicr+}w_j(z;\xi)^2 \\
 =\lim_{\xi \to \xicr-}\frac{\widehat w_j(z^2;\xi)}{\xicr}=\lim_{\xi \to \xicr-}\widehat W_j(z^2;\xi),
\end{multline*}
which proves the claim.

Note also that \eqref{eq: a xicr} and \eqref{eq: alg eq S1} yield
\begin{equation}
z=\lim_{\xi \to \xicr+}\left(\frac{\tau^2-t}{2}\cdot W_j(z;\xi)+\frac{3\tau^2-t}{2}\cdot\frac{1}{W_j(z;\xi)}+\frac{\tau^2}{W_j(z;\xi)^3}\right).
\label{eq: z xicr}
\end{equation}

Given this, the proof of \eqref{eq: continuity} comes down to a calculation.
Using \eqref{eq: F W} and \eqref{eq: z xicr}, we can rewrite the right-hand side of \eqref{eq: continuity} as
\[
\lim_{\xi \to \xicr+} F_j(z,\xi) =
\lim_{\xi \to \xicr+} \left( \frac{\tau^2-t}{2}\cdot \frac{1}{W_j(z;\xi)}+\frac{\tau^2}{W_j(z;\xi)^3}\right)
= \lim_{\xi \to \xicr+} \frac{z}{W_j(z;\xi)^2+1},
\]
where also the second equality follows from \eqref{eq: z xicr}.
Lemma \ref{lemma: continuity lemma} then follows from the claim and equation \eqref{eq: G W}.
\end{proof}

We also need the limiting behavior of the $F_j$ and $G_j$
functions as $\xi \to 0+$. Here we make a connection with the function
$\omega_1$ that is defined in subsection \ref{subsection: vector equilibrium problem},
and the functions $\omega_2$ and $\omega_3$ that are defined on
the interval $[-x^*, x^*]$ in case $t < 0$.
\begin{lemma} \label{lemma: xi to zero}
\begin{itemize}
\item[\rm (a)] If $t \geq \tau^2$, then
\begin{align}
\lim_{\xi \to 0+} F_1(x,\xi)&= 0, && \text{ for }x \in \R, \label{eq: F1} \\
\lim_{\xi \to 0+} F_2(x,\xi)&= x-\tau \omega_1(x), && \text{ for }x \in \R. \label{eq: F2}
\end{align}
\item[\rm(b)] If $t < \tau^2$, then
\begin{align}
\lim_{\xi \to 0+}G_1(x,\xi)&= 0, && \text{ for }x \in \R, \label{eq: G1} \\
\lim_{\xi \to 0+}G_2(x,\xi)&= x-\tau \omega_1(x), && \text{ for } x \in \R. \label{eq: G2}
\end{align}
If $t<0$, then also for $j=3,4$,
\begin{equation}
    \lim_{\xi \to 0+}G_j(x,\xi)= x-\tau \omega_{j-1}(x), \qquad \text{ for }x \in [-x^*,x^*]. \label{eq: G34}
\end{equation}
\item[\rm (c)] If $t \geq \tau^2$ and $z \in (-i \infty,-iy^*] \cup [iy^*,+i \infty)$,
then
\begin{equation}
\lim_{\xi \to 0+}\left(F_2(z+,\xi)-F_2(z-,\xi)\right)= 2\tau \Re \omega_1(z),  \label{eq: F23}
\end{equation}
where $F_2(z\pm,\xi)=\lim\limits_{h \to 0+}F_2(z\mp h,\xi)$.
\item[\rm (d)] If $t < \tau^2$ and $z \in (-i \infty,-iy^*] \cup [iy^*,+i \infty)$, then
\begin{equation}
\lim_{\xi \to 0+}\left(G_2(z+,\xi)-G_2(z-,\xi)\right)= 2\tau \Re \omega_1(z), \label{eq: G23}
\end{equation}
where $G_2(z\pm,\xi)=\lim\limits_{h \to 0+}G_2(z\mp h,\xi)$.
\end{itemize}
\end{lemma}

\begin{proof}

(a)
Let $t \geq \tau^2$. For $\xi > 0$, we define
\begin{equation}
\omega_j(z;\xi)=\frac{\tau a(\xi)}{w_{j+1}(z;\xi)}, \qquad j=0,1,2,3.
\label{eq: omega 1cc}
\end{equation}
Since $s_1(w_j(z;\xi);\xi)=z$, $j=1,2,3,4$, see \eqref{eq: s1}, we obtain that
\eqref{eq: omega 1cc} are the four solutions of the equation
\begin{equation}
 a(\xi) \left( \frac{\tau^2}{\omega} + 3 \omega \right) + t \omega + \omega^3  = \tau z,
\label{eq: s1 omega}
\end{equation}
ordered such that
\[
|\omega_0(z; \xi)| \leq |\omega_1(z; \xi)| \leq |\omega_2(z;\xi)| \leq |\omega_3(z;\xi)|.
\]
As $\xi \to 0+$, we have that $a(\xi) \to 0$. Then,  \eqref{eq: s1 omega} has one solution
that tends to $0$ as well, hence
\[
\lim_{ \xi \to 0+} \omega_0(z; \xi) = 0,
\]
while the other solutions tend to the solutions of the cubic equation
\[
\omega^3+ t \omega  = \tau z,
\]
that we already encountered in \eqref{eq: omegas}.

In terms of \eqref{eq: omega 1cc}, $F_j(z,\xi)$ can be rewritten as
\begin{equation}
F_j(z,\xi) = z-\tau \omega_{j-1}(z;\xi)-\frac{\tau a(\xi)}{\omega_{j-1}(z;\xi)},  \quad j=1,2,3,4.
\label{eq: F lim1}
\end{equation}
Moreover, using \eqref{eq: s1 omega} we get
\[
F_1(z,\xi)=\frac{\omega_{0}(z;\xi)}{\tau}\left( 3a(\xi)+t-\tau^2+\omega_{0}(z;\xi)^2 \right),
\]
from which \eqref{eq: F1} follows by letting $\xi \to 0$.

Let  $x \in \mathbb R$ with $x \neq 0$. Then by Theorem \ref{th: gamma 1cc} (a)
we have $x \notin \Gamma_1(\xi)$ for sufficiently small $\xi$. Thus $w_1(x;\xi)$ and $w_2(x;\xi)$
are real, while $w_3(x;\xi)$ and $w_4(x;\xi)$ are non-real and complex conjugate, see Figure \ref{fig: s1R}.
Hence, by \eqref{eq: omega 1cc}, $\omega_1(x;\xi)$ is real for small $\xi$ and, therefore,
converges as $\xi \to 0$ to a real solution of \eqref{eq: omegas}.
In the present situation there is only one real solution, which was previously defined as $\omega_1(x)$.
Thus $\omega_1(x;\xi) \to \omega_1(x)$ as $\xi \to 0+$ for $x \in \mathbb R$ with $x \neq 0$,
but in view of continuity it also holds for $x = 0$.
Then, \eqref{eq: F2} is obtained by taking the limit  $\xi \to 0$ in \eqref{eq: F lim1}
with $j=1$ and noting that $a(\xi) \to 0$.

\bigskip

(b) We prove part (b) only for $t < 0$ and $x \in [0,x^*]$.
The proof for the other cases follows from similar considerations.

For $t < 0$, we recall that
\begin{equation} \label{eq: s2equation}
    \widehat s_2(w_j(z;\xi);\xi) = z^2, \qquad j=1,2,3,4,
\end{equation}
see Section \ref{subsec: iteration 2cc} and \eqref{eq: wj 2cc}.
By Figure \ref{fig: subregions} we are in the case
$C_{2b}$ for small enough $\xi$. See Figure \ref{fig: s2R2}
for the graph of $\widehat s_2$ in this case.

Let $x \in [0, x^*)$. Then $x^2 < \widehat\delta(\xi)$ for sufficiently small $\xi$ by part (d) of
Theorem \ref{th: gamma 2cc}. It then follows that $w_3(x; \xi)$
and $w_4(x;\xi)$ are real.
We now distinguish two cases depending on whether $x^2$ is smaller
than $\tau^2(\tau^2-t)$ or not.
\begin{itemize}
\item If $x^2 < \tau^2(\tau^2-t)$ then it follows from part (b)
of Theorem \ref{th: gamma 2cc} that $x^2 < \widehat{\beta(\xi)}$ for small $\xi$.
Then by \eqref{eq: s2equation} and Figure \ref{fig: s2R2} we have
\[ w_1(x;\xi) < w_2(x;\xi) < -\xi < w_3(x; \xi ) < w_4(x; \xi) < 0. \]
By the definition \eqref{eq: G} this implies
\begin{equation} \label{eq: Gordering1}
    G_2(x; \xi) < G_1(x; \xi) < 0 < x < G_4(x; \xi) < G_3(x; \xi).
    \end{equation}
\item If $x^2 > \tau^2(\tau^2-t)$ then by part (a) of Theorem \ref{th: gamma 2cc}
we have that $x^2 > \widehat\alpha(\xi)$ for small enough $\xi$.
In this case in Figure \ref{fig: s2R2} the local minimum $\widehat\alpha(\xi)$ at $w_0^*$
is smaller than the local maximum $\widehat\delta(\xi)$ at $w_3^*$.
We then get by \eqref{eq: s2equation} and Figure \ref{fig: s2R2} that
\[ -\xi < w_3(x; \xi) < w_4(x;\xi) < 0 < w_2(x; \xi) < w_1(x;\xi). \]
Thus by \eqref{eq: G}
\begin{equation} \label{eq: Gordering2}
    0 < G_1(x; \xi) < G_2(x;\xi) < x < G_4(x; \xi) < G_3(x; \xi).
    \end{equation}
\end{itemize}

Next, we get from \eqref{eq: s2equation} and \eqref{eq: G}
that the $G_j(x;,\xi)$, $j=1, \ldots, 4$, are the four solutions of
\[
    G^4 - 3xG^3+(3x^2 +t \tau^2+\xi)G^2 + (-x^3+\tau^4 x -t \tau^2 x -2 \xi x)G + \xi x^2 =0.
\]
Letting $\xi \to 0$, we see that one of the $G_j$'s tends to $0$
while the other three tend to the solutions of
\begin{equation} \label{eq: equation G}
    G^3-3xG^2+(3x^2+t \tau^2)G  -x^3+\tau^2(\tau^2 -t)x =0.
\end{equation}
In view of \eqref{eq: Gordering1} and \eqref{eq: Gordering2} it must
be $G_1(x; \xi)$ that tends to $0$ and so we proved  \eqref{eq: G1}.

Now let $\omega_j(x; \xi) = \frac{1}{\tau} (x-G_{j+1}(x; \xi))$,
so that
\begin{equation} \label{eq: omega xi}
    G_j(x; \xi) = x - \tau \omega_{j-1}(x; \xi),
        \qquad  j=2,3,4.
    \end{equation}
It then follows from \eqref{eq: Gordering1} and \eqref{eq: Gordering2}
that in both cases
\begin{equation} \label{eq: omegaordering}
    \omega_1(x; \xi) > 0 > \omega_3(x; \xi) > \omega_2(x; \xi).
    \end{equation}
As $\xi \to 0$, we have that $G_j(x;\xi)$, $j=2,3,4$ tend to the three
solutions of \eqref{eq: equation G}. Then from \eqref{eq: omega xi}
we get that $\omega_j(x;\xi)$, $j=1,2,3$, tend to the solutions of
\begin{equation}
    \omega^3 + t\omega=\tau x. \label{eq: SPE aux}
\end{equation}
In view of \eqref{eq: omegaordering} and the earlier definitions
of $\omega_j(x)$, $j=1,2,3$,  it is then easy to check
that
\[ \lim_{\xi \to 0+}  \omega_j(x; \xi) = \omega_j(x). \]
This proves \eqref{eq: G2} and \eqref{eq: G34} because of \eqref{eq: omega xi}.

(c) Let $t \geq \tau^2$ and $z \in (-i \infty,-iy^*) \cup (i y^*, i \infty)$.
Recall that $\omega_1(z)$ is defined as the solution of $\omega^3+ t \omega  = \tau z$
with positive real part.
As in the proof of part (a) we use again the definition \eqref{eq: omega 1cc}
for $\omega_j(z; \xi)$. These functions are defined with possible
cuts on parts of the real and imaginary axis. We prove that
\begin{equation} \label{eq: limitomega1}
    \lim_{\xi \to 0+} \omega_{1,-}(z;\xi) = \omega_1(z).
    \end{equation}
where $\omega_{1,-}(z; \xi)$ denotes the limiting value on the imaginary
axis taken from the right half plane.

To prove \eqref{eq: limitomega1}, we note that, by Theorem \ref{th: gamma 1cc}(b),
$z \in \Gamma_2(\xi)$ for sufficiently small $\xi$.
Then, $w_{2,+}(z;\xi)$ and $w_{2,-}(z;\xi)$ have the same imaginary part, but opposite real part.
Note that $w_2(x;\xi)$ is positive for sufficiently large $x>0$ and that $w_2( \cdot ;\xi)$ does not take
purely imaginary values for $z$ in the right half plane because it solves the algebraic equation
\eqref{eq: alg eq s1 R}. It then follows from the continuity of $w_2( \cdot ;\xi)$ that $w_{2,-}(z;\xi)$
has positive real part.  Then, by \eqref{eq: omega 1cc}, also $\omega_{1,-}(z,\xi)$ has positive real
part for small enough $\xi$. This proves \eqref{eq: limitomega1}.

Since $w_{2,+}(z;\xi)$ and $w_{2,-}(z;\xi)$ have opposite real part, we know that
\[
F_2(z+,\xi)-F_2(z-,\xi) =-2 \Re F_2(z-,\xi)
 =2\tau \Re \left(\omega_{1,-}(z;\xi)-\frac{\tau a(\xi)}{\omega_{1,-}(z;\xi)}\right).
\]
Then, \eqref{eq: F23} follows by letting $\xi \to 0$ and using \eqref{eq: limitomega1}
and $a(\xi) \to 0$.

\bigskip

(d) Part (d) can be proved in a similar way. We do not give details.
% Suppose that $z \in i\R^+$. Then, $z \in \Gamma_2(\xi)$ for sufficiently
% small $\xi$, see Theorem \ref{th: gamma 2cc}.Thus, $w_2(z;\xi)$ and $w_3(z;\xi)$ are complex conjugate or,
% equivalently, $w_2(z+;\xi)=\overline{w_2(z-;\xi)}$, where $w_2(z \pm; \xi)$ is the limit of $w_2(\tilde z; \xi)$
% as $\tilde z$ tends to $z$ from the left/right half plane. It then follows from \eqref{eq: G},
% \eqref{eq: taking limit} and \eqref{eq: tilde omega} that
% \[
% G_2(z+,\xi)-G_2(z-,\xi) \to -2 \Re G_1(z-)=2\tau \Re \tilde \omega_1(z-).
% \]
% To establish \eqref{eq: G23} the only thing left to prove is that $\tilde \omega_1(z-)$ coincides with
% $\omega_1(z)$ as defined in Section \ref{subsection: vector equilibrium problem}. So we have to show that
% the real part of $\tilde \omega_1(z-)$ is positive. Let us first assume $\Re \tilde \omega_1(z-)<0$. Then,
% by continuity there exists a $z_0$ in the first quadrant of the complex plane such that $\tilde \omega_1(z_0)$
% is purely imaginary. As $\tilde \omega_1(z_0)$ solves \eqref{eq: SPE aux} this is only possible for purely
% imaginary $z_0$ which yields a contradiction. Next, assume that $\tilde \omega_1(z-)$ is purely imaginary.
% This also leads to a contradiction. There is namely only one purely imaginary solution of \eqref{eq: SPE aux}
% and it follows from \eqref{eq: G}, \eqref{eq: taking limit} and \eqref{eq: tilde omega} that this is
% $\tilde \omega_3(z)$. Therefore, the real part of $\tilde \omega_1(z-)$ must be positive, which proves \eqref{eq: G23}.
\end{proof}

\subsection{Proof of Theorem \ref{th: equilibrium problem}}

Theorem \ref{th: equilibrium problem} follows immediately from Theorems \ref{th: zero
distribution} and \ref{th: integrating} and the following result
that connects  $\widetilde{V}_1$, $\widetilde{V}_3$ and $\widetilde{\sigma}$
as defined in Theorem \ref{th: integrating}, with $V_1$, $V_3$ and $\sigma$
that are defined in subsection \ref{subsection: vector equilibrium problem}.

\begin{theorem}
There is a constant $C$, depending on $t$ and $\tau$ such that
\begin{equation} \label{eq: V1equality}
    \widetilde{V}_1(x) = V_1(x)  + C, \qquad x \in \mathbb R.
    \end{equation}
Furthermore we have
\begin{equation} \label{eq: V3equality}
    \widetilde{V}_3(x) = V_3(x), \qquad x \in \mathbb R,
    \end{equation}
and
\begin{equation} \label{eq: sigmaequality}
    \widetilde{\sigma} = \sigma.
    \end{equation}
\end{theorem}

\begin{proof}[Proof of \eqref{eq: V1equality}.]
We distinguish two cases.

\paragraph{Case 1: $t \geq \tau^2$.} In this case \eqref{eq: V1 def} takes the simple form
\[
\widetilde V_1(x)=\int_0^\infty \log \left| \frac{w_1(x;\xi)}{w_2(x;\xi)} \right| \ud \xi.
\]
Fix $x \in \R$. As $\Gamma_1(\xi)$ is an unboundedly increasing set, there exists
$\xi^*(x) \geq 0$ such that $x \in \Gamma_1(\xi)$ if $\xi \geq \xi^*(x)$ and $x \not \in \Gamma_1(\xi)$
if $\xi < \xi^*(x)$. If $x \in \Gamma_1(\xi)$, we have that $|w_1(x;\xi)|=|w_2(x;\xi)|$,
so that the upper bound in the integral can be replaced by $\xi^*(x)$. The derivative of $\widetilde{V}_1$
can be written as
\[
    \widetilde{V}_1'(x)= \lim_{\xi_1 \to 0} \int_{\xi_1}^{\xi^*(x)}
    \left(\frac{w_1'(x;\xi)}{w_1(x;\xi)}-\frac{w_2'(x;\xi)}{w_2(x;\xi)}\right) \ud \xi.
\]
Part (a) of Lemma \ref{lemma: integration lemma} now yields
\begin{equation}
    \widetilde{V}_1'(x)= \lim_{\xi_1 \to 0 }\left( F_2(x,\xi_1)- F_1(x,\xi_1)\right)+
    \left(F_1(x,\xi^*(x))-F_2(x,\xi^*(x))\right).
\label{eq: V1 1}
\end{equation}

By definition of $\xi^*(x)$, $x$ is on the boundary of $\Gamma_1(\xi^*(x))$. Therefore,
$x$ is one of the branch points $\pm \alpha(\xi^*(x))$. The algebraic equation $s_1(w;\xi^*(x))=x$
then has the double solution $w_1(x;\xi^*(x))=w_2(x;\xi^*(x))$, so that the last terms of \eqref{eq: V1 1}
vanish, see \eqref{eq: F}. The first terms can be handled by \eqref{eq: F1} and \eqref{eq: F2}. We obtain
\[
    \widetilde{V}_1'(x) = x-\tau \omega_1(x).
\]
Using \eqref{eq: omegas}, we can integrate this equation with respect to $x$
\[
    \widetilde{V}_1(x)=\frac{x^2}{2}- \int \left( 3 \omega_1(x)^3 + t \omega_1(x) \right) \ud \omega_1(x)
    = \frac{x^2}{2}-\frac34 \omega_1(x)^4 - \frac12 t \omega_1(x)^2 + C,
\]
where $C$ is a constant of integration, which proves \eqref{eq: V1equality} in view of \eqref{eq: V1}.

\paragraph{Case 2: $t< \tau^2$.}
$\widetilde{V}_1$ is given by
\[
\widetilde{V}_1(x) = \frac12 \int_0^{\xicr} \log \left| \frac{w_1(x;\xi)}{w_2(x;\xi)}\right| \ud \xi + \int_{\xicr} ^\infty \log \left| \frac{w_1(x;\xi)}{w_2(x;\xi)}\right| \ud \xi.
\]
Fix $x \in \R$. The set $\Gamma_1(\xi)$ is unboundedly increasing if $\xi$ increases.
Therefore, there exists $\xi^*(x) \geq 0$ such that $x \in \Gamma_1(\xi)$ if $\xi \geq \xi^*(x)$
and $x \not \in \Gamma_1(\xi)$ if $\xi < \xi^*(x)$.

Assume that $\xi^*(x) > \xicr$. The other possibility, $\xi^*(x) \leq \xicr$, is simpler
and will be left to the reader. We obtain for the derivative of $\widetilde{V}_1$
\begin{multline*}
    \widetilde{V}_1'(x)=\frac12 \lim_{\xi_1 \to 0+} \lim_{\xi_2 \to \xicr-}
    \int_{\xi_1}^{\xi_2}\left( \frac{w_1'(x;\xi)}{w_1(x;\xi)}-\frac{w_2'(x;\xi)}{w_2(x;\xi)}\right) \ud \xi \\ +\lim_{\xi_3 \to \xicr+} \int_{\xi_3}^{\xi^*(x)}\left( \frac{w_1'(x;\xi)}{w_1(x;\xi)}-\frac{w_2'(x;\xi)}{w_2(x;\xi)}\right) \ud \xi.
\end{multline*}
By Lemma \ref{lemma: integration lemma} we obtain
\begin{multline*}
    \widetilde{V}_1'(x)=\lim_{\xi_1 \to 0+} \left( G_2(x,\xi_1)-G_1(x,\xi_1)\right) \\
    +\lim_{\xi_2 \to \xicr-}\lim_{\xi_3 \to \xicr+} \left( G_1(x,\xi_2)-F_1(x,\xi_3)
    +  F_2(x,\xi_3)-G_2(x,\xi_2)\right) \\
     +  F_1(x,\xi^*(x))-F_2(x,\xi^*(x)).
\end{multline*}
As in the previous case, it follows from the definition of $\xi^*(x)$ that $F_1(x,\xi^*(x)) = F_2(x,\xi^*(x))$.
The limit at $\xicr$ vanishes as a result of Lemma \ref{lemma: continuity lemma}.
So what remains is the limit for $\xi_1 \to 0+$, which by \eqref{eq: G1} and \eqref{eq: G2},
yields
\[
\widetilde{V}_1'(x) =x-\tau \omega_1(x).
\]
This leads to \eqref{eq: V1equality} in the same way as in the other case.
\end{proof}

\begin{proof}[Proof of \eqref{eq: V3equality}.]
From Theorem \ref{th: gamma 1cc} it follows that $\Gamma_3(\xi)=\R$ if $ \xi > \xicr$.
Then, for every $x \in \R$, we have that
$|w_3(x;\xi)|=|w_4(x;\xi)|$. Therefore, the
second integral of \eqref{eq: V3 def} vanishes. We now distinguish
two cases.

\paragraph{Case 1: $t \geq 0$, or $t < 0$ and $|x| \geq x^*$.}
If $t \geq 0$, then  $\Gamma_3(\xi)=\R$ for every $\xi >0$.
If $t < 0$ and $|x| \geq x^*$, then it follows from Theorem \ref{th: gamma 2cc}
that $x \in \Gamma_3(\xi)$ for all $\xi \in (0,\xicr)$.
So in both situations, we have that $|w_3(x;\xi)|=|w_4(x;\xi)|$ for every $x \in \R$ and
$\xi \in (0,\xicr)$. Then, \eqref{eq: V3 def} ensures us that $\widetilde V_3(x) = 0$.
Als $V_3(x) = 0$ in this case, by \eqref{eq: V3 0} and \eqref{eq: V3}.

\paragraph{Case 2: $t < 0$ and $|x| < x^*$.}

If $|x| < x^*$, there exists $\xi^*(x) < \xicr$ such that $x \in \Gamma_3(\xi)$ if
$\xi^*(x) \leq \xi$ and $x \not \in \Gamma_3(\xi)$ if
$\xi < \xi^*(x)$, because $\Gamma_3(\xi)$ is an increasing
set. Then we obtain
\[
    \widetilde V_3'(x)= \frac12 \lim_{\xi_1 \to 0}\int_{\xi_1}^{\xi^*(x)}\left(
    \frac{w_3'(x;\xi)}{w_3(x;\xi)}-\frac{w_4'(x;\xi)}{w_4(x;\xi)}\right)
    \ud \xi.
\]
Applying part (b) of Lemma \ref{lemma: integration lemma} yields
\[
\widetilde V_3'(x) =
    \lim_{\xi_1 \to 0+} \left( G_4(x,\xi_1)-G_3(x,\xi_1) \right) +
        (G_3(x,\xi^*(x))-G_4(x,\xi^*(x)).
\]
The last two of terms cancel each other because $x=\pm \delta(\xi^*(x))$
and $G_3(x, \xi^*(x)) = G_4(x, \xi^*(x))$
by definition of $\xi^*$. The limit for $\xi_1 \to 0+$ is calculated using
\eqref{eq: G34} and it follows that
\[
   \widetilde V'_3(x)=\tau (\omega_2(x)-\omega_3(x)).
\]
Integrating with respect to $x$ yields
\[
    \widetilde V_3(x) = \frac34  \omega_2(x)^4 + \frac12 t  \omega_2(x)^2-\frac34  \omega_3(x)^4
    - \frac12 t  \omega_3(x)^2+C,
\]
where $C$ is a constant of integration. Substituting $x= \pm x^*$ one can check that $C=0$.
This proves \eqref{eq: V3equality} in view of \eqref{eq: V3}.
\end{proof}

\begin{proof}[Proof of \eqref{eq: sigmaequality}.]
By \eqref{eq: sigma def} we have
\[ \supp(\widetilde \sigma) = \overline{\bigcup_{\xi > 0} \supp(\mu_2^{\xi})}
    = \overline{\bigcup_{\xi > 0} \Gamma_2(\xi)}  \]
and this is either $i \mathbb R$ in case $t \leq 0$, or $(-i \infty, -iy^*]
\cup [iy^*, i\infty)$ in case $t > 0$. This coincides with the support of $\sigma$.

Let $z = iy$ with $y > y^*$. The sets $\Gamma_2(\xi)$ are decreasing as $\xi$
increases, and there is a $\xi^*(z) > 0$ such that
$z \in \Gamma_2(\xi)$ if and only if $\xi \leq \xi^*(z)$.
Then by \eqref{eq: sigma def} we have
\begin{equation} \label{eq: densitysigma}
    \frac{\ud \widetilde \sigma(z)}{\ud z} = \int_0^{\xicr} \frac{\ud \mu_2^{\xi}(z)}{\ud z} \ud \xi,
    \end{equation}
since there is no contribution to the integral for $\xi > \xicr$.

The form of $\frac{\ud \mu_2^{\xi}(z)}{\ud z}$ depends on whether we are in
the one-cut or two-cut case. Let us assume that $t < \tau^2$ and $\xi^*(z) > \xicr$
so that both cases appear in \eqref{eq: densitysigma}. We will not give details
about the other cases, which are simpler.

Since $\xi^*(z) > \xicr$ we split up the integral \eqref{eq: densitysigma}
and we use \eqref{eq: muj 1cc} and \eqref{eq: muj 2cc bis} to obtain
\begin{multline*}
    \frac{\ud \widetilde \sigma(z)}{\ud z}=
        \frac{1}{4\pi i} \int_0^{\xicr} \left( \frac{{w_2'}_-(z;\xi)}{{w_2}_-(z;\xi)} -
        \frac{{w_2'}_+(z;\xi)}{{w_2}_+(z;\xi)}\right) \ud \xi \\
        + \frac{1}{2 \pi i} \int_{\xicr}^{\xi^*(z)} \left( \frac{{w_2'}_-(z;\xi)}{{w_2}_-(z;\xi)}
        - \frac{{w_2'}_+(z;\xi)}{{w_2}_+(z;\xi)}\right) \ud \xi.
\end{multline*}
Applying Lemma \ref{lemma: integration lemma} yields
\begin{multline}
\frac{\ud \widetilde \sigma(z)}{\ud z}
    =\frac{1}{2\pi i} \lim_{\xi_1 \to 0+} \left(G_2(z+,\xi_1)-G_2(z-,\xi_1,)\right) \\
    + \lim_{\xi_2 \to \xicr-}\lim_{\xi_3 \to \xicr+} \left( G_2(z-,\xi_2)-F_2(z,-\xi_3)
        +F_2(z+,\xi_3)-G_2(z+,\xi_2)\right) \\ + F_2(z-,\xi^*(z))-F_2(z+,\xi^*(z)).
\label{eq: sigma 1}
\end{multline}

By definition of $\xi^*(z)$, $z$ is on the boundary of $\Gamma_2(\xi^*(z))$.
Therefore, $z$ is the branch point $i\gamma(\xi^*(z))$. The algebraic equation
$s_1(w;\xi)=z$ then has the double solution $w_2(z;\xi^*(z))=w_3(z;\xi^*(z))$,
so that the last two terms of \eqref{eq: sigma 1} cancel each other.
Also the limit at $\xicr$ vanishes as a result of Lemma \ref{lemma: continuity lemma}.
The limit as $\xi_1 \to 0+$ is calculated using \eqref{eq: G23} which gives
\[
    \frac{\ud \widetilde \sigma(z)}{\ud z} = \frac{\tau}{\pi i}\Re \omega_1(z).
\]
By symmetry, the same formula is valid for $z = -iy$ with $y > y^*$.
This proves the equality \eqref{eq: sigmaequality} in view of the
definition  \eqref{eq: sigma} of $\sigma$.
\end{proof}

\end{document}